\documentclass[12pt,a4paper]{amsart}

\usepackage{amscd}
\usepackage{amsthm}
\usepackage{amsmath}
\usepackage{amssymb}
\usepackage{amsthm}
\usepackage{amsfonts}
\usepackage[utf8]{inputenc}
\usepackage[dvipsnames]{xcolor}
\usepackage{t1enc}
\usepackage[mathscr]{eucal}
\usepackage{indentfirst}
\usepackage{graphicx}
\usepackage{graphics}
\usepackage{pict2e}
\usepackage{epic}
\usepackage[normalem]{ulem}

\numberwithin{equation}{section}
\usepackage[margin=2.6cm]{geometry}
\usepackage{epstopdf} 

\allowdisplaybreaks
\usepackage{bbm}
\usepackage{tikz}
\usepackage{paralist}
\usepackage[onehalfspacing]{setspace}

\usepackage{color}
\usepackage{nicefrac}
\usepackage{mathrsfs}
\usepackage{multirow}
\usepackage{mathtools}
\usepackage{tikz-cd}
\usepackage{bbm}

\newcounter{dummy}
\usepackage{hyperref}
\usepackage{enumitem}

\makeatletter
\newcommand\myitem[1][]{\item[#1]\refstepcounter{dummy}\def\@currentlabel{#1}}
\makeatother

\newtheorem{theorem}{Theorem}
\numberwithin{theorem}{section}
\newtheorem{lemma}[theorem]{Lemma}

\newtheorem{definition}[theorem]{Definition}
\newtheorem{corollary}[theorem]{Corollary}

\newtheorem*{thm*}{Theorem}
\newtheorem*{prop*}{Proposition}
\numberwithin{equation}{section}
\theoremstyle{remark}
\newtheorem{remark}[theorem]{Remark}
\newtheorem{example}[theorem]{Example}

\newcommand{\A}{\mathscr{A}}
\newcommand{\Abb}{\mathbb{A}}

\newcommand{\R}{\mathbb{R}}
\newcommand{\N}{\mathbb{N}}
\newcommand{\Z}{\mathbb{Z}}

\newcommand{\C}{\mathbb{C}}

\newcommand{\D}{\mathscr{D}}
\newcommand{\LL}{\mathcal{L}}

\newcommand{\QA}{{Q}_{\A}}
\newcommand{\QTA}{{Q}_{\widetilde{\A}}}

\newcommand{\dx}[1][x]{\,\textup{d}{#1}}

\newcommand{\ds}{\,\textup{d}s}

\everymath{\displaystyle}
\newcommand{\weakto}{\rightharpoonup}

\DeclareMathOperator{\Lin}{Lin}

\DeclareMathOperator{\Image}{Im}

\DeclareMathOperator{\sym}{sym}

\DeclareMathOperator{\divergence}{div}

\newcommand{\dd}{\,\mathrm{d}}

\definecolor{Gump}{rgb}{0,0.6,0.4}
\definecolor{Hanks}{rgb}{0.7,0.3,0.1}

\newcommand{\Qd}{\mathcal{Q}}
\newcommand{\Pd}{\mathcal{P}}
\newcommand{\Sd}{\mathcal{S}}
\newcommand{\Rd}{\mathcal{R}}

\newcommand{\zeromean}{\#}
\newcommand{\zeromeansp}{{\text{\tiny{{/\kern-.2em /}}}}}

\newcommand{\VSob}{\overset{\smash{\circ}}{\mathrm{W}}}
\newcommand{\WSob}{\mathrm{W}}

\newcommand{\clec}[2][0]{\kern-#1em\overset{\textnormal{\tiny{#2}}}{\lesssim}}  
\newcommand{\commented}[3][0]{\kern-#1em\overset{\textnormal{\tiny{#2}}}{#3}}

\newcommand{\txtin}[1][.5]{\kern#1em\text{in}\ }  
\newcommand{\txton}[1][.5]{\kern#1em\text{on}\ }  
\newcommand{\txtas}[1][.5]{\kern#1em\text{as}\ }  
\newcommand{\txtfor}[1][.5]{\kern#1em\text{for}\ }  
\newcommand{\txtae}[1][.5]{\kern#1em\text{a.e.}\ } 
\newcommand{\txtforae}[1][.5]{\kern#1em\text{for a.e.}\ } 
\newcommand{\txtforall}[1][.5]{\kern#1em\text{for all}\ }
\newcommand{\txtat}[1][.5]{\kern#1em\text{at}\ }  
\newcommand{\txtand}[1][.5]{\kern#1em\text{and}\ }  
\newcommand{\txt}[2][.5]{\kern#1em\textnormal{#2}\ }

\newcommand{\eps}{\varepsilon}  

\renewcommand{\phi}{\varphi}

\newcommand{\cptmap}{\hookrightarrow\hookrightarrow}

\makeatletter
\newcommand{\wconv}[1][]{%
	\relax\if@display
	\overset{#1}{\xrightharpoonup{\quad}}
	\else
	\xrightharpoonup{#1}
	\fi
}
\newcommand{\conv}[1][]{%
	\relax\if@display
	\overset{#1}{\xrightarrow{\quad}}
	\else
	\xrightarrow{#1}
	\fi
}
\newcommand{\longto}{
	\relax\if@display
	\xrightarrow{\quad}
	\else
	\xrightarrow{}
	\fi
}

\makeatother

\newcommand{\dist}{\operatorname{dist}}

\newcommand{\norm}[2]{\left\Vert{#1} \right\Vert_{#2}}  
\newcommand{\morm}[2]{\left\vert \kern-.1em\left\vert\kern-.1em \left\vert{#1}\right\vert\kern-.1em\right\vert\kern-.1em\right\vert_{#2}}

\newcommand{\pair}[2]{\left\langle #1 \right\rangle_{#2}}

\newcommand{\abs}[1]{\left\vert#1\right\vert}  
\newcommand{\signabs}[2]{\left\vert {#1} \right\vert^{{#2}}\!#1 }

\newcommand{\dell}{\partial}

\renewcommand{\div}{\operatorname{div}}  
\newcommand{\diffD}{\operatorname D\!} 
\newcommand{\curl}{\operatorname{curl}}

\newcommand{\IR}{\mathbb R}
\newcommand{\IC}{\mathbb C}
\newcommand{\IZ}{\mathbb Z}
\newcommand{\IN}{\mathbb N}

\newcommand{\IT}{\mathbb T}
\newcommand{\IA}{\mathbb A}

\newcommand{\leb}[1]{L_{{#1}}}

\newcommand{\hoeldC}[1]{\mathrm{C}^{{#1}}}

\newcommand{\sobW}[2]{\mathrm{W}^{{#1}}_{{#2}}}

\newcommand{\fourF}{\mathcal F}

\newlist{proofsteps}{enumerate}{3}
\setlist[proofsteps]{
	label*=\textbf{Step~\arabic*.},
	ref={Step~\arabic*.},
	leftmargin=0em,
	itemindent=*
}
\setlist[proofsteps,2]{
	label*=\textbf{\arabic*.},
	ref={\arabic*},
	leftmargin=0em,
	align=left
}
\setlist[proofsteps,3]{
	label*=\textbf{\arabic*},
	ref={\arabic*},
	leftmargin=0em,
	align=left
}

\begin{document}
\title[Constitutive laws in parabolic problems]{A variational view on constitutive laws in parabolic problems}
\author[Schiffer,~Xylander]{Stefan Schiffer, Espen Xylander}
\address{Max-Planck Institute for Mathematics in the Sciences} 
\email{stefan.schiffer@mis.mpg.de}
\address{Institute of Analysis, Dynamics and Modeling, University of Stuttgart}
\email{espen.xylander@iadm.uni-stuttgart.de}

\begin{abstract}
We consider a variational approach to solve parabolic problems by minimising a functional over time and space. 
To achieve existence results we investigate the notion of $\mathscr{A}$-quasiconvexity for non-homogeneous operators in anisotropic spaces.
The abstract theory is then applied to formulate a variational solution concept for the non-Newtonian Navier--Stokes equations.
\end{abstract}

\maketitle
\section{Introduction}
In this work we introduce a generalised solution concept to the non-Newtonian Navier-Stokes system, that is the set of equations \begin{equation} \label{intro:NS}
	\begin{cases}
		\partial_t u + \divergence (u \otimes u) = \divergence \sigma - \nabla \pi& \text{in } (0,T) \times \Omega, \\
		\divergence u =0 & \text{in } (0,T) \times \Omega, \\
		u(0,\cdot) = u_0 & \text{in } \Omega,
	\end{cases}
\end{equation}
where \begin{itemize}
    \item $u \colon (0,T) \times \Omega \to \IR^d$ is the \emph{velocity} of a fluid;
    \item $\sigma \colon (0,T) \times \Omega \to \IR^{d \times d}_{\sym,0}$ denotes (deviatoric) \emph{stress}; here $\IR^{d \times d}_{\sym,0}$ is the space of symmetric and trace-free matrices;
    \item $\pi \colon (0,T) \times \Omega \to \R$ is the \emph{pressure}.
\end{itemize} 

For the remainder of this introduction, we focus on periodic functions on a cube, i.e. $\Omega=\IT_d$ to avoid a detailed discussion of boundary values. 
Later, however, we also include Dirichlet boundary conditions on Lipschitz domains.

System \eqref{intro:NS} may be derived from general considerations such as mass and momentum conservation. 
In an idealised incompressible world it models a variety of fluids and in the form presented above it is \emph{independent} of the material. 
The difference between distinct materials, e.g. water and oil, lies in the \emph{constitutive law} that relates the stress $\sigma$ to other quantities. 
Here, we assume that $\sigma$ depends only on the symmetric gradient, the \emph{rate-of-strain}, by some law of the form
\begin{equation} \label{intro:constlaw}
	\sigma = \sigma(\epsilon), \quad \epsilon = \epsilon(u) = \tfrac{1}{2} \left(\nabla u +(\nabla u)^T \right).
\end{equation}
This assumption is both well explored in a mathematical framework and also, within reasonable confines, motivated by real-world physical flows.

The mathematically most studied case is the \emph{Newtonian} relation 
\begin{equation*}
	\sigma = 2 \mu_0 \epsilon,
\end{equation*}
where $\mu_0$ is the \emph{viscosity} of the fluid. 
This linear relation yields the so-called Newtonian Navier--Stokes equations for $\mu_0 \neq 0$ and, in the degenerate case $\mu_0=0$, the incompressible Euler equations.
While such a law is reasonable for certain fluids such as water, other fluids do not exhibit such behaviour: 
The viscosity $\mu$ might depend on $\epsilon$ for instance which yields a relation 
\begin{equation} \label{intro:gennewton}
	\sigma= \mu( \vert \epsilon \vert) \cdot \epsilon.
\end{equation}
Such a constitutive law is sometimes also referred to as \emph{generalised Newtonian}. 
One of the simplest extensions is the (pure) \emph{power law} $\mu(\epsilon) = \mu_0\, \vert \epsilon \vert^{p-2}$, which gives rise the following non-Newtonian Navier-Stokes equations
 \begin{equation} \label{intro:NS:2}
	\begin{cases}
		\partial_t u + \divergence (u \otimes u) = \mu_0\, \divergence (\vert \epsilon \vert^{p-2} \epsilon) - \nabla \pi & \text{in } (0,T) \times \IT_d \\
		\divergence u =0& \text{in } (0,T) \times \IT_d, \\
		u(0,\cdot) = u_0 & \text{in } \IT_d.
	\end{cases}
\end{equation}
 
\subsection{Variational approaches to Navier--Stokes} 
In elliptic problems it is often rather straightforward to view a partial differential equation as a variational problem by viewing the equation as an Euler--Lagrange equation. 
For parabolic/time-dependent problems like the Navier--Stokes equations such a viewpoint is not as obvious.

One possible approach is to view the non-Newtonian Navier--Stokes equations as some type of gradient flow and equip them with a variational structure via this identification (see \cite{GM12}). 
This approach involves discretising time for very small time intervals and then solving  variational problems consecutively for each time step.

Another variational approach is to minimise a functional in space-time with an additional weight in time, the so-called WIDE (=Weighted Inertia Dissipation Energy) functional (cf. \cite{MO} for that terminology). 
This method also relates to using a gradient flow structure, but with minimising only \emph{one} functional in space-time instead of many (as in the time discretisation setting). 
In the context of the (non-Newtonian) Navier--Stokes equations this was first studied in \cite{OSS} and then later refined in \cite{BS,LSS}.

We choose a third approach that on the one hand has been well-explored 
in the context of solid mechanics (cf. \cite{CMO,CMO2,lienstromberg.etal_2022}). On the other hand, it is also connected to approaches that view the Navier--Stokes system and similar equations as a differential inclusion and apply convex integration techniques, cf. \cite{BV,BMS,CT}.

To explain our idea, we first draw a comparison to a classical approach: 
A core problem while modelling real-world problems is to determine the constitutive relation
\begin{equation} \label{intro:CR}
    \sigma = \sigma(\epsilon).
    \end{equation}
Such a relation is not known a priori and this needs to be determined by data from measurements. 
After possible preprocessing from experimental data, we think of this as strain-stress pairs $(\epsilon,\sigma)$. 
Consistent with assumption \eqref{intro:constlaw} we assume that no other quantity such as temperature play a role. 
The classical approach is now to fit parameters of a given constitutive law against the observed data (e.g. $\mu_0$ and $p$ for the power-law \eqref{intro:NS:2}), gain a constitutive relation and then solve the well-posed PDE to make predictions.

In contrast, we propose the following: We skip the step in which we extrapolate a constitutive law from data and consider the much larger space of solutions to \eqref{intro:NS} with $\sigma$ as a free variable. 
We then try to find a solution of the undetermined system \eqref{intro:NS} that fits given data best by minimising a functional.
Thus, informally our approach to modelling viscous fluids reads:
	\begin{center}
	\textit{Among all physically possible flows,\\ select one that matches the material properties best.}
\end{center}

In a certain sense this is close to approaches that see fluid equations as differential inclusions. 
There, for instance, one constructs solutions to the system
\begin{equation} \label{intro:NS3}
	\begin{cases}
		\partial_t u + \divergence (u \otimes u) = \mu_0 \,\Delta u - \nabla \pi + \divergence R& \text{in } (0,T) \times \IT_d, \\
		\divergence u =0 & \text{in } (0,T) \times \IT_d,\\
		u(0,\cdot) = u_0 & \text{in } \IT_d,
	\end{cases}
\end{equation}
for an additional stress $R$ and then lets $R \to 0$ in a suitable function space (such as $L^1((0,T) \times \IT_d)$). 
This stress $R$ then is an indicator on \emph{how far away} a given $u$ is to the solution of Navier--Stokes in the sense that the differential condition is \emph{almost satisfied}. 
Our idea is similar, in the sense that we try to determine $\sigma$, so that we solve system \eqref{intro:NS} exactly and then determine \emph{how close} we are to the given data.


\subsection{General setting}
We give an overview over the setup first.

Instead of prescribing the exact relation $\epsilon \mapsto \sigma(\epsilon)$ we consider a function
\begin{equation*}
    (\epsilon,\sigma)\longmapsto f(\epsilon,\sigma)\in [0,\infty).
\end{equation*}
that measures how much a given strain-stress pair deviates from the underlying data set.
In particular, an $f$ corresponding to a relation $\epsilon \mapsto \sigma(\epsilon)$ ought to satisfy $f(\epsilon,\sigma(\epsilon)) =0$ for all $\epsilon$. 
Given a data set $\D= \{(\epsilon_i,\sigma_i): i \in I\}$ one may instead think of 
\[
	f= \dist((\epsilon,\sigma),\D),
\]
where $\dist$ is a suitable distance function.

With a measure for the point-wise deviation, we can integrate over space and time to obtain the functional
\begin{equation} \label{def:I}
    I(u,\sigma) = \begin{cases}
        \int_0^T \int_{\IT_d} f(\epsilon(u),\sigma) \dd x \dd t & (u,\sigma) \text{ obey \eqref{intro:NS},} \\
        \infty & \text{else}.
    \end{cases}
\end{equation}
The primary objective of the present work is to understand the minimisation problem for $I$.
Minimisers of $I$ can be considered to be generalised solutions to \eqref{intro:NS} in the case where there is no prior-known constitutive law.


\subsection{Overview over the results}
While we do not give the precise statements in this introduction, we briefly give the reader an idea of the goals of this article in terms of developing a variational solution concept based on the minimisation of the functional $I$.

Motivated by the power-law relation \eqref{intro:NS:2}, we fix the following assertions on $f$:
\begin{enumerate} [label=(A\arabic*)]
    \item \label{it:global:assf:growth} 
        $f$ satisfies a growth condition from above: 
        \[
        f(\epsilon,\sigma) \leq C_1\,(1 + \vert \epsilon \vert^p + \vert \sigma \vert^q),
        \]
        for dual exponents $p,q$, i.e. $\tfrac{1}{p}+\tfrac{1}{q}=1$;
    \item  \label{it:global:assf:coercive} 
        $f$ satisfies a coercivity condition from below:
        \[
        f(\epsilon,\sigma) \geq C_2\,(\vert \epsilon \vert^p+\vert \sigma \vert^q) - C_3 - C_4\, \epsilon \cdot \sigma;
        \]
    \item  \label{it:global:assf:qcvx} 
        $f$ satisfies a convexity condition ($\A$-quasiconvexity for suitable $\A$); we comment on this further in Section~\ref{sec:technical}.
\end{enumerate}
Most common constitutive laws, $\sigma=\sigma(\epsilon)$ as in \eqref{intro:gennewton}, may be recast using such a function $f$ satisfying all those assumptions (cf. also Section~\ref{sec:6_applNS_largeP:1_integrandsConstitLaw}).
As an example, the relation $\sigma=\vert \epsilon \vert^{p-2} \epsilon$ can be rephrased as the set of zeros of
\[
f(\epsilon,\sigma) = \tfrac{1}{p}\vert \epsilon \vert^p + \tfrac{1}{q}\vert \sigma \vert^q - \epsilon \cdot \sigma.
\]

We call a minimiser of $I$ defined as in \eqref{def:I} a \emph{variational solution} to the non-Newtonian Navier--Stokes equations.
To prove existence of such, we seek to apply the direct method:
\begin{enumerate}[label=(\roman*)]
    \item Prove that $I$ is coercive, i.e. minimising sequences converge in a suitable topology;
    \item Prove lower semicontinutiy, so that the limit of a minimising sequence is a minimiser.
\end{enumerate}
The main results concerning this problem can now be summarised as follows:

\begin{enumerate} [label=(R\arabic*)]
    \item \label{it:mainResults:1} If $f$ is derived from a constitutive law, then the notion of variational solution is consistent with the notion of a weak solution to the corresponding non-Newtonian Navier--Stokes equation.
    \item \label{it:mainResults:2} In the regime $p \ge \tfrac{3d+2}{d+2}$, we have existence of variational solutions under the aforementioned assumptions (cf. Theorem~\ref{thm:existenceMinimisersLargeP}) and we essentially recover a variational version of classical existence results of \cite{Lady1,Lady2,Lady3,Lionsbook}.
    \item \label{it:mainResults:3} In the regime $p > \tfrac{2d}{d+2}$, the given functional remains sequentially lower semicontinuous (cf. Theorem~\ref{thm:applNS:lsc}). An energy inequality is, however, \emph{not granted} and we only get a minimiser among a (physically reasonable) class of functions that obey an additional a-priori energy inequality (cf. Theorem~\ref{thm:existenceSmallP}). 
    \item \label{it:mainResults:4} We can give an alternative proof for the existence of Leray--Hopf solutions comparable to \cite{FMS,BDF,BDS}  for certain constitutive laws $\sigma=\sigma(\epsilon)$, cf. Theorem~\ref{thm:existenceLerayHopf}. While indirectly using the basic idea of Lipschitz truncation, our proof is based on the variational techniques we develop before.
\end{enumerate}
In terms of \emph{how} those results are achieved as a consequence of a much broader theory, we refer to Section~\ref{sec:technical} below.


\subsection{Some remarks on the solution concept}
We point out some features and issues of the idea that a solution to \eqref{intro:NS:2} is a minimiser of a certain functional. 
First, it is a relaxation of the viewpoint of a differential inclusion (cf. \cite{BV,BMS}): 
A solution to the power-law non-Newtonian Navier--Stokes equations \eqref{intro:NS:2} is precisely characterised by $I(u,\sigma)=0$ with
\[
f(\epsilon,\sigma) = \tfrac{1}{p} \vert \epsilon \vert^p + \tfrac{1}{q} \vert \sigma \vert^q - \epsilon \cdot \sigma.
\]

This also means that we inherit difficulties surrounding small power-law exponents $p < \tfrac{3d+2}{d+2}$, so that we need to approach the problem with more care (cf. Section~\ref{sec:7_applNS_smallP}).

While this setup of the functional in space-time is quite convenient, it in general destroys one of the features of time-dependent equations that we aim to have: \emph{causality}. 
In particular, if $T<T'$ are two times, the restriction of the solution on the larger time interval $(0,T')$ to the smaller $(0,T)$ might not be minimal: 
As time evolves the solution on the smaller interval $(0,T)$ might leave an energy optimal state and another solution might become better in the long run.
In principle, this fundamental issue of the solution concept may be attacked by introducing a weight-in-time (as in \cite{BS,LSS,OSS}) or by discretising-in-time and solving step-by-step (cf. \cite{GM12}). 
This, however, is beyond the scope of this article and we refer to future work.


\subsection{$\A$-quasiconvexity} \label{sec:technical}
We embed our system into an abstract framework. 
In particular, we rewrite system \eqref{intro:NS} (also cf. \cite{DLS2,DLS1,DLS3,lienstromberg.etal_2022}) as 
\begin{equation} \label{eq:abstractEQ}
    \A(\epsilon,\sigma) = \Theta (u).
\end{equation}

Here, $\Theta(u)$ shall correspond to the non-linearity  $\divergence( u\otimes u)$. 
Then $\A$ is a linear non-homogeneous differential operator that corresponds to a second-order parabolic problem; i.e. we have \emph{one} derivative in time and \emph{two} derivatives in space.
The precise definition of $\A$ is given in Example \ref{ex:formulations:NSExample}. 

We would then like to apply the theory of $\A$-quasiconvexity and weak lower semicontinuity developed by Fonseca \& M\"uller in \cite{FM} (cf. \cite{DF,FMP,GR2}) to ensure existence of minimisers of the functional $I$. 
Recall that for a differential operator $\A$, a function $f$ is called $\A$-quasiconvex if for any test function, 
\[
\psi \in C^{\infty}(\IT_{d+1};\IR^{d\times d}_{\mathrm{sym},0}\times\IR^{d\times d}_{\mathrm{sym},0}) \quad \text{s.t.}\quad  \int_{\IT_{d+1}} \psi \dx[(t,x)]  =0 \text{ and } \A \psi =0,
\]
it satisfies the Jensen-type inequality
\[
f(\epsilon,\sigma) \leq \int_{\IT_{d+1}} f((\epsilon,\sigma)+\psi(t,x)) \dx[(t,x)].
\]
For \emph{homogeneous} operators $\A$ it is shown in \cite{FM} that $\A$-quasiconvexity corresponds to weak lower semicontinuity of the integrand, which in turn guarantees the existence of minimisers as detailed above.

This theory is not applicable for operators that are genuinely non-homogeneous, i.e. do not satisfy $\A u=0 \Rightarrow \A u_{\lambda} =0$ for $u_{\lambda}(x) = u(\lambda x)$. 
We therefore need to develop the full theory for \emph{anisotropic} spaces. This involves dealing with some new non-linear effects when compared to the homogeneous setting.

Furthermore, due to the incompressibility constraint $\divergence u=0$, the operator $\A$ for the non-Newtonian Navier--Stokes system is actually a \emph{pseudo-differential operator} and we ensure that the methods used for differential operators carry over to the pseudo-differential case.


\subsection{Structure}
The paper is structured into two main parts: a collection of lower semicontinuity results for abstract parabolic differential equations and an application to the non-Newtonian and data-driven Navier--Stokes equations.
First, in Section~\ref{sec:2_Spaces} we introduce the necessary anisotropic function spaces.
We then prove general properties of the linear operator $\A$ in Section~\ref{sec:3_propA}, before we give a detailed treatment of $\A$-quasiconvexity.
Section~\ref{sec:4_wlsc} contains the proof of our main lower semicontinuity result and a brief comment on relaxation.
We refine these results in Section~\ref{sec:5_pseudoDiffOps} to include the case of pseudo-differential constrains.
In the application part of the paper, starting with in Section~\ref{sec:6_applNS_largeP}, we elaborate on constitutive laws and `data-driven' formulations and their relation to \ref{it:mainResults:1}, before we prove \ref{it:mainResults:2}.
Afterwards we discuss issues in the small parameter regime and how to resolve them via \ref{it:mainResults:3} or \ref{it:mainResults:4} throughout Section~\ref{sec:7_applNS_smallP}.


\subsection*{Acknowledgements}
St. S. would like to thank the University of Stuttgart for the kind hospitality during several research states. The research was funded by Deutsche Forschungsgemeinschaft (DFG, German Research Foundation) under Germany's Excellence Strategy -- EXC 2075 - 390740016. We further acknowledge the support by the Stuttgart Center for Simulation Science (SimTech).


\subsection*{Notation}
In the following we will consider domains $\mathcal T\times\Omega$, where
\begin{itemize}
    \item  $\mathcal T = (0,T)$ for some $T>0$ or $\mathcal T=\IT_1$;
    \item $\Omega \subset \IR^d$ is a Lipschitz domain or $\Omega = \IT_d$.
\end{itemize}
We usually write elements of this space as $(t,x) \in \mathcal T\times\Omega$ and refer to it as 'space-time'. 
We use multi-index notation for partial derivatives, $\dell^\nu \coloneqq
\dell_t^{\nu_t}\circ \dell_x^{\nu_x}\coloneqq \dell_t^{\nu_t} \circ \dell_{x_1}^{\nu_{x_1}}\circ\ldots\circ \dell_{x_d}^{\nu_{x_d}}$, where $\nu=(\nu_t,\nu_x)=(\nu_t,\nu_{x_1},\ldots,\nu_{x_d})\in \IN\times \IN^d= \IN^{d+1}$. We write $\nabla u$ to refer to the (distributional) full gradient of a function and $\nabla^k$ for the $k$-th order gradient. 
The symbol $\tfrac 12(\nabla+\nabla^T)$ then denotes the symmetric gradient.

When working with a space $\mathcal X(\mathcal T\times\Omega;\IR^N)$ consisting of functions $\mathcal T\times\Omega\to \IR^N$, we often omit the domain or co-domain if they are clear from the context. 
In these cases we write $\mathcal X(\mathcal T\times\Omega)$ or $\mathcal X$ instead.
Furthermore, we sometimes write $ (\mathcal X_1\times \mathcal X_2)(\mathcal T\times\Omega;\IR^M\times\IR^N) $ for $ \mathcal X_1(\mathcal T\times\Omega;\IR^M)\times \mathcal X_2(\mathcal T\times\Omega;\IR^N) $. 
The same applies to other binary operations on function spaces.

Speaking of function spaces, we denote by $L_p$ the space of functions whose $p$-th power is integrable and by $\WSob^{k}_p$ those functions whose $k$-order distributional derivatives are in $L_p$. We will frequently also use the notation
\[
\WSob^{k_1,k_2}_p \quad \text{and} \quad \WSob^{k_1,k_2}_{p,q}
\]
which is used for anisotropic function spaces, cf. Sections \ref{sec:2_Spaces:3_functionspaces} and \ref{sec:2_Spaces:pblmSpaces} for the exact definition.

For functions $f \in L_1(\IT_{d+1};\C)$ (also $f \in L_1(\IT_{d},\C)$), the discrete Fourier transform is denoted by
\begin{equation}\label{eq:discreteFourier}
	\fourF f(\xi)\coloneqq \int_{\IT_{d+1}} f(z)\, e^{-2\pi i\, \xi\cdot z} \dx[z],\quad \xi\in \IZ^{d+1}.
\end{equation}
For non-integrable functions the Fourier transform may be defined via density.
The inverse Fourier transform of $m \in \ell_{1}(\Z^{d+1};\C)$ is given by
\begin{equation*}
	\fourF^{-1}m(z)\coloneqq \sum_{\xi\in \IZ^{d+1}} m(\xi)\, e^{2\pi i\, \xi\cdot z},\quad z\in \IT_{d+1}.
\end{equation*}

We denote by $C$ a general positive constant that may change from line to line.
Occasionally we highlight the dependence on a parameter e.g. with $C_\eta$ or enumerate different constants in a display with $C_1,\, C_2,\ldots$ etc.
We write $A \sim B$ if there is a constant $C>1$ such that $ C^{-1}\, A \le B \le C\, A$. 

The symbols $\epsilon,\sigma$ always denote either $\IR^m$ or $\IC^m$-valued functions. In contrast, to denote single elements of $\IR^m$ or $\C^m$ we write $\hat\epsilon,\hat\sigma$.
\section{Function spaces for non-homogeneous problems}\label{sec:2_Spaces}

As alluded to in the introduction, we focus on a setting that is applicable to parabolic equations.
Our function spaces are therefore anisotropic and feature different integrability for the arguments $ \epsilon $ and $\sigma$ of the differential operator $ \A(\epsilon,\sigma) $ as in \eqref{eq:abstractEQ}.
To set the scene for a precise definition of $ \A $, we briefly review some notions for homogeneous constant-rank operators.


\subsection{Homogeneous constant-rank operators}\label{sec:2_Spaces:1_consrRankOps}
Consider the differential operator $\Qd$ defined for functions $\sigma \colon \IR^d \to \IR^m$ by
\begin{equation} \label{def:Q}
    \Qd \sigma = \sum_{\vert \alpha \vert = k} Q_{\alpha}\, \partial^{\alpha} \sigma,
\end{equation}
where $Q_{\alpha} \colon \IR^m \to \IR^n$ are linear maps. 
By formally writing $\C^m= \IR^m + i \IR^m$, we may also define this operator on functions $\sigma \colon \IR^d \to \C^m$. 
Since there is usually no ambiguity, we use $\epsilon$ and $\sigma$ both for real-valued and complex-valued functions.

We may define the \emph{Fourier symbol} of the operator $\Qd$ as follows:
\begin{equation} \label{def:FS}
    \Qd [\xi] = (2 \pi i)^k \sum_{\vert \alpha \vert=k} Q_{\alpha}\,\xi^{\alpha} \in \Lin(\C^m;\C^n), \quad \text{for any } \xi \in \IR^{d}.
\end{equation}
 We assume that $\Qd$ satisfies the \emph{constant-rank property} (cf. \cite{Murat,Raita}) in space, i.e. there is $0 \leq r \leq \min\{m,n\}$ such that 
\begin{equation} \label{constantrank:Q}
    \dim \ker \Qd[\xi] =r \quad \text{for all } \xi \in \IR^{d} \setminus \{0\}.
\end{equation}

Furthermore, we define the adjoint $\Qd^{\ast}$ of $\Qd$ on $k$-times differentiable functions $w \colon \IR^d \to \IC^m$ via
\begin{equation}
    \Qd^{\ast}w = \sum_{\vert \alpha \vert =k} Q_{\alpha}^{\ast} \,\partial^{\alpha} w.
\end{equation}
Therefore, $\Qd^{\ast}$ also obeys the constant-rank property and we can define its Fourier symbol as well as the Fourier symbol of $\Qd^{\ast} \circ \Qd$ in the same way as in \eqref{def:FS}.

A constant-rank differential operator $\Pd \colon C^{\infty}(\IR^d;\IC^n) \to C^{\infty}(\IR^d;\IC^m),$
\begin{equation} \label{def:P}
    \Pd u = \sum_{\vert \alpha \vert =k'} P_{\alpha}\, \partial^{\alpha}u,
\end{equation}
is called a potential to $\Qd$ if 
\begin{equation*}
    \Image \Pd [\xi] = \ker \Qd [\xi], \quad \forall \xi \in \IR^{d} \setminus \{0\}
\end{equation*}
(cf. \cite{Raita,DF}).
The existence of such a potential is equivalent to the constant-rank property \cite{Raita}.
Dual to this construction we also call $\Qd$ the \emph{annihilator} for $\Pd$ and, upon taking adjoints, we find that $\Qd^\ast$ is a potential of $\Pd^\ast$.

The operator $\Pd^\ast$ is said to be \emph{spanning} if 
\begin{equation*}
    \operatorname{span} \bigcup_{\xi\in \IR^{d}\setminus\{0\}} \ker \Pd^{\ast}[\xi] = \IC^m.
\end{equation*}

\begin{remark}\label{rem:assumePspanning}
	In the following, we denote a generic \emph{homogeneous} constant-rank operator by $\mathbb{A}$. 
	As it shall act only on space on not on space-time, we assume $\mathbb{A}\colon C^{\infty}(\IR^d;\IR^m) \to C^{\infty}(\IR^d;\IR^n)$.
	The operators $\Pd$ and $\Qd$ are examples of such purely spatial operators. 
	Throughout the paper all operators that act only in space have the constant-rank property and we assume that $\Pd^\ast$ is spanning.
	We reserve the symbol $\A$ for the \emph{parabolic} operator.
\end{remark}

To finish this subsection, we remind the reader of the following fact: 
There is a smooth $(-k)$-homogeneous function $$\IA^{-1} \colon \IR^d \longto \Lin(\IR^n;\IR^m)$$ such that for any $\xi \neq 0$, $\IA^{-1}[\xi]$ reduces to a bijection between $\Image \IA[\xi]$ and $\left(\ker \IA[\xi]\right)^{\perp}$. 
This is called the \emph{Moore--Penrose inverse}, cf. \cite{Penrose_1955,Raita} and we may define the $(-k)$-homogeneous Fourier multiplier
\begin{equation}\label{eq:moorePenrose:multiplierDef}
    \IA^{-1} v \coloneqq \sum_{\xi\in\IZ^{d}} \,
    \, \IA^{-1}[\xi]\,\fourF v(\xi) \, e^{2\pi i\, x\cdot \xi},
    \quad 
     v\in\hoeldC\infty(\IT_d;\IR^M).
\end{equation}
Note that by convention $\IA^{-1}[\xi]=0$ if $\IA[\xi]=0$.


\subsection{Differential operators for parabolic problems}\label{sec:2_Spaces:2_prabolicDiffOps}
We now define  an abstract differential operator for parabolic problems and give a simple example.
We note that, due to the incompressibility constraint, the Navier--Stokes equations do not fit exactly into the setting that we define below.
They require a small adaptation by passing to \emph{pseudo-}differential operators and we refer to Section~\ref{sec:5_pseudoDiffOps} for a detailed treatment of that case.
\smallskip

As hinted to in the introduction, $\epsilon \in C^{\infty}(\IT_{d+1};\C^m)$ might not be the physical variable of interest, but it might arise as a derivative of a physical variable $u$ (e.g. it is the rate-of-strain associated to a velocity field).
Thus, we seek to transfer the physically motivated system
\begin{equation} \label{def:altA:1}
   \begin{cases}
       \partial_t u = \Qd \sigma
       &	\text{in } \IT_{d+1},
   \\  \epsilon = \Qd^\ast u
   &	\text{in } \IT_{d+1},
   \end{cases}
\end{equation}
into equations involving solely $\epsilon$ and $\sigma$:
\begin{equation}\label{def:altA:0}
   \begin{cases}
       \dell_t \epsilon = \Qd^\ast \Qd \sigma
       &	\text{in } \IT_{d+1},
   \\  \Pd^\ast \epsilon=0 &	\text{in } \IT_{d+1},
   \\	\int_{\IT_{d+1}}\epsilon\dx[(t,x)]=0 .
   \end{cases}
\end{equation}
On the space-time torus $\IT_{d+1}$, we can solve the equation $\epsilon =\Qd^\ast u$ using Fourier methods: 
Assuming that $\Qd$ satisfies the constant-rank property and $\epsilon$ fulfills the compatibility condition $\Pd^\ast \epsilon=0$, we can use the (Moore--Penrose) inverse $(\Qd^\ast)^{-1}$ via \eqref{eq:moorePenrose:multiplierDef} to define $u\coloneqq (\Qd^\ast)^{-1}\epsilon$. 
Then $\dell_t\epsilon=\Qd^\ast\Qd \sigma$ becomes
\begin{equation*}
    \dell_t u 
    =   (\Qd^\ast)^{-1}\Qd^\ast\Qd \sigma
    =   \Qd \sigma.
\end{equation*}
Thus, any solution $(\epsilon,\sigma)$ of \eqref{def:altA:0} gives rise to a solution $(u,\epsilon,\sigma)$ of \eqref{def:altA:1} and vice versa.

These considerations motivate the following definition of our main parabolic differential operator. 
\begin{definition} \label{def:A}
    Let $(\epsilon,\sigma) \in C^{\infty}(\IR^{d+1};\C^m \times \C^m)$. 
    We define the differential operator 
    \begin{equation*}
        \A \colon C^{\infty}(\IR^{d+1};\C^m \times \C^m) \to C^{\infty}(\IR^{d+1};\C^m\times \C^n)
    \end{equation*}
    as follows:
    \begin{equation} \label{eq:defA}
    \A(\epsilon,\sigma) 
    =   \begin{pmatrix}
            \partial_t \epsilon - \Qd^{\ast} \Qd \sigma
        \\  \Pd^\ast \epsilon
        \end{pmatrix}.
    \end{equation}
    We define the Fourier symbol for $\xi=(\xi_t,\xi_x) \in \IR^{d+1}$ as 
    \begin{equation}
    \A [\xi](\hat{\epsilon},\hat{\sigma}) 
    =   \begin{pmatrix}
            (2 \pi i)\, \xi_t  \cdot \hat{\epsilon}
            -   \Qd^{\ast} [\xi_x] \circ \Qd [\xi_x] (\hat{\sigma})
        \\  \Pd^\ast[\xi_x]\hat \epsilon
        \end{pmatrix}   
    ,\quad 
        \hat{\epsilon}, \hat{\sigma} \in \C^m.
    \end{equation}
\end{definition}

\noindent To illustrate the setting we present the following example:

\begin{example}[Anisotropic heat flux]
	We consider a simple situation that is connected to the heat equation and therefore might be applied to anisotropic heat flow. 
	With variables $\tau\colon \IR^d\to \IR$ and $J\colon \IR^d\to \IR^d$  representing \emph{temperature} and \emph{heat flux}, we can pose the conservation equations
	\begin{equation*}
		\begin{cases}
			 	\dell_t\tau =\div J
			&\text{in } \IR^{d+1},
			\\ \phi=\nabla \tau
			&\text{in }\IR^{d+1}.
		\end{cases}
	\end{equation*}
	A constitutive law relates the heat flux $J$ to the temperature gradient $\nabla \tau$.
	In our setting, we consider
	\begin{equation*}
		\Qd = \divergence,\quad \Pd^{\ast} = \curl
	\end{equation*}
	and thus obtain the equation
	\begin{equation*}
		\begin{cases}
			     \dell_t \phi -\nabla\div J=0
			&\text{in }\IR^{d+1},
			\\   \curl \phi =0
			&\text{in }\IR^{d+1}.
		\end{cases}
	\end{equation*}
	Observe that if we are given a linear relation $J(\phi) = \phi = \nabla \tau$, we may recover the heat equation for $\tau$.
\end{example}


\subsection{Isotropic and anisotropic function spaces} \label{sec:2_Spaces:3_functionspaces}
In this section, we remind the reader of some well-known function spaces. 
Some of these are suited to anisotropic problems, where we consider different orders of time and space derivatives, while keeping a fixed integrability exponent $p=q=r\in (1,\infty)$.
Later, in Section \ref{sec:2_Spaces:pblmSpaces}, we also discuss \emph{doubly anisotropic} spaces where, in addition, $p\neq q$.

As mentioned in the introduction we work with either $(0,T) \times \Omega$, $\Omega$ Lipschitz domain or $\IT_{d+1}$ as the underlying domain. 
We write $\mathcal T \times \Omega$ for both variants.

We work with function spaces based on $L_p(\mathcal T\times \Omega)$ instead of Bochner-type function spaces. 
To this end, for exponents $1<r<\infty$ and $\alpha,\beta \in \N$, define the anisotropic Sobolev spaces
\begin{equation*}
\begin{aligned}
    \WSob^{\alpha,\beta}_r(\mathcal T\times \Omega) \coloneqq    \Big\{& v \in L_r(\mathcal T\times \Omega) \colon \partial_t^{\alpha'}v \in L_r(\mathcal T \times \Omega), \\ & \nabla_x^{\beta'}v \in L_r(\mathcal T \times \Omega), 
     1 \leq \alpha' \leq \alpha, 1 \leq \beta' \leq \beta\Big\},
\end{aligned}
\end{equation*}
(see e.g. \cite{Amann09}) 
where the distributional derivatives in $t$ and $x$ directions are already $L_r$-functions (for simplicity, we still write $L_r= \WSob^{0,0}_r$). 
The norm of this space is then naturally defined via
    \[
    \Vert v \Vert_{\WSob^{\alpha,\beta}_r(\mathcal T\times \Omega)} = \Vert v \Vert_{L_r(\mathcal T\times \Omega)} + \sum_{\alpha'=1}^{\alpha} \Vert \partial_t^{\alpha'} v \Vert_{L_r(\mathcal T\times \Omega)} + \sum_{\beta'=1}^{\beta} \Vert \nabla_x^{\beta'} v \Vert_{L_r(\mathcal T\times \Omega)}
    \]
By duality with the space of functions with zero trace we may also define spaces with negative exponents, i.e. 
    \[
    \WSob^{-\alpha,-\beta}_r(\mathcal T\times \Omega) \coloneqq \left(\VSob^{\alpha,\beta}_{r'}(\mathcal T\times \Omega) \right)'
    \]
with $\tfrac{1}{r'} + \tfrac{1}{r}=1$, where $\VSob^{\alpha,\beta}_{r'}(\mathcal T\times \Omega)$ is the closure of $C_c^{\infty}(\mathcal T\times \Omega)$ in $\WSob^{\alpha,\beta}_{r'}(\mathcal T\times \Omega)$.

On $\IT_{d+1}$ we can use the discrete Fourier transform \eqref{eq:discreteFourier} to extend the definitions of $\WSob^{\alpha,\beta}_r$ to $\alpha,\beta\in\IR$.
To this end, define $\WSob^{\alpha,\beta}_r(\IT_{d+1})$ to be the closure of the smooth functions on $\IT_{d+1}$ with respect to 
\begin{equation}\label{eq:soboViaFourier}
    \begin{split}
    	\Vert v \Vert_{\WSob^{\alpha,\beta}_r(\IT_{d+1})}
     &\coloneqq \Vert \widetilde{v} \Vert_{L_r(\IT_{d+1})}
     , \\
     \widetilde{v}(t,x) &= \sum_{\xi \in \Z^{d+1}} (1+\vert \xi_t \vert^2)^{\frac\alpha2}\, (1+ \vert \xi_x \vert^2)^{\frac\beta2}\,\fourF{v}(\xi)\, e^{2\pi i\, (t,x) \cdot \xi}.
    \end{split}
\end{equation}

Finally, recall that
\begin{equation*}
	(\VSob^{1,0}_{r'}\cap\VSob^{0,2k}_{r'}) (\mathcal T\times\Omega;\IR^m) 
	\quad\text{and}\quad 
	(\sobW{-1,0}{r}+\sobW{0,-2k}{r}) (\mathcal T\times\Omega;\IR^m),
\end{equation*}
defined as subsets of $ \mathcal D'(\mathcal T\times\Omega;\IR^m) $, are Banach spaces with the norms 
\begin{equation*}
\begin{split}
	\norm{v}{\VSob^{1,0}_{r'}\cap\VSob^{0,2k}_{r'}}
	&\coloneqq 
		\max\left \{
			\norm{v}{\VSob^{1,0}_{r'}} 
			,\,  \norm{v}{\VSob^{0,2k}_{r'}}
		\right \}
	\quad \text{and}
\\	\norm{w}{\WSob^{-1,0}_{r}+\WSob^{0,-2k}_{r}}
	&\coloneqq 
		\inf_{w=w^\mathrm{t}+w^\mathrm{x}}
		\Vert w^\mathrm{t} \Vert_{\WSob^{-1,0}_{r}} 
		+ \Vert w^\mathrm{x} \Vert_{\WSob^{0,-2k}_{r}}.
\end{split}
\end{equation*}
Moreover, the two spaces are canonically dual to each other (cf. \cite[Theorem 2.7.1.]{bergh_interpolation_1976}), 
\begin{equation*}
	\sobW{-1,0}{r}+\sobW{0,-2k}{r} \simeq  \left (\VSob^{1,0}_{r'}\cap\VSob^{0,2k}_{r'}\right )'.
\end{equation*}
Thus, our differential operator $\A$ extends to a bounded linear operator
\begin{equation*}
    \A\colon (\leb r\times\leb r)(\mathcal T\times\Omega;\IR^m\times\IR^m)
    \longto ((\sobW{-1,0}{r}+\sobW{0,-2k}{r})\times\sobW{0,-k'}{r})(\mathcal T\times\Omega;\IR^m\times\IR^n).
\end{equation*}


\subsection{Multiplier, embeddings and compactness}
Below, we commonly use the following observations: 
For the first result on Fourier multipliers we refer to 
\cite[Theorems 4.3.7 and 6.2.4]{Grafakos}, the second is a consequence of the Rellich--Kondrachov compact Sobolev embedding.

\begin{lemma}[Fourier multipliers and compact embeddings] \label{lem:FM}
  Let $1<r<\infty$ and $\alpha,\beta \in \R$.
    \begin{enumerate}[label=(\roman*)]
        \item \label{FM:1} Let $\phi \in C(\IR^{d+1})$ be smooth except possibly on the coordinate planes $\xi_t=0$ or $\xi_x=0$.
        Assume we have for all $\nu =(\nu_t,\nu_{x})\in \IN\times\IN^d$
        \[
        \vert \dell^\nu \phi (\xi) \vert \leq C_\nu\, \vert \xi_t \vert^{-\nu_t}\,  \vert \xi_x \vert^{-\vert \nu_x \vert} ,\quad \xi\in(\IR\setminus\{0\})\times (\IR^d\setminus\{0\}).
        \]
        Then the map $M_{\phi}$ defined by
        \[
        M_{\phi} v = \sum_{\xi \in \Z^{d+1}} \phi \left(\xi\right) \fourF{v}(\xi)\, e^{2\pi i \,(t,x)\cdot \xi}
        \]
        for $v \in C^{\infty}(\IT_{d+1})$ extends to a linear and bounded map from $L_r(\IT_{d+1})$ to $L_r(\IT_{d+1})$.
        \item \label{FM:2} Let $\varepsilon>0$. Then the embedding
            \[
            \WSob^{\alpha,\beta}_r(\IT_{d+1}) \hookrightarrow \hookrightarrow \WSob^{\alpha-\varepsilon,\beta-\varepsilon}_r(\IT_{d+1})
            \]
        is compact.
    \end{enumerate}
\end{lemma}

\begin{remark}\label{rem:spatialOpsAsMultiplier}
	For a constant-rank operator $\IA$ the multiplier $\phi(\xi) = (1+\abs{\xi_x}^2)^{-k/2}\IA[\xi_x]$ is precisely of the form in \ref{FM:1}.
	Hence, $\IA\colon \sobW{\alpha,\beta}{r}(\IT_{d+1})\to \sobW{\alpha,\beta-k}{r}(\IT_{d+1})$ is bounded. 
	For an operator $\IA^{-1}$ as in \eqref{eq:moorePenrose:multiplierDef}, the  map $\xi\mapsto (1+\abs{\xi_x}^2)^{k/2} \IA^{-1}[\xi_x] $ coincides with a suitable smooth function $\widetilde \phi$ in a neighbourhood of the integer grid $\IZ^{d+1}$ (due to $\IA^{-1}[0]=0$). 
	Thus, we also get an extension of \eqref{eq:moorePenrose:multiplierDef} to a continuous map $\IA^{-1}\colon \sobW{\alpha,\beta}{r}(\IT_{d+1})\to \sobW{\alpha,\beta+k}{r}(\IT_{d+1})$. 
\end{remark}

Observe that the embedding in Lemma~\ref{lem:FM} \ref{FM:2} is \emph{not} compact unless we lower the differentiability in time and in space simultaneously;
but the regularity in \eqref{eq:defA} only diminishes the differentiability in time or in space.
Therefore, the following lemma is invaluable.

\begin{lemma}\label{lem:compactembedding}
	Let $1<r<\infty$, $k\in \N$. Then
	\begin{enumerate}[label=(\roman*)]
		\item \label{cptemb:1} $L_r(\IT_{d+1}) \hookrightarrow \hookrightarrow (\WSob^{-1,0}_r + \WSob^{0,-2k}_r)(\IT_{d+1})$;
		\item \label{cptemb:2} $\WSob^{0,-2k+1}_r(\IT_{d+1}) \hookrightarrow \hookrightarrow (\WSob^{-1,0}_r + \WSob^{0,-2k}_r)(\IT_{d+1})$.
	\end{enumerate}
\end{lemma}

\begin{remark} \label{rem:compactembedding} 
	\begin{enumerate}[label=(\roman*)]
		\item
		If we write $X \hookrightarrow \hookrightarrow Y_1 + Y_2$, this means the following for sequences: If $x_j \weakto x$ weakly in $X$, then there is $x_j=y_j^1+y_j^2$ such that $y_j^1 \to y^1$ strongly in $Y_1$ and $y_j^2 \to y^2$ strongly in $Y_2$.
		\item 
		At this point we are not working with exponents $p$ and $q$, as for certain choices of $p$ and $q$ we would not have the compact embedding
		\[
		L_p(\IT_{d+1})\hookrightarrow \hookrightarrow (\WSob^{-1,0}_p + \WSob^{0,-2k}_q)(\IT_{d+1});
		\]
		indeed if $p$ is close to $1$ and $q$ close to infinity we cannot hope for such a statement.
       
		\item The first statement is actually quite clear as $(\WSob^{-1,0}_r+ \WSob^{0,-1}_r)(\IT_{d+1}) = \WSob^{-1}_r(\IT_{d+1})$; for completeness we give a short argument below. In contrast, the differentiability actually improves for one space in the second statement which is not completely obvious.
	\end{enumerate}
\end{remark}

Lemma~\ref{lem:compactembedding} is a consequence of the following construction, which provides us with a way to decompose a function into a temporal and a spatial component:

\begin{lemma}[Space-time decomposition]\label{lem:spacetimeConverter}
	Let $1<r<\infty$, $\alpha,\beta\in \IR$ and $\gamma>0$. 
	Pick $\phi\in C^\infty((0,\infty))$ so that $\phi=0$ on $(0,1/3)$ and $\phi=1$ on $(2/3,\infty)$. 
	Then the two Fourier multipliers 
	\begin{align*}
		\Phi^{\mathrm{t}}_\gamma v
		\coloneqq& 
		\sum_{\xi \in \Z^{d+1}\setminus\{0\}} \phi \!\left(
			\frac{(1+\abs{\xi_t}^2)^{1/2}}{(1+\abs{\xi_x}^2)^{\gamma/2}}
		\right) 
		\fourF{v}(\xi)
		\, e^{2\pi i\, (t,x) \cdot \xi},
		\\  \Phi^{\mathrm{x}}_\gamma v
		\coloneqq& 
		\sum_{\xi \in \Z^{d+1}\setminus\{0\}} 
		(1-\phi)\!\left(\frac{(1+\abs{\xi_t}^2)^{1/2}}{(1+\abs{\xi_x}^2)^{\gamma/2}}\right) \fourF{v}(\xi)\, e^{2\pi i\, (t,x) \cdot \xi},
		\quad v\in C^\infty(\IT_{d+1}),
	\end{align*}
	extend to bounded linear operators
	\begin{align*}
		\Phi^{\mathrm{t}}_\gamma
		&\colon 
		\WSob^{\alpha,\beta}_r(\IT_{d+1})
		\to \WSob^{\alpha-s,\beta+s\gamma}_r(\IT_{d+1}),
		\\  \Phi^{\mathrm{x}}_\gamma
		&\colon 
		\WSob^{\alpha,\beta}_r(\IT_{d+1})
		\to \WSob^{\alpha+s,\beta-s\gamma}_r(\IT_{d+1})
	\end{align*}
	for each $s\ge 0$.
\end{lemma}

\begin{proof}
	This can be shown from a direct computation via Lemma~\ref{lem:compactembedding} \ref{FM:1} and \eqref{eq:soboViaFourier}, where we crucially use that $(1+\abs{\xi_x}^2)^{\gamma/2}\le C\, (1+\abs{\xi_t}^2)^{1/2}$ whenever $\phi\neq 0$ and $(1+\abs{\xi_t}^2)^{1/(2\gamma)}\le C'\, (1+\abs{\xi_x}^2)^{1/2}$ whenever $1-\phi\neq 0$.
\end{proof}

\begin{remark}
	The decomposition constructed in the previous lemma is fully scalar (and symmetric in $\xi$), i.e. it preserves any differential constraint as well as the fact that a function $v$ is real valued. 
	In particular, if $v \in L_r(\IT_{d+1};\IR^m)$ obeys some linear differential constraint $\mathbb{A} v=0$, then $\Phi_\gamma^\mathrm{t} v$ and $\Phi_\gamma^\mathrm{x} v$ constructed component-wise, also satisfy $\mathbb{A}(\Phi_\gamma^\mathrm{t} v) = \mathbb{A}(\Phi_\gamma^\mathrm{x} v) =0$.
\end{remark}

\begin{proof}[Proof of Lemma~\ref{lem:compactembedding}]
	We use $\Phi^{\mathrm{t}}_\gamma$ and $\Phi^{\mathrm{x}}_\gamma $ from Lemma~\ref{lem:spacetimeConverter} with $\gamma = 2k$. 
	The claims are a consequence of the compact embedding in Lemma~\ref{lem:FM} \ref{FM:2} followed by the mapping $v\mapsto (\fourF v(0) + \Phi^{\mathrm{t}}_{2k} v) +\Phi^{\mathrm{x}}_{2k} v$:    
	\begin{enumerate}
		\item [\ref{cptemb:1}]
		$
		\leb r(\IT_{d+1})
		\hookrightarrow \hookrightarrow 
		\WSob^{\frac{-2k}{2k+1},\frac{-2k}{2k+1}}_r(\IT_{d+1})
		\hookrightarrow
		(\WSob^{-1,0}_r+\WSob^{0,-2k}_r)(\IT_{d+1});
		$
		\item [\ref{cptemb:2}]
		$
		\WSob^{0,-2k+1}_r(\IT_{d+1})
		\hookrightarrow \hookrightarrow
		\WSob^{\frac{-1}{2k+1},-2k+\frac{2k}{2k+1}}_r(\IT_{d+1})
		\hookrightarrow
		(\WSob^{-1,0}_r+\WSob^{0,-2k}_r)(\IT_{d+1}).
		$
	\end{enumerate}
\end{proof}


\subsection{Doubly anisotropic spaces}\label{sec:2_Spaces:pblmSpaces}
We now come back to the equation $$\partial_t \epsilon - \Qd^{\ast} \Qd \sigma=0.$$
While it is easy to formulate this equation in the space of distributions, $\mathcal{D}'(\mathcal T\times \Omega;\IR^m)$, in the case $ p\neq q $ there is no immediate natural Sobolev space $ \sobW{\alpha,\beta}{r} $ for this equation. 
However, if we evaluate the distribution $\partial_t \epsilon - \Qd^{\ast} \Qd \sigma$ for $\epsilon \in L_p$ and $\sigma \in L_q$ with a test function $\psi$, the distribution extends continuously to all $\psi$ with
\begin{equation*}
	\dell_t \psi\in \leb {p'}(\mathcal T\times \Omega;\IR^m),\quad \dell_x^\nu\psi\in \leb {q'}(\mathcal T\times \Omega;\IR^m),\quad \nu\in \IN^d,\,\abs{\nu}= 2k,
\end{equation*}
where $ \tfrac 1p +\tfrac 1{p'}=1, $  and  $\tfrac{1}{q}+ \tfrac 1{q'}=1 $.
This motivates us to define a Banachspace of test functions $\VSob^{1,2k}_{p',q'}(\mathcal T\times\Omega;\IR^m)$ as the closure of $C_c^\infty(\mathcal T\times\Omega;\IR^m)$ with respect to the norm
\begin{equation*}
	\norm{\psi}{\VSob^{1,2k}_{p',q'}(\mathcal T\times\Omega)}
	\coloneqq \norm{\dell_t\psi}{\leb{p'}(\mathcal T\times\Omega)}
			+	\norm{\nabla^{2k}_x\psi}{\leb {q'}(\mathcal T\times\Omega)}
			+	\abs{ \int_{\mathcal T \times \Omega} \psi \dx[(t,x)] }.
\end{equation*}
The last summand is only necessary on the torus, since on $(0,T) \times \Omega$ it can be controlled via Poincar\'e's inequality.

The equation $\dell_t\epsilon-\Qd^\ast \Qd\sigma=0$ can then be posed in the dual space, 
\begin{equation*}
	\WSob^{-1,-2k}_{p,q}(\mathcal T\times\Omega;\IR^m) 
	\coloneqq
	\left(\VSob^{1,2k}_{p',q'}(\mathcal T\times\Omega;\IR^m)\right)',
\end{equation*}
equipped with the natural norm.

Hence, our parabolic operator $\A$ is bounded and linear as a map
\begin{equation*}
    \A\colon 
    (\leb p\times \leb q)(\mathcal T\times\Omega;\IR^m\times\IR^m)
    \longto (\WSob^{-1,-2k}_{p,q}\times \sobW{0,-k'}{p})(\mathcal T\times\Omega;\IR^m\times\IR^n).
\end{equation*}

To compare the scale $ \WSob^{-1,-2k}_{p,q} $ with the spaces $ \WSob^{-1,0}_r+\WSob^{0,-2k}_r $ from the previous sections, we provide a characterisation in the following lemma.
Towards this, for $1<p,q<\infty$, we equip the subspace 
\begin{equation} \label{eq:summand}
	\IR^m + \dell_t \leb p(\mathcal T\times\Omega;\IR^m) + (\nabla_x^{2k})^\ast \leb q(\mathcal T\times\Omega;(\IR^d)^{2k}\otimes\IR^m)
\end{equation}
of $ \mathcal D'(\mathcal T\times\Omega;\IR^m) $ with the norm
\begin{equation*}
	\norm{\chi}{\IR^m+\dell_t \leb p + (\nabla_x^{2k})^\ast \leb q}
	\coloneqq
		\inf\left\{\abs{\chi_0}+\norm{\epsilon}{\leb p}+\norm{\sigma}{\leb q}:\chi =\chi_0 + \dell_t\epsilon + (\nabla_x^{2k})^\ast\sigma\right\}.
\end{equation*}
Note again that on a domain $ (0,T)\times\Omega $ the summand $ \IR^m $ in \eqref{eq:summand} gets absorbed by the other spaces.

\begin{lemma}\label{lem:identifyEquationSpace}
	Let $1<p,q<\infty$.
	\begin{enumerate}[label=(\roman*)]
		\item \label{it:identifyEquationSpace:pqCase}
		We have an isomorphism 
		\begin{equation*}
			\WSob^{-1,-2k}_{p,q}(\mathcal T\times\Omega;\IR^m) 
			\simeq \IR^m+(\dell_t \leb p + (\nabla_x^{2k})^\ast \leb q)(\mathcal T\times\Omega;\IR^m).
		\end{equation*}
		\item\label{it:identifyEquationSpace:rCase}
		 If $\mathcal{T}\times\Omega=\IT_{d+1}$ and $p=q=r$ we have 
		\begin{equation*}
			\WSob^{-1,-2k}_{r,r} (\IT_{d+1};\IR^m)
			\simeq 
			(\sobW{-1,0}{r}+ \sobW{0,-2k}r)(\IT_{d+1};\IR^m),
		\end{equation*}
	\end{enumerate}
\end{lemma}

\begin{proof}
	\begin{proofsteps}
		\item [\ref{it:identifyEquationSpace:pqCase}]
		The space $ \VSob^{1,2k}_{p',q'}(\mathcal T\times\Omega;\IR^m)$ embeds isometrically into a closed subspace of $\mathcal L_{p',q'}\coloneqq \IR^m\times(\leb {p'}\times \leb {q'})(\mathcal T\times\Omega;\IR^m\times(\IR^d)^{\otimes2k}\otimes\IR^m)$ via the map
        \begin{equation*}
        	F\colon \psi\longmapsto \left (\int_{\mathcal T\times\Omega}\psi\dx[(t,x)],\, -\dell_t\psi,\, \nabla_x^{2k}\psi\right ).
        \end{equation*}
		The adjoint map $F^\ast\colon (\chi_0,\epsilon,\sigma)\mapsto \chi_0 + \dell_t\epsilon+(\nabla_x^{2k})^\ast\sigma $ identifies $ \WSob^{-1,-2k}_{p,q}(\mathcal T\times\Omega) $ as the quotient $\mathcal L_{p,q}/\ker F^\ast$. 
		Upon extracting the definitions, this is precisely the desired isomorphism.
		\item [\ref{it:identifyEquationSpace:rCase}] 
        On $\IT_{d+1}$ with $p=q=r$ Fourier theory is available.
		With the mapping properties of $\Phi_{2k}^\mathrm t , \Phi_{2k}^\mathrm x$ in Lemma~\ref{lem:spacetimeConverter}, the map
		\begin{align*}
			G\colon\sobW{-1,0}{r}(\IT_{d+1}) + \sobW{0,-2k}r (\IT_{d+1})&\longto \mathcal L_{r,r}/\ker F^\ast,
		\\	\chi&\longmapsto \Big (\fourF\psi(0),\ \dell_t^{-1}\Phi_{2k}^\mathrm{t}\chi,\ ((\nabla_x^{2k})^\ast)^{-1} \Phi_{2k}^\mathrm{x}\chi,\Big)
		\end{align*}
		is continuous. 
		However, for all $\chi \in C^\infty(\IT_{d+1};\IR^m)$ we have $F^\ast (G(\chi)) = \chi$, so that a density argument yields the required isomorphism.
	\end{proofsteps}
\end{proof}

\begin{remark}
	We emphasise that in general the space $\WSob^{-1,0}_p + \WSob^{0,-2k}_q$ is \emph{not} the same space as $\WSob^{-1,-2k}_{p,q}$.
	Morally speaking, in $\WSob^{-1,0}_p + \WSob^{0,-2k}_q$ we can move purely spatial or purely temporal singularities from one summand to the other, while in $\WSob^{-1,-2k}_{p,q}$ this is not possible.
	Indeed, consider the case where $ p\gg q $ so that we can find $ \theta\in(0,d)\setminus\IN $ with
	\begin{equation*}
		\frac{d+1}{p}\le \theta< \frac{d+1}{q}-1.
	\end{equation*}
	Let $\eta\in C_c^\infty(-\tfrac{1}{2},\tfrac 12)$ with $\eta=1$ on $B_{\frac 14}(0)$, and consider the periodic extension of 
	\begin{equation*}
		v_\theta(t,x)\coloneqq 
		\dell_t\left (\eta(t)\, \abs{t}^{-\theta}\right ) 
		= \abs{t}^{-\theta-1} 
		\left (-\theta\,\eta(t)\tfrac{t}{\abs{t}} + \eta'(t)\, \abs{t}\right )
		,\quad (t,x)\in(-\tfrac{1}{2},\tfrac 12)^{d+1}.
	\end{equation*}
	With  Lemma~\ref{lem:identifyEquationSpace}, $ v_\theta $ is \emph{not} in $ \WSob^{-1,-2}_{p,q}(\IT_{d+1}) $, while we do have $ v_\theta\in \WSob^{-1,0}_p + \WSob^{0,-2k}_q  $, since
	\begin{equation*}
		\norm{v_\theta}{\WSob^{-1,0}_p + \WSob^{0,-2}_q}\le \norm{v_\theta}{\WSob^{0,-2}_q} = \norm{v_\theta}{\leb q} < \infty.
	\end{equation*}
\end{remark}
\section{$\A$-quasiconvexity for non-homogeneous operators}\label{sec:3_propA}
In this section, we first discuss some fundamental properties of the differential operator $\A$ as specified in \eqref{def:A}.
We then use these tools to investigate the notion of $\A$-quasiconvexity for non-homogeneous operators. 
Although most results of this section have an established analogue in the homogeneous case, our non-homogeneous setting introduces some new and non-linear technical difficulties.


\subsection{The Fourier symbol of $\A$}
First of all, recall that for $(\hat\epsilon,\hat\sigma) \in \C^m \times \C^m$ the function given by a wave
\[
(\epsilon,\sigma) = (\hat{\epsilon},\hat{\sigma}) \cdot e^{2\pi i\, \xi \cdot (t,x)}
\]
is in $\ker \A$ if and only if $(\hat{\epsilon},\hat{\sigma}) \in \ker \A [\xi]$. We therefore study $\ker \A[\xi]$ further in the following lemma.
\begin{lemma}\label{lem:studyKernel}
	\begin{enumerate} [label=(\roman*),ref=(\roman*)]
		\item \label{it:studyKernel:scaling}
		If $(\epsilon,\sigma)\in C^\infty(\IT_d;\IR^m\times\IR^m)$ satisfies $\A(\epsilon,\sigma) =0$ then for any $\lambda\in \IN$ we have $\A(\epsilon_{\lambda},\sigma_{\lambda}) =0$, where 
		\begin{equation*}
			\epsilon_\lambda(t,x)
			\coloneqq 
			\epsilon(\lambda^{2k}t,\lambda x)
			\quad\text{and}\quad
			\sigma_\lambda(t,x)
			\coloneqq      
			\sigma(\lambda^{2k}t,\lambda x).
		\end{equation*}
		\item \label{it:studyKernel:decomposition}
		We define the characteristic cone of $\A$ as 
		\[
		\Lambda_{\A} =  \bigcup_{\xi \in \IR^{d+1} \setminus \{0\}} \ker \A[\xi]\subset (\IC^m\times\IC^m).
		\]
		It can be decomposed along frequencies $\{\xi_t=0,\xi_x\neq 0\}$, $\{\xi_t\neq0,\xi_x=0\}$ and $\{\xi_t\neq0,\xi_x\neq 0\}$ into  
		\begin{equation*}
			\Lambda_\A = \Lambda_1\cup \Lambda_2 \cup \Lambda_3,
		\end{equation*}    
		where
		\begin{align*}
			\Lambda_1 
			&=   \bigcup_{\xi_x\in\IR^d\setminus\{0\}}
			\{
			(\hat\epsilon,\hat\sigma)
			:
			\Qd[\xi_x]\hat\sigma =0,
			\   \Pd^\ast[\xi_x]\hat\epsilon=0  
			\},
			&   \Lambda_2
			&=  \{0\}\times \IC^m
			\txtand
			\\  \Lambda_3
			&=   \bigcup_{\substack{
					\xi_t\in\IR,\,\xi_x\in\IR^d
					\\  \xi_t\neq 0,\,\xi_x\neq 0
			}}
			\{
			(\hat\epsilon,i\hat\sigma)
			:   
			2\pi\,
			\xi_t \hat\epsilon
			= \Qd^\ast\Qd[\xi_x]\hat\sigma
			\}.
		\end{align*}
		\item \label{it:studyKernel:constantRank}
		The operator $\A$ satisfies the constant-rank property, i.e. $\dim \ker \A[\xi]$ is constant along all $\xi \in \IR^{d+1} \setminus \{0\}$.
	\end{enumerate}    
\end{lemma}

\begin{remark}\label{rem:realCone}
	Although the sets $\Lambda_1,\Lambda_2$ and $\Lambda_3$ are complex valued, we will only use the real-real and real-imaginary subsets 
	\begin{equation*}
		\Lambda_j^{\IR} \coloneqq \Lambda_j\cap (\IR^m\times\IR^m),
		\quad j\in\{1,2\}\quad\txtand\quad \Lambda_3^{\IR} \coloneqq \Lambda_3\cap (\IR^m\times i\IR^m).
	\end{equation*}
\end{remark}

\begin{proof}[Proof of Lemma~\ref{lem:studyKernel}]
	\begin{enumerate}[leftmargin=0pt,itemindent=1cm]
		\item[\ref{it:studyKernel:scaling}] is clear by considering the order of the involved differential equation.
		\item[\ref{it:studyKernel:decomposition}]
		First consider $\xi\in\IR^{d+1}$ with $\xi_t=0$ and $\xi_x\neq 0$. 
		Then 
		\begin{equation*}
			\ker\A[\xi] 
			= 
			\{
			(\hat\epsilon,\hat\sigma)
			:   
			\Qd^\ast \Qd[\xi_x]\hat\sigma =0,\, \Pd^\ast[\xi_x]\hat\epsilon=0
			\}
			= 
			\{
			(\hat\epsilon,\hat\sigma)
			:   
			\Qd[\xi_x]\hat\sigma =0,\, \Pd^\ast[\xi_x]\hat\epsilon=0
			\}.
		\end{equation*} 
		Taking the union over $\xi\in\{\xi_t=0,\xi_x\neq0\}$ gives $\Lambda_1$.
		
		Second, let $\xi\in\IR^{d+1}$ with $\xi_t\neq 0$ and $\xi_x=0$ and consider $(\hat\epsilon,\hat\sigma)\in \ker\A[\xi] $.
		Then $\hat\sigma$ is arbitrary, whilst $2\pi\, \xi_t \hat\epsilon=0 $ implies $\hat \epsilon=0$. 
		Thus, $(\hat\epsilon,\hat\sigma)\in \Lambda_2$. 
		Conversely, every $(0,\hat\sigma) \in\Lambda_2$ is contained in $\ker\A[(\xi_t,0)]$ for any $\xi_t\neq 0$.
		
		Finally, if both $\xi_t,\xi_x\neq 0$, the first equation 
		\begin{equation*}
			2\pi i\, \xi_t \hat\epsilon = \Qd^\ast\Qd[\xi_x] \hat\sigma
		\end{equation*}
		already entails $\Pd^\ast[\xi_x] \hat\epsilon = 0$.
		Hence,
		\begin{equation*}
			\ker\A[\xi]
			=   
			\{
			(\hat\epsilon,i\hat\sigma)
			:   
			2\pi \,\xi_t\hat\epsilon  =\Qd^\ast \Qd[\xi_x]\hat\sigma
			\}, 
		\end{equation*}
		which gives $\Lambda_3$ when taking the union over all $\xi\in\{\xi_t\neq0,\xi_x\neq 0\}$.
		
		\item[\ref{it:studyKernel:constantRank}]
		In the previous step we have computed $\ker \A[\xi]$ for all $\xi\in\IR^{d+1}\setminus\{0\}$. 
		If $\xi_t=0,\xi_x\neq 0$, we recall that $(\ker\Qd[\xi_x])^\perp = \ker\Pd^\ast[\xi_x]$.
		This means 
		\begin{equation*}
			\dim \ker \A[\xi] = \dim \ker \Qd[\xi_x] + \dim \ker \Pd^\ast[\xi_x] = m.
		\end{equation*}
		If $\xi_t=0,\xi_x\neq 0$, obviously $\dim \ker\A[\xi] = m$.
		Lastly, if $\xi_t\neq0,\xi_x\neq 0$, we may write $\ker\A[\xi]$ as the graph of the function $\hat\sigma \mapsto \tfrac{1}{2\pi i\xi_t}\Qd^\ast\Qd[\xi_x]\hat \sigma$, which also gives
		$\dim\ker\A[\xi]=m$.
	\end{enumerate}    
\end{proof}


\subsection{Projections onto the nullspace of $ \A$}
We remind the reader of the Fourier theory for homogeneous differential constraints (cf. \cite[Lemma 2.14]{FM} and also \cite{GR}), which we adapt in Lemma~\ref{lem:AcorrectionOperator} below. 
In particular, there is a bounded and linear projection operator 
\begin{equation}\label{eq:kernelProjectionIsotropic}
	\Pi_\IA\colon \leb p(\IT_d;\IR^M)\longto \leb p(\IT_d;\IR^M)\cap \ker \IA
\end{equation}
for homogeneous constant-rank operators $\IA\colon C^\infty(\IT_d;\IR^M)\to C^\infty(\IT_d;\IR^N)$.

\begin{lemma}[Parabolic Projection] \label{lem:AcorrectionOperator}
	Let $p,q\in(1,\infty)$. 
	There exists a continuous \textbf{non-linear} projection map
	\begin{equation*}
		\widetilde \Pi_\A\colon(\leb p\times\leb q)(\IT_{d+1};\IR^m\times\IR^m)
		\longto (\leb p\times\leb q)(\IT_{d+1};\IR^m\times\IR^m),
	\end{equation*}
	such that 
	\begin{enumerate}[label=(\roman*)]
		\item \label{it:AcoorectionOperator:ALzero}
		$\A\circ \widetilde\Pi_\A=0$;
		\item \label{it:AcoorectionOperator:projection} 
		$\widetilde \Pi_\A\circ \widetilde \Pi_\A= \widetilde \Pi_\A $;
		\item \label{it:AcoorectionOperator:mean}
		for $w\in(\leb p\times \leb q)(\IT_{d+1};\IR^m\times\IR^m)$ it holds 
		\begin{equation*}
			\int_{\IT_{d+1}}\widetilde\Pi_\A w\dx[(t,x)] 
			= \int_{\IT_{d+1}} w\dx[(t,x)];
		\end{equation*}
		\item \label{it:AcoorectionOperator:estimate} 
		there is a constant $C_{p,q}>0$ such that
		\begin{equation*}
		\bigl \Vert w-\widetilde\Pi_\A w \bigr \Vert_{(\leb p\times \leb q)(\IT_{d+1})} 
			\le C_{p,q}\,\norm{\A w}{(\WSob^{-1,-2k}_{p,q}\times \WSob^{0,-k'}_p)(\IT_{d+1})}.
		\end{equation*}
	\end{enumerate}
\end{lemma}

\begin{proof}
	Let $w=(\epsilon,\sigma)\in(\leb p\times \leb q)(\IT_{d+1};\IR^m\times\IR^m)$ and write $\chi\coloneqq \dell_t\epsilon-\Qd^\ast \Qd \sigma$.
	We first project $\epsilon$ onto $\ker \Pd^\ast$ and correct it by the \emph{temporal} part of $ \chi $. 
	Similarly, $\sigma$ will be adjusted to remove the \emph{spatial} component.
	
	To identify these objects, consider the variational problem
	\begin{equation*}
		\inf\left \{\morm{(\epsilon',\sigma')}{} : \chi = \dell_t\epsilon' + (\nabla^{2k})^\ast \sigma' \right \},
		\quad \text{where}\quad \morm{(\epsilon',\sigma')}{}\coloneqq \left (\norm{\epsilon'}{\leb p}^2 + \norm{\sigma'}{\leb q}^2\right )^{1/2}
	\end{equation*}
	for $\epsilon'\in \leb p(\IT_{d+1};\IR^m)$ and $ \sigma' \in \leb q(\IT_{d+1};(\IR^d)^{\otimes2k}\otimes\IR^m) $.
	Since the space $ (\leb p\times \leb q,\morm{\circ}{})$ is uniformly convex  (see \cite{saitoUniformConvexityPsdirect2003}), there exists a unique minimiser $ (\epsilon_\chi,\sigma_\chi) $ in the closed subspace $\{(\epsilon',\sigma'):\chi = \dell_t\epsilon' + (\nabla^{2k})^\ast \sigma'\}$.
	With Lemma~\ref{lem:identifyEquationSpace} we find
	\begin{equation*}
		\norm{\chi}{\WSob^{-1,-2k}_{p,q}(\IT_{d+1})}\sim \morm{(\epsilon_\chi,\sigma_\chi)}{}.
	\end{equation*}
	Note that the mapping $ \chi\mapsto (\epsilon_\chi,\sigma_\chi) $ is in general not linear,
	but on bounded subsets of the space $ \WSob^{-1,-2k}_{p,q}(\IT_{d+1}) $ uniform convexity of $( \leb p\times\leb q, \morm{\circ}{}) $ implies that it is uniformly continuous with respect to the strong topologies of $ \WSob^{-1,-2k}_{p,q}(\IT_{d+1}) $ and $ (\leb p\times\leb q)(\IT_{d+1}) $.
	We define
	\begin{equation*}
		\widetilde\epsilon
		\coloneqq \Pi_{\Pd^\ast} \epsilon 
		- \Pi_{\Pd^\ast}
		\, \dell_t^{-1}
		\, \dell_t
		\, \epsilon_\chi
		\quad\text{and}\quad 
		\widetilde\sigma
		\coloneqq \sigma 
		+  	\Pi_{\Pd^\ast}	
		\,	(\Qd^\ast\Qd)^{-1}
		\, 	(\nabla^{2k})^\ast 
		\, 	\sigma_\chi.
	\end{equation*}
	It now suffices to verify that
	\begin{equation*}
		\widetilde \Pi_\A\colon (\epsilon,\sigma)\longmapsto(\widetilde\epsilon,\widetilde\sigma)
	\end{equation*}
	satisfies \ref{it:AcoorectionOperator:ALzero}--\ref{it:AcoorectionOperator:estimate}:
	\begin{enumerate}[itemindent=1cm, leftmargin=0em]
		\item [\ref{it:AcoorectionOperator:ALzero}]
		Due to the algebraic properties of the Moore--Penrose inverse and using that $ \Pd $ is a potential for $ \Qd $, we have
		\begin{equation*}
			\Pi_{\Pd^\ast}\circ \Qd^\ast\Qd= \Qd^\ast\Qd
			,\quad
			\Qd^\ast\Qd\circ \Pi_{\Pd^\ast}	
			\circ	(\Qd^\ast\Qd)^{-1} \circ (\nabla^{2k})^\ast
            = \Pi_{\Pd^\ast}\circ (\nabla^{2k})^\ast.
		\end{equation*}
		Thus,
		\begin{equation*}
			\dell_t \widetilde\epsilon -\Qd^\ast\Qd \widetilde \sigma 
			= \Pi_{\Pd^\ast} \left (
			\dell_t \epsilon-\Qd^\ast\Qd\sigma 
			-(\dell_t \epsilon_\chi + (\nabla^{2k})^\ast \sigma_\chi )
			\right )=0.
		\end{equation*}
		With \eqref{eq:kernelProjectionIsotropic} we also get $ \Pd^\ast \widetilde\epsilon =0$.
	\item [\ref{it:AcoorectionOperator:projection}]
		Clearly $ (\epsilon_\chi,\sigma_\chi)=0 $ if $\chi=0$ and the claim follows by the above.
	\item [\ref{it:AcoorectionOperator:mean}]
		This is due to \eqref{eq:moorePenrose:multiplierDef} and \eqref{eq:kernelProjectionIsotropic}, since generally $ \fourF( \IA^{-1}\IA v )(0)=0$ for all $ v\in C^\infty(\IT_{d+1}) $.
	\item [\ref{it:AcoorectionOperator:estimate}] 
		Using the multiplier estimates in Lemma~\ref{lem:FM}, we have
		\begin{equation*}
			\norm{\epsilon-\widetilde \epsilon}{\leb p}
			\le \norm{\epsilon-\Pi_{\Pd^\ast}\epsilon}{\leb p} 
			+\norm{\Pi_{\Pd^\ast}\dell_t^{-1}\dell_t\epsilon_\chi}{\leb p}
			\le C\, \left ( \norm{\Pd^\ast \epsilon}{\WSob^{0,-k'}_p} + \norm{\epsilon_\chi}{\leb p}\right )
		\end{equation*}
		and
		\begin{equation*}
			\norm{\sigma-\widetilde\sigma}{\leb q}
			\le \norm{\Pi_{\Pd^\ast}(\Qd^\ast\Qd)^{-1}(\nabla^{2k})^\ast \sigma_\chi}{\leb q}
			\le C\, \norm{\sigma_\chi}{\leb q}.
		\end{equation*}
		As $ \Vert \partial_t \epsilon - \Qd^{\ast} \Qd \sigma \Vert_{W^{-1,-2k}_{p,q}} = \norm{\chi}{\WSob^{-1,-2k}_{p,q}}\sim \norm{\epsilon_\chi}{\leb p} + \norm{\sigma_\chi}{\leb q} $, the claim follows.
	\end{enumerate}
\end{proof}

In the simply anisotropic case $ p=q=r $, we obtain a \emph{linear} result reminiscent of \cite[Lemma 2.14]{FM}, using similar techniques as therein.

\begin{corollary}[Linear Parabolic Projection]\label{lem:linearProjection}
	If $1<r<\infty$ then there exists a \textbf{linear} projection map 
	\begin{equation*}
		\Pi_\A \colon \leb r(\IT_{d+1};\IR^m\times\IR^m) 
		\longto \leb r(\IT_{d+1};\IR^m\times\IR^m)
	\end{equation*}
	that also satisfies properties \ref{it:AcoorectionOperator:ALzero}--\ref{it:AcoorectionOperator:estimate} from Lemma~\ref{lem:AcorrectionOperator}.
\end{corollary}

\begin{proof}
	With the operators $ \Phi_{2k}^\mathrm{t},\Phi_{2k}^\mathrm{x} $ from Lemma~\ref{lem:spacetimeConverter} we can give a simpler decomposition of $ \chi = \dell_t\epsilon -\Qd^\ast\Qd\sigma $ into spatial and temporal components:
	\begin{equation*}
		\bar\epsilon
		\coloneqq \Pi_{\Pd^\ast} \epsilon 
		- \Pi_{\Pd^\ast}
		\, 	\dell_t^{-1}
		\,	\Phi_{2k}^{\mathrm{t}}\chi
		\quad\text{and}\quad 
		\bar\sigma
		\coloneqq \sigma 
		+  	\Pi_{\Pd^\ast}	
		\, 	(\Qd^\ast\Qd)^{-1}
		\, 	\Phi_{2k}^{\mathrm{x}}\chi.
	\end{equation*}
	Properties \ref{it:AcoorectionOperator:ALzero}--\ref{it:AcoorectionOperator:estimate} follow for the \emph{linear} operator
	\begin{equation*}
		\Pi_\A\colon(\epsilon,\sigma)\longmapsto (\bar\epsilon,\bar\sigma)
	\end{equation*}
	by a similar computation as in Lemma~\ref{lem:AcorrectionOperator}.
	For the estimates we use the mapping properties of $ \Phi_{2k}^\mathrm{t},\Phi_{2k}^\mathrm{x} $ in the sense that
	\begin{equation*}
		\Phi_{2k}^\mathrm{t}\colon \sobW{-1,0}{r}+ \sobW{0,-2k}{r}\longto \sobW{-1,0}{r}
		,\quad
		\Phi_{2k}^\mathrm{x}\colon \sobW{-1,0}{r}+ \sobW{0,-2k}{r}\longto \sobW{0,-2k}{r}.
	\end{equation*}
\end{proof}


\subsection{Convergence properties of $\A$}
In the proof of the lower semicontinuity result Theorem~\ref{thm:anisotropicSuff} we make use of the following fact on the convergence of $\A w_j$ whenever the sequence $\{w_j=(\epsilon _j,\sigma_j)\}\subset (\leb p\times\leb q)(\IT_{d+1})$ is \emph{$ (p,q) $-equi-integrable}, that is
\begin{equation}\label{eq:defequi-integrable}
	\lim_{\delta\to 0} \sup_{\substack{U\subset \mathcal T\times \Omega,\\ \LL^{d+1}(U)\le \delta}}
	\sup_{j\in\IN}
	\int_{U} \abs{\epsilon_{j}}^p + \abs{\sigma_j}^q\dx[(t,x)] =0.
\end{equation}
See \cite[Lemma 3.1]{guerraCompensatedCompactnessContinuity2022} for a proof in the isotropic setting and \cite{contiDivCurlLemma2011} for a version in the setting of the $\div$-$\curl$ lemma.

\begin{lemma}\label{lem:DiffOpVsEquiInt}
	Let $ 1<p,q<\infty $ and  $1<r\le\min\{p,q\}$.
	Consider a $(p,q)$-equi-integrable sequence $\{w_j\}_{j\in\IN}\subset(\leb p\times\leb q)(\IT_{d+1})$ and assume that
	\begin{equation*}
		\A w_j\conv 0 \quad \txtin (\WSob^{-1,-2k}_{r,r}\times \WSob_r^{0,-k'})(\IT_{d+1}).
	\end{equation*}
	Then we already have
	\begin{equation*}
		\A w_j\conv 0 \quad \txtin (\WSob^{-1,-2k}_{p,q}\times \WSob_p^{0,-k'})(\IT_{d+1}).
	\end{equation*}
\end{lemma}

\begin{proof}
	The idea is to use parabolic Lipschitz truncation on the dual formulation of the $\WSob^{-1,-2k}_{p,q}\times \WSob_p^{0,-k'}$-norm.
	Recall that for any test function $\Psi\in \VSob^{1,2k}_{p',q'}\times \VSob_{p'}^{0,k'}$ and $\lambda>1$, there exists a function $\Psi^{\lambda}$ with the following properties (see e.g \cite{BDS}):
	\begin{enumerate}[label=(\roman*)]
		\item $\LL^{d+1}(\{\Psi\neq\Psi^\lambda\})
		\le C_{r,d}\, \lambda^{-1}\, \Big(1+\norm{\Psi}{\VSob^{1,2k}_{p',q'}\times \VSob_{p'}^{0,k'}}\Big)^{r'}$ 
		\hfill (small 'bad' set);
		\item $\norm{\Psi^\lambda}{\VSob^{1,2k}_{r',r'}\times \VSob_{r'}^{0,k'}}\le C_{r,d}\, \lambda$
		\hfill (Lipschitz truncation);
		\item $\norm{\Psi^\lambda}{\VSob^{1,2k}_{p',q'}\times \VSob_{q}^{0,k'}} \le C_d\, \norm{\Psi}{\VSob^{1,2k}_{p',q'}\times \VSob_{q}^{0,k'}}$
		\hfill (stability).
	\end{enumerate}
	To test the convergence, let $\Psi=(\psi,\phi)\in C^\infty(\IT_{d+1};\IR^m\times\IR^n)$ with $\norm{\Psi}{ \VSob^{1,2k}_{p',q'}\times \WSob_{p',0}^{0,k'}}\le 1$.
	Let $\Psi^\lambda=(\psi^\lambda,\phi^\lambda)$ be the corresponding Lipschitz truncation.
	We split the calculation into two parts
	\begin{equation*}
		\pair{\A w_j,\Psi}{}
		=\pair{\A w_j,\Psi-\Psi^\lambda}{} + \pair{\A w_j,\Psi^\lambda}{}.
	\end{equation*}
	On the first term we use the $(p,q)$-equi-integrability of $\{w_j\}_{j\in\IN}$. 
	Note that for a multiplier $\phi$ as in Lemma~\ref{lem:FM} the sequence $\{M_\phi w_j\}_{j\in\IN}$ is also $(p,q)$-equi-integrable (c.f. \cite[Lemma 2.14]{FM}).
	We will use this fact for the multipliers
	\begin{equation*}
		M_1\coloneqq ((\nabla_x^{k'})^\ast)^{-1}\Pd^\ast\quad \text{and}\quad M_2\coloneqq ((\nabla_x^{2k})^\ast)^{-1}\Qd^\ast\Qd,
	\end{equation*}
	which are $0$-homogeneous in space-time. 
	With this, we compute
	\begin{equation*}
		\begin{split}
			\pair{\A w_j,\Psi-\Psi^\lambda}{}
			&= 	\int_{\IT_{d+1}}
			\epsilon_j\cdot \big(-\dell_t(\psi-\psi^\lambda) + \Pd(\phi-\phi^\lambda)\big)
			+	\sigma_j\cdot(\Qd^\ast\Qd(\psi-\psi^\lambda))			
			\dx[(t,x)]
			\\	&= 	\int_{\IT_{d+1}}
			\big(
			-\epsilon_j\cdot\dell_t(\psi-\psi^\lambda) 
			+ 	M_1\epsilon_j 
			\cdot \nabla_x^{k'}(\phi-\phi^\lambda)
			\big)
			\dx[(t,x)]
			\\	&\kern4em
			+	\int_{\IT_{d+1}}	
			M_2\sigma_j
			\cdot	
			(\nabla_x^{2k}(\psi-\psi^\lambda))			
			\dx[(t,x)]
			\\ 	&\le 
			\left (
			\int_{\{\psi\neq\psi^\lambda\}} 	
			\abs{\epsilon_{j}}^p
			+\abs{M_1\epsilon_j}^p
			\dx[(t,x)]\right )^{1/p}
			\cdot \left (
			\norm{\psi-\psi^{\lambda}}{\VSob_{p'}^{1,0}}
			+\norm{\phi-\phi^{\lambda}}{\VSob_{p'}^{0,k'}} 
			\right )
			\\	&\kern4em
			+\left (\int_{\{\psi\neq\psi^\lambda\}} 
			\abs{M_2\sigma_j}^q
			\dx[(t,x)]
			\right )^{1/q}
			\cdot
			\norm{\psi-\psi^{\lambda}}{\VSob_{q'}^{0,2k}}
			\\ 	&\le C\, \omega(C\, \lambda^{-1})
			\,  \norm{\Psi}{\VSob^{1,2k}_{p',q'}\times \VSob_{p'}^{0,k'}} 
            \\ &\leq C \omega (C \lambda^{-1}),
		\end{split}
	\end{equation*}
	where 
	\begin{equation*}
		\omega(\delta)\coloneqq \sup_{\abs{U}\le \delta}\sup_{j\in \IN} 
		\left (
		\int_{U} 	
		\abs{\epsilon_{j}}^p
		+\abs{M_1\epsilon_j}^p
		\dx[(t,x)]
		\right )^{1/p}
		+
		\left (
		\int_{U} 
		\abs{M_2\sigma_j}^q
		\dx[(t,x)]
		\right )^{1/q}
	\end{equation*}
	is the modulus of $(p,q)$-equi-integrabilty. 
	For the second term, we can estimate
	\begin{equation*}
		\begin{split}
			\pair{\A w_j,\Psi^\lambda}{}
			&\le \norm{\A w_j}{ \WSob^{-1,-2k}_{r,r}\times \WSob_{r}^{0,-k'}}\,
			\norm{\Psi^\lambda}{\VSob^{1,2k}_{r',r'}\times \VSob_{r'}^{0,k'}}
			\\	&\le C\,\lambda\, \norm{\A w_j}{ \WSob^{-1,-2k}_{p,q}\times \WSob_{r}^{0,-k'}}.
		\end{split}
	\end{equation*}
	Since the right-hand sides of both previous inequalities are independent of $\Psi$, we can take the supremum and observe
	\begin{equation*}
		\norm{\A w_j}{\WSob^{-1,-2k}_{p,q}\times \WSob_{p}^{0,-k'}} 
		\le 
		C\, \omega(C\, \lambda^{-1})
		+	C\,\lambda\, \norm{\A w_j}{ \WSob^{-1,-2k}_{r,r}\times \WSob_{r}^{0,-k'}}.
	\end{equation*}
	Since $\, \norm{\A w_j}{ \WSob^{-1,-2k}_{r,r}\times \WSob_{r}^{0,-k'}} \to 0$ by assumption, by taking a suitable subsequence $\lambda_j\to\infty$, we conclude 
	\begin{equation*}
		\A w_j\conv 0 \quad \txtin  (\WSob^{-1,-2k}_{p,q}\times \WSob_{p}^{0,-k'})(\IT_{d+1}).
	\end{equation*}
\end{proof}


\subsection{$\A$-quasiconvexity}\label{sec:3_propA:2_Aqcvx}
We now consider a functional $I$ together with an integrand $f$ that depends on both variables $\epsilon$ and $\sigma$. 
In more detail, we assume that $f \colon \IR^m \times \IR^m \to [0,\infty)$ is continuous and furthermore obeys
\begin{equation} \label{eq:growth}
	f(\hat \epsilon,\hat \sigma) 
	\leq C_1\,(1+ \vert \hat \epsilon \vert^p + \vert \hat \sigma \vert^q).
\end{equation}
We consider the $(d+1)$-torus $\IT_{d+1}$. 
Denote by $C_{\#}^{\infty}(\IT_{d+1})$ those smooth functions that have average zero on the torus. 
In accordance with \cite{FM}, we define the concept of $\A$-quasiconvexity as follows.
\begin{definition} \label{def:quasiconvexity}
	We call a continuous function $f \colon \IR^m \times \IR^m \to [0,\infty)$ $\A$-quasiconvex if for all $(\hat {\epsilon},\hat {\sigma}) \in \IR^m \times \IR^m$ and all $(\epsilon,\sigma) \in C^{\infty}_{\#}(\IT_{d+1};\IR^m \times \IR^{m})$ with $\A(\epsilon,\sigma)=0$ we have 
	\begin{equation} \label{Jensen}
		f(\hat {\epsilon},\hat {\sigma}) \leq \int_{\IT_{d+1}} f(\hat {\epsilon}+\epsilon(t,x),\hat {\sigma}+\sigma(t,x)) \dx[(t,x)].
	\end{equation}
\end{definition}
Of course, we could extend the definition to functions $f \colon \IR^m \to \R \cup \{\infty\}$, but the lower-emicontinuity results below often use the non-negativity assumption (see also \cite{GR2}).
While the lower semicontinuity results of \cite{FM} are left untouched (also cf. Theorem~\ref{thm:lsc_nec} below), we first prove some elementary observations on functions enjoying this $\A$-quasiconvexity condition.

\begin{lemma} \label{lemma:base}
	Suppose that $f \colon \IR^m \times \IR^m \to [0,\infty)$ is $\A$-quasiconvex. 
	Then the following statements hold:
	\begin{enumerate}[label=(\roman*)]
		\item \label{it:base:1_coneConvex}
		$f$ is convex along $\Lambda_1^\IR \cup \Lambda_2^\IR$ (recall Remark \ref{rem:realCone}), i.e. for any $w \in \Lambda_1^\IR \cup \Lambda_2^\IR$ and $(\hat\epsilon,\hat\sigma)\in \IR^m\times\IR^m$ the function
		\[
		\lambda \longmapsto f((\hat \epsilon,\hat \sigma)+\lambda w)
		\]
		is convex.
		\item \label{it:base:2_PQcvx}
		Fix $\hat \sigma_0 \in \IR^m$. 
		Then $\hat \epsilon \mapsto f(\hat \epsilon,\hat {\sigma}_0)$ is $\Pd^{\ast}$-quasiconvex.
		\item \label{it:base:3_2ndConvex}
		Fix $\hat \epsilon_0 \in \IR^m$. 
		Then $\hat \sigma \mapsto f(\hat {\epsilon}_0,\hat \sigma)$ is convex.
		\item \label{it:base:4_locallyLipschitz}
		If $f$ in addition obeys \eqref{eq:growth}, then $f$ is locally Lipschitz continuous and
		\begin{align*}
			\abs{f(\hat \epsilon,\hat \sigma) -f(\hat \epsilon',\hat \sigma')} 
			\,\leq&
			C\,
			(1  + \abs{\hat \epsilon}^{p-1} 
			+ \abs{\hat \epsilon'}^{p-1}     
			+ \abs{\hat \sigma}
			+ \abs{\hat \sigma'}
			)
			\,  
			\abs{\hat \epsilon-\hat \epsilon'}
			\\  &+
			C\,
			(1  + \abs{\hat \epsilon}
			+ \abs{\hat \epsilon'}
			+ \abs{\hat \sigma}^{q-1}
			+ \abs{\hat \sigma'}^{q-1}
			)
			\,  
			\abs{\hat \sigma-\hat \sigma'}.
		\end{align*}
		\item \label{it:base:5_epsSigQuasiAffine}
		The functions $(\hat \epsilon,\hat \sigma) \mapsto \hat \epsilon \cdot \hat \sigma$ and $(\hat \epsilon,\hat \sigma) \mapsto -\hat \epsilon \cdot \hat \sigma$ are both $\A$-quasiconvex.
	\end{enumerate}
\end{lemma}
\begin{proof} 
	\begin{enumerate}[leftmargin=0cm,itemindent=1cm]
	\item[\ref{it:base:1_coneConvex}]  
		Convexity along the directions $\Lambda_1^\IR \cup \Lambda_2^\IR$ follows from the well-known case of $\IA$-quasiconvexity for isotropic operators $\IA$:
		By extending functions to be constant in time and space respectively, we observe that $\A$-quasiconvexity implies 
		\begin{equation}\label{eq:base:isotropicConvex}
			\begin{split}
				f(\hat\epsilon,\hat\sigma) 
				\le& \int_{\IT_{d}} f((\hat\epsilon,\hat\sigma) + w)\dx,
				\quad  \forall w\in C_\#^\infty(\IT_{d};\IR^{m}\times\IR^m),
				\,      \Qd w_2,\Pd^\ast w_1 =0
				\ \text{and}
				\\  f(\hat\epsilon,\hat\sigma) 
				\le& \int_{\IT_{1}} f((\hat\epsilon,\hat\sigma) + w)\dx[t],
				\quad \forall  w\in C_\#^\infty(\IT_{1};\IR^{m}\times\IR^m),
				\,      \dell_t w_1=0.
			\end{split}
		\end{equation}
		From there, the proof can be adapted from \cite[Prop. 3.4]{FM} to our setting of higher-order operators.
		\item[\ref{it:base:2_PQcvx}]
		$\Pd^\ast$-quasiconvexity follows from \eqref{eq:base:isotropicConvex}.
	\item[\ref{it:base:3_2ndConvex}]
		Convexity in the second argument is a consequence of claim \ref{it:base:1_coneConvex}, as $\Lambda_2^\IR=\{0\}\times \IR^m$.
	\item[\ref{it:base:4_locallyLipschitz}]
		This is a straightforward adaptation from the homogeneous case, we refer to \cite{MP,GR2} for details.
	\item[\ref{it:base:5_epsSigQuasiAffine}]
		Let $(\epsilon,\sigma)\in C_\#^\infty(\IT_{d+1};\IR^m\times\IR^m)$ with $\A (\epsilon,\sigma)=0$.
		Recall that the second equation $\Pd^\ast \epsilon=0$ allows us to solve $\epsilon=\Qd^\ast u$ via  \eqref{eq:moorePenrose:multiplierDef}.
		Hence, we may integrate by parts to obtain
		\begin{align*}
			\int_{\IT_{d+1}}
			\epsilon\cdot \sigma
			\dx[(t,x)]
			& =   \int_{\IT_{d+1}}
			\Qd^\ast u\cdot \sigma
			\dx[(t,x)]
			=   (-1)^k
			\int_{\IT_{d+1}}
			u\cdot \Qd\sigma
			\dx[(t,x)]
			\\  &=   (-1)^k
			\int_{\IT_{d+1}}
			u\cdot \dell_t u
			\dx[(t,x)]
			=0.
		\end{align*}
		In particular, for any constants $(\hat \epsilon,\hat \sigma)\in \IR^m\times\IR^m$,
		\begin{equation*}
			\hat \epsilon\cdot \hat  \sigma
			= \int_{\IT_{d+1}}
			(\hat \epsilon+\epsilon)\cdot (\hat  \sigma + \sigma)
			\dx[(t,x)].
		\end{equation*}
		From here the claim is immediate.
	\end{enumerate}
\end{proof}


\subsection{$\Lambda_{\A}$-convexity and polyconvexity}\label{sec:3_propA:3_otherConvexity}
Before we come to the lower semicontinuity result in the spirit of \cite{FM}, let us briefly discuss notions that are related to $\A$-quasiconvexity and severely easier to check. 
While any \emph{convex} function is $\A$-quasiconvex, the notion of convexity is usually way too restrictive for the type of functions we consider.

Instead, we study the following sufficiently sharp chain of implications (see also \cite{Mskript,BP})
\begin{equation} \label{implicationchain}
	f \text{ $\A$-polyconvex} \quad \Longrightarrow \quad f \text{ $\A$-quasiconvex} \quad \Longrightarrow \quad
	f \text{ $\Lambda_{\A}$-convex}.
\end{equation}
To give \eqref{implicationchain} some meaning, we define the notions of poly- and cone-convexity.

\begin{definition}[Polyconvexity]
	\label{def:polyconvexity}
	Let $g \colon \IR^m \times \IR^m \to \IR^l$ be a function such that all its coordinates are $\A$-quasiaffine. 
    Let $h \colon \IR^l \to \R$ be convex. We then call
	\[
	f \colon \IR^m \times \IR^m \longto \R, \quad f(\hat \epsilon,\hat \sigma) = h(g(\hat \epsilon,\hat \sigma))
	\]
	an \textbf{$\A$-polyconvex} function.
\end{definition}
Observe that the (affine) identity map $(\hat \epsilon,\hat \sigma) \mapsto (\hat \epsilon,\hat \sigma)$ and $(\hat \epsilon,\hat \sigma) \mapsto \hat \epsilon \cdot \hat \sigma$ are both $\A$-quasiaffine. 
So we usually consider  $\A$-polyconvex functions of the form 
\[
f(\hat \epsilon,\hat \sigma) = h(\hat \epsilon,\hat \sigma,\hat \epsilon \cdot \hat \sigma)
\]
for convex $h \colon \IR^m \times \IR^m \times \R \to \R$.

\begin{definition}[Cone convexity]
	Let $f \colon \IR^m \times \IR^m \to \R$. We say that $f$ is $\Lambda_{\A}$-convex if the following assertions hold true:
	\begin{enumerate}[label=(\roman*)]
		\item \label{cone:1} 
		For any $w \in \Lambda_1^\IR \cup \Lambda_2^\IR$ and any $(\hat \epsilon,\hat \sigma) \in \IR^m \times \IR^m$ the function
		\[
		h \longmapsto f((\hat \epsilon,\hat \sigma) + hw) 
		\]
		is convex;
		\item \label{cone:2} 
		Let $w=(w_1,iw_2) \in \Lambda_3^\IR$ and $(\hat \epsilon,\hat \sigma) \in \IR^m \times \IR^m$. 
		The function
		\[
		(h_1,h_2) \longmapsto f(\hat \epsilon+h_1 w_1, \hat \sigma+h_2 w_2)
		\]
		is subharmonic, i.e.
		\begin{equation*}
			f(\epsilon,\sigma)
			\le \int_0^1 f(
			\hat\epsilon+ R\, \sin(2\pi i \theta)\, w_1,
			\hat\sigma + R\, \cos(2\pi i \theta)\, w_2
			)\dx[\theta]
		\end{equation*}
		for all $R>0$.
	\end{enumerate}
\end{definition}
We highlight that condition \ref{cone:1} is the standard condition for convexity for \emph{homogeneous} differential operators and corresponds to \emph{rank-one} convexity in the special case of quasiconvexity. 
In contrast, condition \ref{cone:2} is only appearing in the non-homogeneous case and seems to be, at least in the context of $\A$-quasiconvexity, quite new.

\begin{lemma}\label{lem:polyConeCvx}
	Let $f \in C(\IR^m \times \IR^m)$. 
	Then the chain of implications \eqref{implicationchain} holds true:
	\begin{enumerate} [label=(\roman*),ref=(\roman*)]
		\item\label{it:implicationchain:polyToQuasi} if $f$ is $\A$-polyconvex, then $f$ is $\A$-quasiconvex;
		\item\label{it:implicationchain:QuasiToCone} if $f$ is $\A$-quasiconvex, then $f$ is $\Lambda_{\A}$-convex.
	\end{enumerate}
\end{lemma}

\begin{proof}
	\begin{enumerate}[leftmargin=0pt, itemindent=1cm]
		\item [\ref{it:implicationchain:polyToQuasi}]
		Let $f=h\circ g$, with $h$ convex and $g$ $\A$-quasiaffine in each component.
		Then, using $\A$-quasiaffinity and Jensen's inequality,
		\begin{equation*}
			f(\hat\epsilon,\hat\sigma)
			= h\left(
			\int_{\IT_{d+1}}
			g((\hat\epsilon,\hat\sigma)+w(t,x))
			\dx[(t,x)]
			\right)
			\le \int_{\IT_{d+1}}
			(h\circ g)((\hat\epsilon,\hat\sigma)+w(t,x))
			\dx[(t,x)],
		\end{equation*}
		for all $w\in C^\infty_\#(\IT_{d+1})$ with $\A w=0$.
		\item[\ref{it:implicationchain:QuasiToCone}]
		Convexity along directions $w\in \Lambda_1^\IR \cup \Lambda_2^\IR$ was the content of Lemma~\ref{lemma:base}. 
		Therefore, we focus on the subharmonicity of $f$ along planes $\{(\hat\epsilon+h_1w_1,\hat\sigma+h_2w_2):h_1,h_2\in\IR\}$, whenever $w=(w_1,iw_2)\in\Lambda_3^\IR$.
        Without loss of generality, assume $R=1$.
		First, assume that $\xi\in\IZ^{d+1}$ with $\xi_t,\xi_x\neq 0$ and $2\pi \xi_t w_1 = \Qd ^\ast\Qd [\xi_x]w_2.$
		In this case we define
		\begin{equation*}
			w(t,x)
			\coloneqq 
			\Big(w_1\, \sin(2\pi\, \xi\cdot ((-1)^kt,x))
			,w_2\, \cos(2\pi\, \xi\cdot ((-1)^k t,x) )\Big)
			,\quad(t,x)\in\IT_{d+1}.
		\end{equation*}
		Clearly $w\in C^\infty_\#(\IT_{d+1};\IR^m\times\IR^m)$ and moreover, by construction,
		$\A w=0$.
		Thus, $\A$-quasiconvexity gives
		\begin{equation*}
			f(\hat\epsilon,\hat\sigma)
			\le \int_{\IT_{d+1}} f((\hat\epsilon,\hat\sigma) + w(t,x))\dx[(t,x)]
			=   \int_{0}^1 f (
			(\hat\epsilon,\hat\sigma) 
			+   (   w_1\,  \sin(2\pi r)
			,   w_2\,  \cos(2\pi r))
			)
			\dx[r].
		\end{equation*}
		The general case, where $\xi\in \IR^{d+1}$ and $2\pi \xi_t w_1 = \Qd ^\ast\Qd [\xi_x]w_2$, can be obtained as a limit of the integer-valued cases:
		Observe that the set 
		\begin{equation*}
			\Xi\coloneqq\{(\lambda^{-2k}\xi_t',\lambda^{-1}\xi'_x): \lambda\in\IN,\xi'\in\IZ^{d+1},\,\xi'_t\neq 0,\xi'_x\neq 0\}
		\end{equation*}
		is dense in $\{\xi'\in\IR^{d+1}:\xi_t'\neq 0,\xi_x'\neq 0\}\subset \IR^{d+1}$.
		Pick a sequence $\xi^j\in\Xi$ so that $\xi^j\conv \xi$ and define 
		\begin{equation*}
			w_1^j\coloneqq \frac{1}{2\pi \xi_t^j}\Qd^\ast \Qd [\xi_x^j]w_2.
		\end{equation*}
		Taking into account the scaling of $\A$ as per Lemma~\ref{lem:studyKernel} \ref{it:studyKernel:scaling}, we see that $(w_1^j,iw_2)\in \ker \A[\widetilde \xi^j]$ for appropriate $\widetilde \xi^j\in \IZ^{d+1}$. 
		The map $\xi'\mapsto \tfrac{1}{2\pi \xi_t'}\Qd^\ast \Qd [\xi_x']$ is continuous on $\{\xi'_t\neq 0,\xi'_x\neq0\}$, hence $w_1^j\conv w_1$.
		As a consequence of the integer-valued case and the continuity of $f$, we conclude
		\begin{align*}
			f(\hat\epsilon,\hat\sigma)
			&\le \lim_{j\to\infty}
			\int_{0}^1 f (
			(\hat\epsilon,\hat\sigma) 
			+   (   w_1^j\,  \sin(2\pi r)
			,   w_2\,  \cos(2\pi r))
			)
			\dx[r]
			\\  &=   \int_{0}^1 f (
			(\hat\epsilon,\hat\sigma) 
			+   (   w_1\,  \sin(2\pi r)
			,   w_2\,  \cos(2\pi r))
			)
			\dx[r].
		\end{align*}
	\end{enumerate}
\end{proof}
\section{Weak lower semicontinuity}\label{sec:4_wlsc}
We now come to an analogue of the weak lower semicontinuity statement of Fonseca \& M\"uller \cite{FM} under differential constraints that are adjusted to our parabolic setting. We only deal with the autonomous setup (i.e. the function $f$ does not depend on $(t,x)$), although our method can be routinely generalised via approximation with locally autonomous functions. 
The main difficulty and difference to \cite{FM} lies in the non-homogeneity of the involved spaces.
To this end, define the functional $I$ acting on $(L_p\times L_q)(\mathcal T\times \Omega;\IR^m\times\IR^m)$, $\mathcal T= (0,T)$, $\Omega \subset \IR^d$ open and bounded, by
\begin{equation} \label{eq:defI}
	I(\epsilon,\sigma) = \int_{\mathcal{T}\times \Omega} f(\epsilon,\sigma) \dx[(t,x)].
\end{equation}

\begin{theorem}[Necessity of $\A$-quasiconvexity] \label{thm:lsc_nec}
	Suppose that $I$ as defined in \eqref{eq:defI} is weakly lower semicontinuous along $\ker \A$, i.e. for all $(\epsilon_j, \sigma_j) \weakto (\epsilon,\sigma)$ in $(L_p \times L_q)(\mathcal{T} \times \Omega;\IR^m\times\IR^m)$
	with $\A(\epsilon_j,\sigma_j) =0$ we have
	\[
	I(\epsilon,\sigma) \leq \liminf_{j \to \infty} I(\epsilon_j,\sigma_j).
	\]
	Assume that $f$ is locally bounded.
	Then $f$ is $\A$-quasiconvex.
\end{theorem}
\begin{proof}[Proof] 
	Consider $(\epsilon,\sigma)\in C^\infty(\IT_{d+1};\IR^m\times\IR^m)$ with
	\begin{equation*}
		\A (\epsilon,\sigma)=0
		\quad\text{and}\quad 
		\int_{\IT_{d+1}} (\epsilon,\sigma)\dx[(t,x)]
		= (\hat\epsilon,\hat\sigma).
	\end{equation*}
	For each $j\in\IN$ we set
	\begin{equation*}
		(\epsilon_j,\sigma_j)(t,x)
		\coloneqq
		(\epsilon,\sigma)(j^{2k} t,j x)
		,\quad 
		(t,x)\in\mathcal{T}\times\Omega.
	\end{equation*}
    We extend this function periodically to $\R \times \R^d$ and then restrict it to $\mathcal{T} \times \Omega$.
	As argued in Lemma~\ref{lem:studyKernel} \ref{it:studyKernel:scaling}, we still have that
	$\A(\epsilon_j,\sigma_j)=0$.
	Moreover, it is easy to see that
	$(\epsilon_j,\sigma_j)\wconv(\hat\epsilon,\hat\sigma)$ in $(\leb{p}\times \leb q)(\mathcal{T}\times\Omega)$, as $j\to\infty$.
	Thus, weak lower semicontinuity of $I$ gives
	\begin{equation*}
		f(\hat \epsilon,\hat\sigma)
		\le \liminf_{j\to\infty}
		\frac{1}{\LL^{d+1}(\mathcal{T}\times\Omega)}
		\int_{\mathcal{T}\times\Omega} f(\epsilon_j,\sigma_j)\dx[(t,x)].
	\end{equation*}
	On the other hand, since $f$ is locally bounded, it follows that
	\begin{equation*}
		f(\epsilon_j,\sigma_j)
		\wconv[\ast] 
		\int_{\IT_{d+1}}f(\epsilon,\sigma)\dx[(t,x)]
	\end{equation*}
	in $\leb\infty(\mathcal{T}\times\Omega)$.
	From here the claim is a direct consequence.
\end{proof}

To show that $\A$-quasiconvexity is sufficient is slightly more challenging, due to the involved non-homogeneous spaces. 
In particular, we also require a statement where the differential constraint $\A(\epsilon,\sigma)$ is not satisfied exactly, but compact in a suitable negative space.
This result is established in Theorem~\ref{thm:anisotropicSuff} and immediately implies the following:

\begin{theorem}[Sufficiency of $\A$-quasiconvexity] \label{thm:lsc_suff}
	Suppose that $f$ is $\A$-quasiconvex and satisfies the growth condition \eqref{eq:growth}. 
	Then for any sequence  $(\epsilon_j, \sigma_j) \weakto (\epsilon,\sigma)$ in $ (L_p \times L_q )(\mathcal{T} \times \Omega;\IR^m \times \R^m)$
	with $\A(\epsilon_j,\sigma_j) =0$ we have
	\[
	I(\epsilon,\sigma) \leq \liminf_{j \to \infty} I(\epsilon_j,\sigma_j).
	\]
\end{theorem}

Although the proof of this theorem is in spirit quite close to the case for homogeneous operators in \cite{FM}, it faces some technical difficulties. 

If $p=q=2$, then the proof closely follows the structure of \cite{FM} with a few differences. 
One may approximate the limit function $(\epsilon,\sigma)$ by a function that is a sum of characteristic functions on small space-time cubes and therefore may reduce to a statement where $(\epsilon,\sigma)$ is constant. 
Using a non-homogeneous \cite[Lemma 2.15]{FM}, i.e. Corollary~\ref{lem:linearProjection}, one may then construct out of the sequence $(\epsilon_j,\sigma_j)$ a periodic $\A$-free sequence of $\leb\infty$ functions close to $(\epsilon_j,\sigma_j)$ for which we may use the definition of $\A$-quasiconvexity (which involves $\A$-free periodic functions). 
The Jensen-type inequality in the definition of $\A$-quasiconvexity then ensures weak lower semicontinuity for $\epsilon_j,\sigma_j$.
The case $p\neq q$ however does not allow this approach, since the compact embeddings of Lemma~\ref{lem:compactembedding} do not extend to the doubly anisotropic spaces $\WSob^{-1,-2k}_{p,q}$ defined in Section~\ref{sec:2_Spaces:pblmSpaces}.

For this reason (and also due to its later relevance in Section~\ref{sec:6_applNS_largeP}) we instead prove a generalised version of Lemma~\ref{thm:lsc_suff}: 
The differential constraint need not be satisfied exactly, but it suffices that it lies compactly in some negative Sobolev space. This space is, however, slightly larger than the negative Sobolev space of \cite{FM}.

The key step of this result lies in the Lipschitz truncation (cf. Lemma~\ref{lem:DiffOpVsEquiInt}), which is well-known in slightly different contexts, for instance existence theory for non-linear fluid equations (cf. \cite{BDF,BDS}). 
That being said, in the context of $\A$-quasiconvexity and weak lower semicontinuity this argument rarely appears. 
We first state the result for the homogeneous case and briefly describe the strategy of the proof before we proceed with the proof of Theorem~\ref{thm:lsc_suff}.


\subsection{The statement for a static problem}\label{sec:4_wlsc:2_intermezzo}
First, we discuss a very slight generalisation of the result by Fonseca \& M\"uller in the static homogeneous setting. 
To this end, let $\Abb \colon C^{\infty}(\IR^d;\IR^m) \to C^{\infty}(\IR^d;\IR^n)$ denote a homogeneous differential operator of order k.
\begin{theorem} \label{FM:generalised}
	Let $f \colon \IR^m \to \R$ be $\mathbb{A}$-quasiconvex and satisfy the growth condition  
    \begin{equation*}
		0\le 
        f(\hat u) 
		\leq C\,(1+ \vert \hat u \vert^p )
		,\quad \hat u\in \IR^m.
	\end{equation*}
    
	Fix $1<r<p$. 
	If a sequence $\{u_j\}_{j\in\IN} \subset L_p(\Omega;\IR^m)$ satisfies $u_j \weakto u$ weakly in $L_p(\Omega;\IR^m)$ and $\mathbb{A} u_j \longto \mathbb{A} u$ strongly in $\WSob^{-k}_r(\Omega;\IR^m)$, as $j \to \infty$, then
	\[
	\int_{\Omega} f(u(x)) \dx \leq \liminf_{j \to \infty} \int_{\Omega} f(u_j(x)) \dx.
	\]   
\end{theorem}
In Fonseca-M\"uller \cite{FM}, it is required that $\Abb u_j \to \Abb u$ in the space $\WSob^{-k}_p(\Omega;\IR^m)$, i.e. the space with the natural growth exponent, while here $r$ is smaller than $p$.
There are two possible ways to prove Theorem~\ref{FM:generalised}. 
In either method, one first reduces to $p$-equi-integrable sequences $u_j$. After that,
\begin{itemize}
	\item one can perform a Lipschitz truncation directly on the sequence $\{u_j\}_{j\in\IN}$ to obtain functions $u_j^{\lambda}$ bounded in $L_{\infty}$ and project them onto the space of $\A$-free functions again (cf. \cite[Proposition 6.1]{BGS}). 
	One may then use the standard version of the lower semicontinuity result. 
	For the original sequence $u_j$, lower semicontinuity is obtained by carefully estimating the error that appears due to truncation and projection;
	\item one can use a homogeneous analogue of Lemma~\ref{lem:DiffOpVsEquiInt} to show that for a $p$-equi-integrable sequence the convergence, which is a priori only weak in $\WSob^{-k}_p$ and strong in $\WSob^{-k}_r$, $r<p$, is in fact strong in $\WSob^{-k}_{p}$ as well.
	We therefore reduce to the classical case. 
	Observe that for such an argument we implicitly also need Lipschitz truncation as in Lemma \ref{lem:DiffOpVsEquiInt} to show strong convergence in $\WSob^{-k}_{p}$.
\end{itemize}


\subsection{General exponents and a generalisation of Theorem~\ref{thm:lsc_suff}}\label{sec:4_wlsc:2_generalExponents}
In this section we prove a generalisation of Theorem~\ref{thm:lsc_suff} that is also suited to our later application to the non-Newtonian Navier--Stokes equations.

\begin{theorem}\label{thm:anisotropicSuff}
	Let $p,q,r\in (1,\infty)$ so that $1<r \le \min\{p,q\}$. 
	Assume 
	\begin{equation*}
		\begin{cases}
			\hfill   w_j\wconv[\quad] w \hfill
			&   \text{in } 
			(\leb p \times \leb q) (\mathcal T\times \Omega;\IR^m\times\IR^m ),
			\\\hfill  \A w_j\conv[\quad] \A w \hfill
			&   \text{in } 
			(\WSob^{-1,-2k}_{r,r}
			\times \sobW{0,-k'}{r})(\mathcal T\times \Omega;\IR^m\times\IR^n ),
		\end{cases}
	\end{equation*}
	and that $f$ is $\A$-quasiconvex and permits the growth condition
	\begin{equation*}
		0\le 
        f(\hat \epsilon,\hat \sigma) 
		\leq C_1\,(1+ \vert \hat \epsilon \vert^p + \vert \hat \sigma \vert^q)
		,\quad (\hat \epsilon,\hat \sigma) \in \IR^m\times\IR^m.
	\end{equation*}

	Then, we have
	\begin{equation*}
		\int\limits_{\mathcal{T}\times \Omega} 
		f(w)
		\dx[(t,x)]
		\le 
		\liminf_{j\to\infty}
		\int\limits_{\mathcal{T}\times \Omega} 
		f(w_j)
		\dx[(t,x)].
	\end{equation*}
\end{theorem}
\begin{proof}
	We subdivide the proof into four main steps:
	\begin{enumerate}[itemindent=1cm, label=\textit{Step \arabic*}]
		\item \label{pf:anisotropicSuff:1_equiInt}
		Replace $\{w_j\}_{j\in\IN}$ with  $(p,q)$-equi-integrable sequence;
		\item \label{pf:anisotropicSuff:2_transferTrous}
		Reduce to the case $\mathcal{T}\times \Omega = (0,1)^{d+1}$ and $w=0$;
		\item \label{pf:anisotropicSuff:3_convDiffEq}
		Reduce to the case where the domain is a torus in space-time;
		\item \label{pf:anisotropicSuff:4_tructation}
		Upgrade the convergence and use $\A$-quasiconvexity.
	\end{enumerate}
	The expert reader might already be familiar with \ref{pf:anisotropicSuff:1_equiInt}--\ref{pf:anisotropicSuff:3_convDiffEq} in the homogeneous setup, and apart from the 'hard part' achieved in Lemma~\ref{lem:DiffOpVsEquiInt}, the last step is a standard application of the definition of $\A$-quasiconvexity.	
	
	\noindent \ref{pf:anisotropicSuff:1_equiInt}:
	We first reduce to the case where $\{w_j=(\epsilon_j,\sigma_j)\}_{j\in \IN}$ is $(p,q)$-equi-integrable. 
	That is, we want
	\begin{equation*}
		\lim_{\delta\to 0} \sup_{\substack{U\subset \mathcal T\times \Omega,\\ \LL^{d+1}(U)\le \delta}}
		\sup_{j\in\IN}
		\int_{U} \abs{\epsilon_{j}}^p + \abs{\sigma_j}^q\dx[(t,x)] =0.
	\end{equation*}
	Set $v_j\coloneqq w_j-w$. 
	Clearly $v_j\wconv 0$ in $ (\leb p\times \leb q )(\mathcal T\times \Omega)$.
	According to \cite[Lemma 1.2]{FMP} resp. \cite[Lemma 2.15]{FM}, there exists a sequence $\{\widetilde v_j\}_{j\in \IN}\subset  (\leb p\times \leb q )(\mathcal T\times \Omega)$ which 
	\begin{enumerate}[label =(\roman*),itemindent=1em]
		\item is $(p,q)$-equi-integrable;
		\item satisfies $\widetilde v_j-v_j\conv 0$ in $L_r(\mathcal T\times \Omega;\IR^m\times\IR^m)$;
		\item satisfies $\widetilde v_j\wconv 0$ in $( \leb p\times \leb q )(\mathcal T\times \Omega;\IR^m \times \R^m)$.
	\end{enumerate}
	Set $\widetilde w_j\coloneqq \widetilde v_j + w$.
	From the above it follows that
	\begin{equation*}
		\begin{cases}
			\hfill \{\widetilde w_j\}_{j\in \IN} \hfill
			&   \text{is $(p,q)$-equi-integrable},
			\\ \hfill \widetilde w_j\wconv[\quad] w \hfill
			&   \text{in } (\leb p\times \leb q)(\mathcal T\times \Omega; \R^m \times \R^m),\text{ and }
			\\ \hfill \A \widetilde w_j\conv[\quad] \A w \hfill
			&   \text{in } (\WSob^{-1,-2k}_{r,r}\times \sobW{0,-k'}{r})(\mathcal T\times \Omega;\R^n \times \R^n).
		\end{cases}
	\end{equation*}
	
	We will now prove that the sequence $\{\widetilde w_j\}_{j\in \IN}$ does not increase the liminf:
	\begin{equation*}
		\liminf_{j\to\infty}
		\int_{\mathcal{T}\times \Omega} 
		f(\widetilde w_j)
		\dx[(t,x)]
		\le 
		\liminf_{j\to\infty}
		\int_{\mathcal{T}\times \Omega} 
		f(w_j)
		\dx[(t,x)].
	\end{equation*}
	Let $\lambda\in \IN$ be a cut-off parameter and define
	\begin{equation*}
		U^\lambda_j
		\coloneqq  
		\{
		(t,x)\in \mathcal T\times \Omega
		:   \abs{\widetilde w_j(t,x)}\le\lambda, 
		\abs{w_j(t,x)}\le\lambda
		\}.    
	\end{equation*}
	On the one hand, continuity and $(p,q)$-growth of $f$, together with $\widetilde w_j-w_j\conv 0$ in measure imply that for each $\lambda\in\IN$
	\begin{equation*}
		\limsup_{j\to\infty}
		\abs{
			\int_{U^\lambda_j} 
			f(\widetilde w_j)-f(w_j)
			\dx[(t,x)]}=0.
	\end{equation*}
	On the other hand, as $\sup_{j\in\IN}\LL^{d+1}\left ( (\mathcal{T}\times \Omega )\setminus U_j^{\lambda}\right )\conv 0$ for $\lambda \to \infty$, $(p,q)$-equi-integrability of $\{\widetilde w_j\}_{j\in \IN}$ yields
	\begin{equation*}
		\lim_{\lambda \to\infty} \sup_{j\in\IN}
		\int_{(\mathcal{T}\times \Omega)\setminus U_j^\lambda} f(\widetilde w_j)\dx[(t,x)] =0.
	\end{equation*}
	Hence, using $f\ge 0$, we conclude
	\begin{align*}
		\liminf_{j\to \infty}
		\int_{\mathcal{T}\times \Omega} 
		f(\widetilde w_j)
		\dx[(t,x)]
		=&   \lim_{\lambda\to\infty}
		\liminf_{j\to \infty}
		\int_{U_j^\lambda} 
		f(\widetilde w_j)
		\dx[(t,x)]
		\\  =&   \lim_{\lambda\to\infty}
		\liminf_{j\to \infty}
		\int_{U_j^\lambda} 
		f(w_j)
		\dx[(t,x)]
		\\  \le& \liminf_{j\to \infty}
		\int_{\mathcal{T}\times \Omega} 
		f(w_j)
		\dx[(t,x)].
	\end{align*}
	\ref{pf:anisotropicSuff:2_transferTrous}: We may reduce to the case where $\mathcal{T} \times \Omega =(0,1)^{d+1}$ in the standard fashion (see for instance the original work \cite{Morrey} or \cite{AF84}): We approximate the target function $w $ by functions that are constant on small cubes of parabolic scaling; let us call these $w^{\varepsilon}$. 
	Then we may consider the alternative sequence
	\[
	w_j^{\varepsilon} = w_j + (w^{\varepsilon} -w),
	\]
	show the statement for this sequence, and let $\varepsilon \to 0$ afterwards. 
	The lower semicontinuity statement on the whole domain then reduces to proving result on each parabolic cube. 
	By rescaling we may reduce to $(0,1)^{d+1}$.\\
	\ref{pf:anisotropicSuff:3_convDiffEq}: 
	Suppose now that we have a sequence of functions that converges weakly to a constant function $w_0$ on the cube and satisfies,
	\begin{equation*}
		\begin{cases}
			\hfill \{ w_j\}_{j\in \IN} 
			&   \text{is $(p,q)$-equi-integrable},
			\\ \hfill  w_j\wconv[\quad] w_0 
			&   \text{in } (\leb p\times \leb q )((0,1)^{d+1}), \text{ and }
			\\ \hfill \A w_j \conv[\quad] 0 
			&   \text{in } (\WSob^{-1,-2k}_{r,r}\times \sobW{0,-k'}{r})((0,1)^{d+1}).
		\end{cases}
	\end{equation*}
	We may suppose that $w_0 =0$ and aim to show that 
	\begin{equation} \label{eq:pf:3:claim}
		f(0) \leq \liminf_{j \to \infty} \int_{(0,1)^{d+1}} f(w_j) \dx[(t,x)].
	\end{equation}
	If we can reduce this statement to the same statement, where $(0,1)^{d+1}$ is replaced by $\IT_{d+1}$, we may continue with \ref{pf:anisotropicSuff:4_tructation}. 
	The difficulty lies in the fact that $\A w_j$ does not necessarily converge to $0$ in $(\WSob^{-1,-2k}_{r,r}\times \sobW{0,-k'}{r})(\IT_{d+1})$ (that is, taking test functions with \emph{periodic} instead of zero boundary conditions on the cube).
	
	We now may argue as in \cite[Lemma 2.15]{FM}: 
	For $\varepsilon>0$ consider a cut-off function $\eta_{\varepsilon}$ such that 
	\begin{enumerate}[label=(\roman*)]
		\item $\eta_{\varepsilon} \in C_c^{\infty}((0,1)^{d+1})$;
		\item $\eta_{\varepsilon} =1$ \text{on } $(\varepsilon,1-\varepsilon)^{d+1}$;
		\item $\Vert D^k \eta_{\varepsilon} \Vert_{L_{\infty}} \leq C_k \,\varepsilon^{-k}$ for any $k \in \N$.
	\end{enumerate}
	We claim that there is a sequence $\varepsilon_j \to 0$ such that for $\widetilde{w}_j := \eta_{\varepsilon_j} w_j $
	\begin{equation} \label{pf:3:1}
		\liminf_{j \to \infty} \int_{(0,1)^{d+1}} f(\widetilde{w}_j) \dx[(t,x)] = \liminf_{j \to \infty} \int_{(0,1)^{d+1}} f(w_j) \dx[(t,x)]
	\end{equation}
	\emph{and}
	\begin{equation}\label{pf:3:2}
		\A w_j \conv 0 \quad \text{in } (\WSob^{-1,-2k}_{r,r}\times \sobW{0,-k'}{r})(\IT_{d+1}).
	\end{equation}
	If so, Equation \eqref{eq:pf:3:claim} indeed reduces to the statement on the torus and we are finished.
	
	Equation \eqref{pf:3:1} follows routinely by observing that, due the growth condition,
	\begin{align*}
	    &\limsup_{j \to \infty} \int_{(0,1)^{d+1}} \vert f(\widetilde{w}_j) -f(w_j)\vert \dx[(t,x)]
		\leq 
		\\ &\limsup_{j \to \infty} \int_{(0,1)^{d+1} \setminus (\varepsilon_j,(1-\varepsilon_j))^{d+1}} C\,(1+ \vert \epsilon_j \vert^p + \vert\sigma_j \vert^q) \dx[(t,x)] =0,
	\end{align*}
	as $w_j=(\epsilon_j,\sigma_j)$ is $(p,q)$-equi-integrable and $\varepsilon_j \to 0$. 
	
	It remains to construct a sequence $\varepsilon_j$ such that $\A \widetilde{w}_j \to 0$.
	To this end, we look at both equations entailing $\A$ separately. 
	Recall from Lemma~\ref{lem:identifyEquationSpace} that we can identify $ \WSob^{-1,-2k}_{r,r}(\IT_{d+1})\simeq (\WSob^{-1,0}_r+ \WSob^{0,-2k}_r)(\IT_{d+1}) $, which we will use throughout this step.
	
	First,
	\begin{equation} \label{first}
		\partial_t (\eta_{\varepsilon}\,  \epsilon_j) - \Qd^{\ast} \Qd (\eta_{\varepsilon}\, \sigma_j) 
		= \eta_{\varepsilon} \left( \partial_t \epsilon_j - \Qd^{\ast} \Qd \sigma_j \right) + R_j,
	\end{equation}
	where $R_j$ denotes all the remainder terms in which at least one derivative acts on $\eta_{\varepsilon}$. 
	Observe that, for fixed $\eps>0$,
	\[
	\eta_\eps \left( \partial_t \epsilon_j - \Qd^{\ast} \Qd \sigma_j \right) \longrightarrow 0 \quad
	\text{in } (\WSob^{-1,0}_r+ \WSob^{0,-2k}_r)(\IT_{d+1}) \quad \text{as } j \to \infty.
	\]
	On the other hand, the remainder term can be decomposed into 
	\begin{equation}\label{eq:proof_suff:step_3:defR}
		R_j= \partial_t \eta_{\varepsilon}\, \epsilon_j + \sum_{i=1}^{2k} Q_i(\nabla^i_x \eta_{\varepsilon} \otimes \nabla^{2k-i}_x \sigma_j)
	\end{equation}
	for suitable linear maps $Q_i$ coming out of $\Qd^{\ast} \Qd$.
	As $\epsilon_j \weakto \epsilon$ in $L_r(\IT_{d+1})$ and $\sigma_j \weakto \sigma$ in $L_r(\IT_{d+1})$, we obtain that, for fixed $\varepsilon>0$,
	\[
	R_j \wconv 0 \quad \text{in } (L_r + \WSob^{0,-2k+1}_r)(\IT_{d+1})
	\]
	Applying the compact Sobolev embedding of Lemma~\ref{lem:compactembedding} one obtains
	\[
	R_j \conv 0 \quad \text{in }( \WSob^{-1,0}_r+ \WSob^{0,-2k}_r)(\IT_{d+1}).
	\]
	This finishes the argument that the term in \eqref{first} converges, i.e.
	\begin{equation}\label{eq:proof_suff:convergence_first}
		\bigl(\partial_t (\eta_{\varepsilon}\,  \epsilon_j) - \Qd^{\ast} \Qd (\eta_{\varepsilon}\, \sigma_j) \bigr) \longrightarrow 0 \quad \text{in }( \WSob^{-1,0}_r+ \WSob^{0,-2k}_r)(\IT_{d+1}).
	\end{equation}
	By a diagonal-sequence argument we can select a sequence $\varepsilon_j \to 0$ such that \eqref{first} converges strongly in $(\WSob^{-1,0}_r + \WSob^{0,-2k}_r)(\IT_{d+1})$.
	
	For the second component of $\A (\epsilon,\sigma)$, we can construct a similar decomposition
	\begin{equation}\label{eq:proof_suff:step_3:2ndEqn}
		\Pd^{\ast} (\eta_{\varepsilon}\,  \epsilon_j) 
		= \eta_{\varepsilon}\,  \Pd^{\ast} \epsilon_j 
		+ \widetilde S_j,
	\end{equation}
	where $\widetilde S_j$ consists of lower-order terms of the form $\dell_x^{\alpha-\beta}\eta_\eps\,\dell_x^\beta\epsilon_j$, with $\abs{\alpha}=k'$ and $\beta <\alpha$, $\alpha,\beta\in \IN^{d}$.    
	In particular, we get for fixed $\eps>0$
	\[
	\eta_{\varepsilon}\, \Pd^{\ast} \epsilon_j \conv 0  \quad \text{in } \WSob^{0,-k'}_r(\IT_{d+1}) 
	\quad \text{and} \quad 
	\widetilde S_j \wconv 0 \quad \text{in } \WSob_r^{0,-k'+1}(\IT_{d+1}).
	\]
	If we were to use the same argument as before, we would only be able to assert
	\begin{equation*}
		\widetilde S_j\conv 0 \txtin (\sobW{-1,0}{r}+\sobW{0,-k'}{r})(\IT_{d+1}).
	\end{equation*} 
	To get rid of the temporal part of the sum, we instead use the first component of $\A$:
	Taking the operators $\Phi_{2k}^\mathrm{t}$ and $\Phi_{2k}^\mathrm{x}$ from Lemma~\ref{lem:spacetimeConverter}, we refine the decomposition in \eqref{eq:proof_suff:step_3:2ndEqn} into
	\begin{equation}\label{eq:proof_suff:step_3:defST}
		\begin{split}
			\Pd^{\ast} (\eta_{\varepsilon}\,  \epsilon_j) 
			&= \Phi_{{2k}}^{\mathrm{t}} \Pd^{\ast} (\eta_{\varepsilon}\,  \epsilon_j)
			+ \Phi_{2k}^\mathrm{x}\Pd^{\ast} (\eta_{\varepsilon}\,  \epsilon_j)\\
			&= \left (\Phi_{{2k}}^{\mathrm{t}} (\eta_\eps\, \Pd^\ast\epsilon_j)
			+ \Phi_{{2k}}^{\mathrm{t}} \widetilde S_j\right )
			+ \dell_t^{-1} \Phi_{2k}^\mathrm{x} \Pd^\ast\Bigl(\dell_t (\eta_{\varepsilon}\,  \epsilon_j)-\Qd^\ast\Qd(\eta_\eps\, \sigma_j) \Bigr)
			\\	&\eqqcolon \Phi_{{2k}}^{\mathrm{t}} (\eta_\eps\,\Pd ^\ast\epsilon_j)
			+ S_j
			+ T_j,
		\end{split}
	\end{equation}
	where we also used that $\Pd^\ast\Qd^\ast\Qd(\eta_\eps\, \sigma_j)=0$. \\
	The first summand $ \Phi_{{2k}}^{\mathrm{t}} (\eta_\eps\,\Pd ^\ast\epsilon_j) $ converges in $\sobW{0,-k'}{r}(\IT_{d+1})$ by assumption.\\
	For $S_j$, Lemma~\ref{lem:FM} and Lemma~\ref{lem:spacetimeConverter} imply $\Phi_{{2k}}^\mathrm{t}\colon \sobW{0,-k'+1}{r}(\IT_{d+1})\cptmap \sobW{0,-k'}{r}(\IT_{d+1})$, so that for fixed $\eps>0$ we have 
	\begin{equation*}
		S_j\conv 0 \quad \txtin \sobW{0,-k'}{r}(\IT_{d+1}).
	\end{equation*}
	For $T_j$, since $\dell_t^{-1}\circ\Phi_{{2k}}^\mathrm{x}\circ \Pd^\ast $ defines a continuous map $ \sobW{-1,0}{r}+ \sobW{0,-2k}{r}\to \sobW{0,-k'}{r}$, \eqref{eq:proof_suff:convergence_first} yields
	\begin{equation*}
		T_j\conv 0 \quad \txtin \sobW{0,-k'}{r}(\IT_{d+1}).
	\end{equation*}
	This allows us to adjust the sequence $\eps_j\conv 0$ to ensure 
	\begin{equation*}
		\Pd^{\ast} (\eta_{\varepsilon_j}\,  \epsilon_j) \conv 0 \quad \txtin \sobW{0,-k'}{r}(\IT_{d+1})\quad \txt{as} j\to \infty.
	\end{equation*}
	This finishes the proof of \ref{pf:anisotropicSuff:3_convDiffEq}.
	
	\noindent\ref{pf:anisotropicSuff:4_tructation}:
	Suppose that
	\begin{equation*}
		\begin{cases}
			\hfill \{ w_j\}_{j\in \IN} 
			&   \text{is $(p,q)$-equi-integrable},
			\\ \hfill  w_j\wconv[\quad] w_0 
			&   \text{in } (\leb p\times \leb q )(\IT_{d+1}), \text{ and }
			\\ \hfill \A w_j \conv[\quad] 0 
			&   \text{in } (\WSob^{-1,-2k}_{r,r}\times \sobW{0,-k'}{r})(\IT_{d+1}).
		\end{cases}
	\end{equation*}
	Assuming that $f$ is $\A$-quasiconvex and obeys the $(p,q)$-growth bound,  we claim that
	\begin{equation} \label{eq:claim}
		f(w_0) \leq \liminf_{j \to \infty} \int_{\IT_{d+1}} f(w_j) \dx[(t,x)].
	\end{equation} 
	Without loss of generality, we may assume $ w_0=0 $.
	First, note that due to $(p,q)$-equi-integrability of the sequence, Lemma~\ref{lem:DiffOpVsEquiInt} implies that the convergence of $\A w_j$ happens in $(\WSob^{-1,-2k}_{p,q} \times \WSob^{0,-k'}_p)(\IT_{d+1})$.
	As discussed in Lemma~\ref{lem:AcorrectionOperator}, we can use the non-linear projection onto $ \ker \A $ to define
	\begin{equation*}
		\bar w_j\coloneqq \widetilde \Pi_\A w_j.
	\end{equation*}
	Then $\A \bar w_j=0$ and 
	\begin{equation*}
		w_j-\bar w_j\conv 0\quad\txtin (\leb p\times\leb q )(\IT_{d+1}).
	\end{equation*}
	Consequently, with $\A$-quasiconvexity and Lemma~\ref{lemma:base} \ref{it:base:4_locallyLipschitz}, we get
	\begin{equation*}
		\begin{split}
			f(0)
			&\le
			\liminf_{j\to\infty} f\left (\int_{\IT_{d+1}}\bar w_j\dx[(t,x)]\right ) 
			\\	&\le  \liminf_{j\to\infty} \int_{\IT_{d+1}}f(\bar w_j)\dx[(t,x)] 
			\\  &=  \liminf_{j\to\infty} \int_{\IT_{d+1}}f( w_j)\dx[(t,x)].
		\end{split}
	\end{equation*}
	This concludes the proof.
\end{proof}


\subsection{Relaxation}\label{sec:4_wlsc:3_relax}
If $f$ fails to be $\A$-quasiconvex, the functional $I$ will not be lower semicontinuous (see Theorem \ref{thm:lsc_nec}).
For this reason we seek to replace $f$ with its \emph{$\A$-quasiconvex envelope}, which, adapting \cite{FM,BFL}, is defined to be
\begin{equation}\label{eq:qcvxEnvelope}
   \begin{split}
   	 \QA f(\hat\epsilon,\hat\sigma)
    \coloneqq
    \inf\Big\{
    \int_{\IT_{d+1}} f(\hat\epsilon+\epsilon,\hat\sigma+\sigma)\dx[(t,x)]
    :
    \A(\epsilon,\sigma)=0, 
    &
    \\      
    \epsilon,\sigma\in C^\infty_{\zeromean}(\IT_{d+1};\IR_{\sym,0}^{d\times d})
    &
    \Big\},
    \quad\hat\epsilon,\hat\sigma\in \IR_{\sym,0}^{d\times d}.
   \end{split}
\end{equation}

We can then prove a suitable \emph{relaxation} statement in the same vein as the homogeneous case \cite{BFL}:
\begin{theorem}[Relaxation] \label{thm:relaxation}
    Let $p,q\in(1,\infty)$.
    Suppose that $f \in C(\IR^m \times \IR^m;[0,\infty))$ obeys the growth condition
    \[
    f(\epsilon,\sigma) \leq C\, (1+ \vert \epsilon \vert^p+\vert \sigma \vert^q).
    \]
    Then, for any $\varepsilon>0$ and any $(\epsilon,\sigma) \in (L_p \times L_q)(\mathcal{T} \times \Omega;\IR^m \times \IR^m)$ there exists a  bounded sequence $\{(\epsilon_j,\sigma_j)\}_{{j\in\IN}} \subset (L_p \times L_q)(\mathcal{T} \times \Omega;\IR^m \times \IR^m)$ such that
    \begin{enumerate} [label=(\roman*)]
        \item $(\epsilon_j,\sigma_j) \weakto (\epsilon,\sigma)$ in $L_p \times L_q$;
        \item $\A(\epsilon_j,\sigma_j) \longto \A(\epsilon,\sigma)$ in $\WSob^{-1,-2k}_{p,q} \times \WSob^{0,-k'}_p$; \label{item2}
        \item we have
        \[
        \liminf_{j \to \infty}\int_{\mathcal{T} \times \Omega} f(\epsilon_j,\sigma_j) \dx[(t,x)]
        \le \int_{\mathcal{T} \times \Omega} \QA f(\epsilon,\sigma) \dx[(t,x)] + \varepsilon.
        \]
    \end{enumerate}
\end{theorem}
    Together with the argument of \ref{pf:anisotropicSuff:3_convDiffEq} in the proof of Theorem \ref{thm:anisotropicSuff} (see also \cite[Lemma 2.15]{FM}), we can sharpen \ref{item2} to an equality $\A(\epsilon,\sigma)=\A(\epsilon_j,\sigma_j)$.
    Thus, if we define the functional
    \begin{equation*}
        \widetilde I (\epsilon,\sigma) 
        \coloneqq
            \int_{\mathcal T\times \Omega} \QA f(\epsilon,\sigma)\dx[(t,x)] 
        ,
    \end{equation*}
    acting on $ (\leb p\times \leb q)(\mathcal T\times \Omega;\IR^m\times\IR^m)$,
    we find that $\widetilde I$ is indeed a relaxation of $I$ as defined in \eqref{eq:defI}:
    \begin{equation*}
    	\widetilde I (\epsilon,\sigma) 
    	= \inf\Big\{
    	\liminf_{j \to \infty} I(\epsilon_j,\sigma_j): 
    	(\epsilon_j,\sigma_j)\wconv (\epsilon,\sigma)\txtin \leb p\times \leb q
        ,\, \A(\epsilon_j,\sigma_j) = \A(\epsilon,\sigma)
    	\Big\}.
    \end{equation*}
\begin{proof}[Sketch of Theorem \ref{thm:relaxation}.]
    We refrain from giving a full proof of this statement, as it can be obtained with similar methods to \cite{AF84} (the proof in \cite{BFL} is via Young measures). 
    In particular, in contrast to Theorem \ref{thm:anisotropicSuff}, the additional difficulty of dealing with anisotropic spaces is minor. 
    The strategy is as follows:
	\begin{itemize}
		\item First, approximate $(\epsilon,\sigma)$ by a function $(\bar{\epsilon},\bar{\sigma})$ that is piecewise constant on small parabolic cubes;
		\item On each small parabolic cube find a compactly supported  $\A$-free function that almost realises the definition of $\QA^R f$. 
        Here, $\QA^R f$ is defined by
		\begin{equation*}
			\begin{split}
				\QA^R f(\hat\epsilon,\hat\sigma)
			\coloneqq
			\inf\Big\{
			&
			\int_{\IT_{d+1}} f(\hat\epsilon+\epsilon,\hat\sigma+\sigma)\dx[(t,x)]
			:
			\A(\epsilon,\sigma)=0, 
			\\  &    
			\epsilon,\sigma\in C^\infty_{\zeromean}(\IT_{d+1};\IR_{\sym,0}^{d\times d})
            ,\  \Vert \epsilon \Vert_{L^{\infty}}, \Vert \sigma \Vert_{L^{\infty}} \leq R
			\Big\},
			\end{split}
		\end{equation*}
		We then get a function $\psi$ with $L^{\infty}$ norm bounded by $2R$ such that on $D$
		\[
		\frac 1{\LL^{d+1}(D)}\int_{D}  f((\bar{\epsilon},\bar{\sigma}) + \psi) \dx[(t,x)]  \leq \QA^R f(\bar{\epsilon},\bar{\sigma}) + \varepsilon
		\]
		\item By combining these functions defined on disjoint cubes, we can construct a function $\psi$ such that
		\begin{equation*}
		    \begin{split}  
		\int_{\mathcal{T}\times \Omega} f((\bar{\epsilon},\bar{\sigma})+ \psi)  \dx[(t,x)]
            &   \leq \sum_D \LL^{d+1}(D) \bigl(\QA^R f(\bar{\epsilon},\bar{\sigma}) + \varepsilon \bigr)
        \\ &    = \int_{\mathcal{T} \times \Omega} (\QA^R f((\bar{\epsilon},\bar{\sigma})) + \varepsilon) \dx[(t,x)]
		    \end{split}
		\end{equation*}
		\item Letting the size of the cubes tend to zero yields the result, as $\QA^R f \to \QA f$ monotonically, i.e. one may choose $R$ large enough such that the difference in the integrals is only of scale $\varepsilon$.
	\end{itemize}
\end{proof}
We highlight that with a suitable coercivity condition for $f$, we may also replace $\varepsilon$ in Theorem \ref{thm:relaxation}
with $0$, as one can take a diagonal sequence in $j$ and $\varepsilon$. 
This is not possible without an a priori bound on the $(L_p \times L_q)$ norm of $(\epsilon_j,\sigma_j)$ as one might see via the example
\[
f(\epsilon,\sigma) = e^{- \vert \epsilon \vert^2 -\vert \sigma \vert^2},
\]
whose relaxation is $\QA f=0$.
\section{Adjustments for some pseudo-differential operators}\label{sec:5_pseudoDiffOps}

As the (non-Newtonian) Navier-Stokes equations feature the incompressibility constraint $\divergence u=0$ together with a pressure $\pi$, they do not immediately fit into the previously developed framework. Therefore, we need a few minor adjustments that are explained in this section.


\subsection{Abstract setup for additional side constrains}\label{sec:5_pseudoDiffOps:1_reformulation}
Motivated by the incompressibility condition in the Navier-Stokes equation, we now consider the system
\begin{equation}\label{eq:system11}
	\begin{cases}
		\dell_t u = \Qd\sigma - \Sd^\ast \pi &\text{in } \IT_{d+1},\\
		\Sd u=0 &\text{in } \IT_{d+1},\\
		\Qd^\ast u = \epsilon  &\text{in } \IT_{d+1},
	\end{cases}
\end{equation}
which includes an additional constraint $\Sd u =0$, where $\Sd\colon C^\infty(\IT_d;\IR^n)\to C^\infty(\IT_d;\IR^\ell)$ is another constant-rank differential operator (recall Section~\ref{sec:2_Spaces:1_consrRankOps} for this terminology).
The function $\pi$ (the `pressure') can be thought of as a Lagrange multiplier for the equation.
Our first goal is to rewrite \eqref{eq:system11} in a way that is comparable to \eqref{eq:defA}. 

To accomplish this, we assume the following compatibility condition:
\begin{equation}\label{eq:formulations:compatibility}
	\Image \Sd^\ast[\xi_x]\subset \Image \Qd[\xi_x],
	\quad\xi_x\in\IR^d\setminus\{0\}.
\end{equation}

With this, we can fix a homogeneous constant-rank potential $\Rd\colon C^\infty(\IT_d;\IR^\ell)\to C^\infty(\IT_d;\IR^n)$ for $\Sd$. 
Observe that for the projection $\Pi_{\Sd}$ onto the kernel of $\Sd$ as in \eqref{eq:kernelProjectionIsotropic} we have
\begin{equation*}
	\Pi_\Sd v = \Rd\circ\Rd^{-1}v + \int_{\IT_{d}} v\dx,\quad v\in C^\infty(\IT_{d};\IR^n).
\end{equation*}
Moreover, there exists an annihilator $\widetilde \Pd^\ast\colon C^\infty(\IT_d;\IR^m)\to C^\infty(\IT_d;\IR^\ell)$ for $\Qd^\ast \circ \Rd$.
Here, \eqref{eq:formulations:compatibility} ensures the constant-rank condition.

Using these operators, we are able to formulate two equivalent conditions for the existence of $u$ and $\pi$ so that \eqref{eq:system11} holds.

\begin{lemma}
	Let $\epsilon,\sigma\in C^\infty(\IT_{d+1};\IR^m)$.
	Assume \eqref{eq:formulations:compatibility}.
	Then, the following conditions are equivalent:
	\begin{enumerate}[label=(\roman*)]
		\item \label{it:formulationsPseudo:uPi}
		There exist $u\in C^\infty(\IT_{d+1};\IR^n)$ and $\pi \in C^\infty(\IT_{d+1};\IR^\ell)$ so that \eqref{eq:system11} holds;
		\item \label{it:formulationsPseudo:epsSig}
		\begin{equation*}
			\begin{cases}
				\dell_t \epsilon - \Qd^\ast\Pi_\Sd\Qd \sigma=0
				&\text{in }\IT_{d+1},
				\\  \widetilde{\Pd}^\ast \epsilon=0
				&\text{in }\IT_{d+1},
				\\ \int_{\IT_{d+1}}\epsilon\dx[(t,x)]=0.
			\end{cases}
		\end{equation*}
	\end{enumerate}
	In particular, assuming \ref{it:formulationsPseudo:epsSig}, we can choose $u$ and $\pi$ as
	\begin{equation}\label{eq:recoverUPi:1}
		u\coloneqq (\Qd^\ast)^{-1}\epsilon
		\quad
		\text{and}
		\quad
		\pi\coloneqq (\Sd^{\ast})^{-1}\Qd\sigma.
	\end{equation}
\end{lemma}

\begin{proof}
	The proof uses algebraic properties of the Moore--Penrose inverse \eqref{eq:moorePenrose:multiplierDef} as stated for example in \cite{Penrose_1955}.
	\begin{enumerate}[leftmargin=0pt,itemindent=2cm]
		\item [\ref{it:formulationsPseudo:uPi}$\implies$\ref{it:formulationsPseudo:epsSig}.]
		We apply the operator $\Qd^\ast \Pi_\Sd$ to $\dell_t u-\Qd\sigma+ \Sd^\ast \pi=0$. We then obtain $\dell_t\epsilon-\Qd^\ast\Pi_\Sd\Qd\sigma=0$.
		As $\Sd u=0$, we get $\epsilon=\Qd^\ast u= \Qd^\ast \Pi_\Sd u = \Qd^\ast \Rd\Rd^{-1} u$, so that $\widetilde{\Pd}^\ast\epsilon=0$.
		Finally,
		\begin{equation*}
			\int_{\IT_{d+1}}\epsilon\dx[(t,x)] = \int_{\IT_{d+1}}\Qd^\ast u\dx[(t,x)]=0.
		\end{equation*}
		\item [\ref{it:formulationsPseudo:epsSig}$\implies$\ref{it:formulationsPseudo:uPi}.]
		Note that $\dell_t\epsilon-\Qd^\ast\Pi_\Sd\Qd\sigma=0$ and the zero-mean condition on $\epsilon$ imply $\int_{\IT_d}\epsilon(t)\dx=0$ for all $t$.
		We define $u\coloneqq (\Qd^\ast)^{-1}\epsilon$
		and note that
		\begin{equation*}
			\Qd^\ast u 
			=   \Qd^\ast (\Qd^{\ast})^{-1} \epsilon
			=   \epsilon,
		\end{equation*}
		since $\epsilon = \Qd^\ast\Rd(\Qd^\ast\Rd)^{-1}\epsilon$.
		Next, on account of \eqref{eq:formulations:compatibility}, $ \Sd (\Qd^{\ast})^{-1}\Qd^\ast =  \Sd $, so that
		\begin{equation*}
			\Sd u 
			= \Sd (\Qd^{\ast})^{-1}\epsilon
			= \Sd (\Qd^{\ast})^{-1}\Qd^\ast\Rd(\Qd^\ast\Rd)^{-1}\epsilon
			= \Sd \Rd(\Qd^\ast\Rd)^{-1}\epsilon
			= 0.
		\end{equation*}
		Finally, we recover $\pi$ as $\pi\coloneqq (\Sd^{\ast})^{-1}\Qd\sigma$ and verify
		\begin{equation*}
			\dell_t u -\Qd\sigma  +\Sd^\ast \pi 
			=   (\Qd^\ast)^{-1}(\dell_t\epsilon -\Qd^\ast \Qd\sigma +\Qd^\ast\Sd^\ast\pi)
			=   (\Qd^\ast)^{-1}(\dell_t\epsilon -\Qd^\ast \Pi_\Sd\Qd\sigma )
			=   0.
		\end{equation*}
	\end{enumerate}
\end{proof}

This motivates the following adaptation of Definition~\ref{def:A}:
\begin{definition}\label{def:tildeA}
	Let $(\epsilon,\sigma) \in C^{\infty}(\IR^{d+1};\C^m \times \C^m)$. 
	We define the pseudo-differential operator 
	\begin{equation*}
		\widetilde{\A} \colon C^{\infty}(\IR^{d+1};\C^m \times \C^m) \longto (\IR^{d+1};\C^m\times \C^\ell)
	\end{equation*}
	as follows:
	\begin{equation} \label{eq:defTildeA}
		\widetilde{\A}(\epsilon,\sigma) 
		=   \begin{pmatrix}
			\partial_t \epsilon - \Qd^{\ast}\Pi_\Sd\Qd \sigma
			\\  \widetilde{\Pd}^\ast \epsilon
		\end{pmatrix}.
	\end{equation}
\end{definition}

\begin{example}[Reformulation of (inertialess) Navier--Stokes equations]\label{ex:formulations:NSExample}
	For later applications we elaborate on the construction above in the setting of the Navier--Stokes equation without the non-linear term $(u\cdot \nabla)u$:
	\begin{equation}\label{eq:formulations:NSExample}
		\begin{cases}
			\dell_t u  = \div \sigma - \nabla \pi
			&	\text{in }(0,T)\times\IT_d,
			\\  \div u =0
			&	\text{in }(0,T)\times\IT_d,
			\\  \tfrac{1}2(\nabla u+ (\nabla u)^T) = \epsilon
			&	\text{in }(0,T)\times\IT_d.
		\end{cases} 
	\end{equation}
	Let us briefly remark that due to the absence of $(u \cdot \nabla) u$ this system bears little physical relevance.\\
	We call ${u\colon(0,T)\times\IT_d\to \IR^d}$, $d\ge 2$, the velocity field of the fluid, $\sigma\colon(0,T)\times\IT_d\to \IR^{d\times d}_{\sym,0}$ the deviatoric stress tensor and $\pi\colon(0,T)\times\IT_d\to \IR$ the pressure.
	With $\IR^m \cong \IR^{d\times d}_{\sym,0}$, the space of symmetric trace-free matrices, we identify
	\begin{equation*}
		\Qd^\ast u 
		=   \frac{\nabla u + (\nabla u)^T}{2}
		-   \frac{\div u}{d} E_d
		\quad\text{and}\quad  \Sd u = \div u
	\end{equation*}
	(We denote with $E_d$ the unit matrix in $d$ dimensions).
	Observe that 
	\begin{equation*}
		\Image\Qd[\xi]=(\ker \Qd^\ast[\xi])^\perp =\IR^d
	\end{equation*}
	so that \eqref{eq:formulations:compatibility} is satisfied.
	Comparing this equation to the abstract framework developed before, we note:
	\begin{enumerate}[label=(\roman*)]
		\item 
		A potential for $\Sd =\div (=\nabla^{\ast})$ is given by  $\Rd = \curl^\ast$, where $\curl$
		is the usual higher-dimensional curl operator, $\curl u = (\dell_iu_j-\dell_ju_i)_{1\le i,j\le d}$;
		\item 
		The projection onto $\ker S$ is performed via the Leray Projector (e.g. \cite[Definition 2.8.]{Robinson_Rodrigo_Sadowski_2016}),
		\begin{equation*}
			\Pi_\Sd = (\mathrm{Id} -\Delta^{-1}\nabla\div)\colon C^\infty(\IT_d;\IR^d)\longto C^\infty(\IT_d;\IR^d)\cap \ker \div;
		\end{equation*}
		\item 
		A valid annihilator $\widetilde{\Pd}^\ast$ for the composition $\Qd^\ast \circ \Rd$ is
		\begin{equation*}
			\widetilde{\Pd}^\ast\coloneqq\Delta - (\nabla+\nabla^T)\div.
		\end{equation*}
		Via a direct computation one may show that
		\begin{equation*}
			\Image \Qd^\ast\Rd[\xi] 
			=   \left\{  
			\hat \sigma 
			\in \IR^{d\times d}_{\mathrm{sym},0}
			:   \abs{\xi}^2\, \hat \sigma 
			=  \hat \sigma (\xi\otimes \xi) 
			+ (\xi\otimes \xi)\hat \sigma
			\right\}.
		\end{equation*}
		The operator  $\widetilde{\Pd}^\ast$ above is also spanning.
	\end{enumerate}
	We conclude that $u,\pi,\sigma,\epsilon$ satisfy \eqref{eq:formulations:NSExample} if and only if
	\begin{equation*}
		\begin{cases}
			\dell_t \epsilon 
			- \Delta^{-1}\tfrac{\nabla + \nabla^T}{2}(\Delta-\nabla \div)\div\sigma
			=0
			&\text{in }(0,T)\times\IT_d,
			\\  \Delta\epsilon - (\nabla+\nabla^T)\div\epsilon=0
			&\text{in }(0,T)\times\IT_d,
			\\ \int_{\IT_{d+1}}\epsilon\dx[(t,x)] =0.
		\end{cases}
	\end{equation*}
	With \eqref{eq:recoverUPi:1}, the velocity $u$ and the pressure $\pi$ can be recovered as 
	\begin{equation*}
		u\coloneqq \Delta^{-1}\div \epsilon,
		\quad \pi \coloneqq \Delta^{-1}\div^2 \sigma.
	\end{equation*}
\end{example}


\subsection{$\widetilde \A$-quasiconvexity and weak lower semicontinuity for pseudo-differential operators}\label{sec:5_pseudoDiffOps:2_wlsc}
Section~\ref{sec:2_Spaces:2_prabolicDiffOps} revolved around the  differential operator $\dell_t \epsilon-\Qd^\ast\Qd\sigma$.
In this section we adjust the previous results to the setting of pseudo-differential operators as in Section~\ref{sec:5_pseudoDiffOps:1_reformulation}.
In particular, we stick to the compatibility condition \eqref{eq:formulations:compatibility}
\begin{equation*}
	\Image \Sd^\ast[\xi_x]\subset\Image \Qd [\xi_x],
	\quad \xi_x\in\IR^{d}\setminus\{0\}.
\end{equation*}
Definition~\ref{def:tildeA} then entails a corresponding notion of $\widetilde{\A}$-quasiconvexity:
\begin{definition} \label{def:tildeQuasiconvexity}
	We call a continuous function $f \colon \IR^m \times \IR^m \to [0,\infty)$ $\widetilde{\A}$-quasiconvex, if for all $(\hat{\epsilon},\hat{\sigma}) \in \IR^m \times \IR^m$ and all $(\epsilon,\sigma) \in C^{\infty}_{\#}(\IT_{d+1};\IR^m \times \IR^{m})$ with $\widetilde{\A}(\epsilon,\sigma)=0$ we have 
	\begin{equation} \label{tildeJensen}
		f(\hat{\epsilon},\hat{\sigma}) 
		\leq \int_{\IT_{d+1}} 
		f(\hat{\epsilon}+\epsilon(t,x),\hat{\sigma}+\sigma(t,x)) 
		\dx[(t,x)].
	\end{equation}
\end{definition}

All the results from Section~\ref{sec:3_propA} can be easily adjusted to the pseudo-differential setting and we will only briefly touch upon the changes necessary:
\begin{enumerate}[itemindent=1cm]
	\item [\textit{Lemma~\ref{lem:studyKernel}}] 
	$\widetilde {\A}$ admits the same scaling as $\A$, i.e. 
	$\widetilde \A(\epsilon,\sigma) =0$ if and only if 
	\begin{equation*}
		\widetilde \A(\epsilon(\lambda^{2k}t, \lambda x),\sigma(\lambda^{2k}t, \lambda x)) =0, \quad \lambda\in\IN_{\ge1}. 
	\end{equation*}
	We may decompose the cone $\Lambda_{\widetilde \A}$ in a similar fashion as for $\Lambda_{\A}$ by observing
	\begin{equation*}
		\ker \Qd^\ast\Pi_\Sd \Qd[\xi_x]
		= \ker \Rd^\ast\Qd[\xi_x] 
		= \Image \widetilde \Pd[\xi_x],
		\quad \xi_x\in \IR^{d}\setminus\{0\}.
	\end{equation*}
	From this it also follows that $\widetilde\A$ satisfies the constant-rank condition. 
	\item [\textit{Lemma~\ref{lem:AcorrectionOperator}}]
	The non-linear projection operator $\widetilde \Pi_{\widetilde{\A}}$ can be defined by replacing $\A$ with $\widetilde{\A}$.
	The only adjustment necessary in the proof is to replace $\Qd^\ast\Qd$ with $\Qd^\ast \Pi_\Sd\Qd$.
	Condition \eqref{eq:formulations:compatibility} enters in the algebra necessary to compute the inverse $(\Qd^\ast\Pi_\Sd\Qd)^{-1}$.
	\item [\textit{Lemma~\ref{lem:DiffOpVsEquiInt}}]
	It suffices to replace $\Qd^\ast \Qd$ with $\Qd^\ast\Pi_\Sd\Qd$ in the proof.
	\item [\textit{Lemma~\ref{lemma:base}}]
	If $f$ is $\widetilde \A$-quasiconvex it is also $\widetilde{\Pd}^\ast$-quasiconvex in the first argument and convex in the second.
	The local Lipschitz estimate continues to hold.
	\item [\textit{Lemma~\ref{lem:polyConeCvx}}]
	The chain of implications \eqref{implicationchain} transfers to $\widetilde \A$. 
	Each occurrence of $\Qd^\ast\Qd$ in the proof can be replaced with $\Qd^\ast\Pi_\Sd\Qd$. 
\end{enumerate}

For the corresponding analogue of the weak lower semicontinuity result, Theorem~\ref{thm:anisotropicSuff}, we require some additional statements about pseudo-differential operators.

First, it is possible to define pseudo-differential operators on domains by assuming $\Omega\subset (0,1)^{d}\subset\IT_d$ so that we may apply the previous Fourier theory. 
The action of a pseudo-differential operator on a function $\sigma\in \leb q(\Omega)$ is then given by duality,
\begin{equation*}
	\pair{\widetilde \IA\sigma,\psi}{}\coloneqq (-1)^k\,\int_\Omega\sigma\cdot \widetilde\IA^\ast\psi\dx,\quad \psi\in C_c^\infty(\Omega),
\end{equation*}
since functions in $C_c^\infty(\Omega)$ are periodic with respect to $\IT_d$.

Second, to substitute for the lack of a product rule for pseudo-differential operators, we mention the following fact (e.g. \cite{ruzhansky_quantization_2010, ruzhansky_Lp_2015}):
\begin{lemma}\label{lem:commutatorEstimate}
	Let $1<r<\infty$, $\alpha, k\in \IR$, $\eta\in C^\infty(\IT_{d})$ and $\widetilde \IA\colon C^\infty(\IT_{d};\IR^M)\to C^\infty(\IT_{d};\IR^M)$ be a $k$-homogeneous constant-rank pseudo-differential operator.
	Then the commutator
	\begin{equation*}
		M\colon\sigma\longmapsto \eta\cdot \widetilde{\IA}\sigma -\widetilde \IA (\eta\cdot \sigma)
	\end{equation*}
	extends to a bounded linear operator $M\colon \sobW{\alpha}r(\IT_d)\to \sobW{\alpha-k+1}{r}(\IT_d)$.
\end{lemma}

\begin{corollary}\label{thm:pseudoAnisotropicSuff}
	Let $p,q,r\in (1,\infty)$ so that $1<r \leq \min\{p,q\}$. 
	Assume that
	\begin{equation*}
		\begin{cases}
			   \hfill w_j\wconv w  
			&   \text{in } 
			(\leb p\times \leb q ) (\mathcal T\times \Omega;\IR^m\times\IR^m ),
			\\   \hfill \widetilde\A w_j\conv \widetilde \A w  
			&   \text{in } 
			(\WSob^{-1,-2k}_{r,r}
			\times \sobW{0,-k'}{r})(\mathcal T\times \Omega;\IR^m\times\IR^n),
		\end{cases}
	\end{equation*}
	and that $f$ is $\widetilde \A$-quasiconvex and permits the growth condition
	\begin{equation*}
		0\le 
        f(\hat \epsilon,\hat \sigma) 
		\leq C_1\,(1+ \vert \hat \epsilon \vert^p + \vert \hat \sigma \vert^q)
		,\quad (\hat \epsilon,\hat \sigma) \in \IR^m\times\IR^m.
	\end{equation*}
	Then we have
	\begin{equation*}
		\int\limits_{\mathcal T\times \Omega} 
		f(w)
		\dx[(t,x)]
		\le 
		\liminf_{j\to\infty}
		\int\limits_{\mathcal T\times \Omega} 
		f(w_j)
		\dx[(t,x)].
	\end{equation*}
\end{corollary}

\begin{proof}[Proof (Sketch)]
	The proof follows the same structure as for Theorem~\ref{thm:anisotropicSuff}.
	In particular, \ref{pf:anisotropicSuff:1_equiInt} and \ref{pf:anisotropicSuff:2_transferTrous} remain mostly unchanged.
	\ref{pf:anisotropicSuff:3_convDiffEq} requires more detail, since the product rule does not apply to general pseudo-differential operators. 
	However, we see that with Lemma~\ref{lem:commutatorEstimate}, for the function $\eta_{\varepsilon} w_j$ we can still show strong convergence of $\A (\eta_{\varepsilon_j} w_j)$ in a suitable $L_r$-based space:
		Recall that we aim to evaluate the first component 
		\[
		\A_1 ( \eta_{\varepsilon}\, (\epsilon_j,\sigma_j)) = \partial_t(\eta_{\varepsilon} \epsilon_j) + \left(\Qd^{\ast} \Pi_{\mathcal{S}} \Qd \right) (\eta_{\varepsilon} \sigma_j),
		\]
        similar to \eqref{first}.
		The first summand can be routinely decomposed into
		\[
		\partial_t \eta_{\varepsilon} \epsilon_j + \eta_{\varepsilon} \epsilon_j.
		\]
		For the second term we find,
		\[
		\left(\Qd^{\ast} \Pi_{\mathcal{S}} \Qd \right) (\eta_{\varepsilon} \sigma_j) = \eta_{\varepsilon} \left(\Qd^{\ast} \Pi_{\mathcal{S}} \Qd \right) \sigma_j + M_j
		\]
		where $M_j$ is the commutator of Lemma \ref{lem:commutatorEstimate}.
        The same reference implies that $M_j$ is bounded in $\WSob^{0,-2k+1}_r(\IT_d)$. 
        We can therefore proceed with the argument parallel to Theorem \ref{thm:anisotropicSuff}.
	
	As the argument of Lemma~\ref{lem:DiffOpVsEquiInt} is solely based on Fourier multipliers, it is easily adapted and \ref{pf:anisotropicSuff:4_tructation} follows by the same argument as before.
\end{proof}

\begin{remark}
	One might also imagine to show lower semicontinuity more in the style of \cite{Morrey,AF84}, i.e. arguing with $u$ instead of $\epsilon$. 
    Due to the additional constraint $\Sd u=0$ however, it is still necessary to argue with some Fourier analysis arguments when considering the potential.
\end{remark}
\section{Applications to Navier--Stokes} \label{sec:6_applNS_largeP}
We now come to the application of the previous abstract setup to the non-Newtonian Navier--Stokes equation. 
To avoid an in-depth discussion on boundary values, we assume that either $\Omega= \IT_d$ is the $d$-dimensional torus (corresponding to periodic boundary values) or that $\Omega \subset \IR^d$ is open with Lipschitz boundary and $u$ has vanishing (Dirichlet) boundary data. 
We model a fluid flow by the following quantities:
\begin{itemize}
	\item $u \colon (0,T) \times \Omega \to \IR^d$ is the fluid's velocity;
	\item $\epsilon \colon (0,T) \times \Omega \to \IR^{d \times d}_{\sym,0}$ denotes the rate-of-strain, where $\IR^{d \times d}_{\sym,0}$ is the space of symmetric, trace-free matrices and $\epsilon = \epsilon(u) \coloneqq \tfrac{1}{2} (\nabla u + (\nabla u)^T)$;
	\item $\sigma \colon (0,T) \times \Omega \to \IR^{d \times d}_{\sym,0}$ denotes the (deviatoric) stress;
	\item $\pi \colon (0,T) \times \Omega \to \R$ denotes the pressure.
\end{itemize}
We assume throughout that these variables obey the following differential equation
\begin{equation} \label{NavSto}
	\begin{cases}
		\partial_t u + \divergence (u \otimes u) = \divergence \sigma -\nabla \pi & \text{in } (0,T) \times \Omega\\
		\divergence u=0 & \text{in } (0,T) \times \Omega
	\end{cases}
\end{equation}
together with an initial condition 
\begin{equation} \label{initial}
	u (0,\cdot) = u_0
\end{equation}
that is required to hold in a suitable sense (cf. Remark \ref{rem:applNS:wellPosed}). We also assume aforementioned boundary conditions, but omit writing them in detail throughout.

Observe that the setting \eqref{NavSto}, excluding the non-linearity $\divergence (u \otimes u)$, can be put into the framework of the previous section, cf. Example \ref{ex:formulations:NSExample}.

As discussed in the introduction, instead of considering a constitutive law $\sigma=\sigma(\epsilon)$, we assume that we are given a function $f \colon \IR^{d \times d}_{\sym,0} \times \IR^{d \times d}_{\sym,0} \to [0,\infty)$ that measures how much the pair $(\epsilon,\sigma)$ deviates from the observed material properties.
In the following, we assume that 
\begin{enumerate} [label=(A\arabic*)]
	\addtocounter{enumi}{-1}
	\item \label{it:assf:contNN} $f \colon \IR^{d \times d}_{\sym,0} \times \IR^{d \times d}_{\sym,0}\to \IR$ is continuous and non-negative;
	\item \label{it:assf:growth} we have the following growth condition
	\begin{equation*} \label{eq:assf:growth}
		f(\hat\epsilon,\hat\sigma) \leq C_1\, ( 1 + \vert \hat\epsilon \vert^p + \vert \hat\sigma \vert^q);
	\end{equation*}
	\item \label{it:assf:coercive} we have the following coercivity condition
	\begin{equation*} \label{eq:assf:coercive}
		f(\hat\epsilon,\hat\sigma) \geq C_2\, (\vert \hat\epsilon \vert^p + \vert \hat\sigma \vert^q) -C_3 - C_4\, \hat\epsilon \cdot \hat\sigma;
	\end{equation*}
	\item \label{it:assf:qcvx}  $f$ is $\widetilde \A$-quasiconvex.
\end{enumerate}
The objective of this section is to prove the existence of minimisers of the functional
\begin{equation}\label{eq:NSfunctional}
	I(\epsilon,\sigma)
	\coloneqq
	\begin{cases}
		\int_{(0,T)\times \Omega} f(\epsilon,\sigma)\dx[(t,x)] 
		&   \text{if there are $u$ and $\pi$ so that \eqref{NavSto} holds,}
		\\  \infty 
		&   \text{else}.
	\end{cases}
\end{equation}

\subsection{Integrands arising from constitutive laws}\label{sec:6_applNS_largeP:1_integrandsConstitLaw}
Depending on the situation of study, several choices for $f$ are available, distinguished by whether the relation $\epsilon$ vs. $\sigma$ is known prior, or whether it has to be inferred from measurement data.

The case of constitutive laws permits a natural choice of $f$. 
Following models studied in e.g. \cite{B17}, consider the case where 
\begin{equation*}
	\sigma(\epsilon) = D W(\epsilon),
\end{equation*}
for a convex function $W\in C^1(\IR^{d \times d}_{\sym,0} \times \IR^{d \times d}_{\sym,0})$ with superlinear growth.
This class includes the so-called \emph{generalised-Newtonian} constitutive laws, where $\sigma(\epsilon) = \mu(\abs{\epsilon})\epsilon$ for a continuous function $\mu\colon[0,\infty)\to \IR$ for which $s\mapsto \mu(s)s$ is non-decreasing.
This leads to the non-Newtonian Navier--Stokes system 
\begin{equation}\label{eq:NNNS}
	\begin{cases}
		\dell_t u - \div (u\otimes u) = \div(\mu(\abs{\epsilon})\epsilon) -\nabla \pi 
		&\text{in } (0,T) \times \Omega,
		\\  \div u =0
		&\text{in } (0,T) \times \Omega,
		\\  \epsilon= \tfrac{1}{2}(\nabla u + (\nabla u)^T) 
		&\text{in } (0,T) \times \Omega,
	\end{cases}
\end{equation}
which has been studied for example in \cite{B17}.
A classical example for a constitutive law is the power-law model (also known as Ostwald--De Waele model, cf. \cite{Ostwald1929}):
\begin{equation*}
	\sigma = \mu_0 (\kappa + \abs{\epsilon})^{p-2}\epsilon,
	\quad   \kappa \ge0,~\mu_0>0, ~p>1. 
\end{equation*}
As we will see in Lemma~\ref{lem:verifyAssConstitLaws} below, we have $\sigma = D W(\epsilon)$ if and only if $f(\epsilon,\sigma)=0$, where
\begin{equation}\label{eq:fFromConstitLaw}
	f(\hat\epsilon,\hat\sigma)\coloneqq W(\hat\epsilon) + W^\ast(\hat\sigma)-\hat\epsilon\cdot \hat\sigma,
	\quad \hat\epsilon,\hat\sigma\in \IR^{d \times d}_{\sym,0}.
\end{equation}
Here, $W^\ast(\hat\sigma) \coloneqq \sup_{\hat\epsilon'}\hat\epsilon'\cdot \hat\sigma-W(\hat\epsilon')$ is the convex conjugate function to $W$.
The function $f$ can be viewed as a quantitative measure of how much a given strain-stress pair deviates from the constitutive law.
In particular, we observe that
\begin{equation*}
	(\epsilon,\sigma)\text{ is a weak solution of \eqref{eq:NNNS}}
	\iff
	I(\epsilon,\sigma)=0.
\end{equation*}
We will use this perspective in Section~\ref{sec:7_applNS_smallP} to reproduce existence results for Leray-Hopf solutions. 
However, even if the infimum of $I$ is \emph{positive}, minimisers of $I$ can be viewed as a relaxed solution concept to the non-Newtonian Navier-Stokes system \eqref{eq:NNNS}.
In this sense, our variational formulation is a generalisation of the classical PDE formulation and we obtain Result~\ref{it:mainResults:1} as claimed in the introduction.

To connect back to the abstract theory in Section~\ref{sec:4_wlsc} resp. \ref{sec:5_pseudoDiffOps}, we verify that integrands arising from constitutive laws satisfy the required regularity conditions \ref{it:assf:contNN}--\ref{it:assf:qcvx}.
\begin{lemma}\label{lem:verifyAssConstitLaws}
	Let $W\in C^1(\IR^{d \times d}_{\sym,0} \times \IR^{d \times d}_{\sym,0})$ be convex with superlinear growth and define $f$ according to \eqref{eq:fFromConstitLaw}.
	Then
	\begin{enumerate}[label=(\roman*)]
		\item \label{it:propWtoF:contNN}
		$f$ is continuous, non-negative;        
		\item\label{it:propWtoF:growth}
		if $W$ satisfies
		\begin{equation*}
			c_1\, (\abs{\hat\epsilon}^p + 1) 
			\ge W(\hat\epsilon) 
			\ge c_2\abs{\hat\epsilon}^p - c_3,
		\end{equation*}
		then $f$ satisfies the growth and coercivity condition
		\begin{equation*}
			C_1\, (\abs{\hat\epsilon}^p+\abs{\hat\sigma}^q + 1 )
			\ge f(\hat\epsilon,\hat\sigma) 
			\ge C_2\, (\abs{\hat\epsilon}^p+\abs{\hat\sigma}^q) - C_3 - C_4\, \hat\epsilon\cdot\hat\sigma;
		\end{equation*}
		\item\label{it:propWtoF:gradIffVanish}
		for $\hat\epsilon,\hat\sigma \in \IR^{d \times d}_{\sym,0}$ we have
		$\hat\sigma=D W(\hat\epsilon)\iff f(\hat\epsilon,\hat\sigma)=0$;
		\item\label{it:propWtoF:Aqcvx}
		$f$ is $\widetilde\A$-quasiconvex.
	\end{enumerate}
\end{lemma}
\begin{proof}
	\begin{enumerate}[leftmargin=0pt, itemindent=2em ]
		\item [\ref{it:propWtoF:contNN}]
		Non-negativity and continuity follow from the definition of $W^\ast$.
		\item[\ref{it:propWtoF:growth}]
		It is immediate from the definition of $W^\ast$ that it satisfies analogous growth bounds with exponent $q$ instead of $p$.
		From here, the growth and coercivity of $f$ follow by estimating $W$ and $W^\ast$.
		\item [\ref{it:propWtoF:gradIffVanish}]
		This is a classical fact for convex functions, see e.g. \cite[Proposition 16.16]{bauschke2017}.
		\item[\ref{it:propWtoF:Aqcvx}]
		The sum of two convex functions $W$ and $W^\ast$ and the $\widetilde\A$-quasiaffine function $(\hat\epsilon,\hat\sigma)\mapsto -\hat\epsilon\cdot \hat\sigma$, $f$ is $\widetilde\A$-polyconvex, as discussed in Lemma~\ref{lemma:base}.
	\end{enumerate}
\end{proof}


\subsection{Data-driven formulations and relaxation}\label{sec:6_applNS_largeP:2_integrandsDataDriven}
In applications, the specific shape of a constitutive law must be derived by expert modelling and parameters have to be fitted against measurement data.
Only afterwards can we start to solve the PDE.
Here, we follow the data-driven paradigm, explored, for example, in \cite{CMO,CMO2} in the context of elasticity:
We propose to use the data in the definition of our integrand directly, to bypass modelling errors stemming from simplified models.

Assume that we have sampled the viscous properties of a fluid in a dataset of point-wise strain-stress pairs,
\begin{equation*}
	\D
	=\{
	(\hat\epsilon_i,\hat\sigma_i)
	:   i\in I
	\}\subset \IR_{\sym,0}^{d\times d}\times \IR_{\sym,0}^{d\times d}.
\end{equation*}
We then define the integrand
\begin{equation}\label{eq:dataDrivenf}
	f(\hat\epsilon,\hat\sigma)
	\coloneqq \inf_{i\in I} \abs{\hat\epsilon-\hat\epsilon_i}^p+\abs{\hat\sigma-\hat\sigma_i}^q,
	\quad (\hat\epsilon,\hat\sigma)\in\IR_{\sym,0}^{d\times d}\times \IR_{\sym,0}^{d\times d}
\end{equation}
for $1<p,q<\infty$, $1=\tfrac{1}{p}+\tfrac{1}{q}$, to measure how much a given strain-stress pair deviates from the measurement data.

If the underlying data is taken to be a constitutive law, $ \D =\{(\hat\epsilon,\diffD W(\hat\epsilon)):\hat\epsilon\in\IR_{\sym,0}^{d\times d}\}$, 
the data-driven distance function \eqref{eq:dataDrivenf} provides an alternative to the formula in \eqref{eq:fFromConstitLaw}.
The two notions \emph{do not} coincide, but they share the same set of zeros and exhibit comparable growth conditions.

To verify the applicability of our abstract theory, observe the following properties:
\begin{enumerate}[label=(\roman*)]
	\item $f$ as in \eqref{eq:dataDrivenf} is continuous and non-negative;
	\item It satisfies $(p,q)$-growth as in \ref{it:assf:growth}, since 
	\begin{equation*}
		f(\hat\epsilon,\hat\sigma) 
		\le \abs{\hat\epsilon-\hat\epsilon_0}^p 
		+\abs{\hat\sigma-\hat\sigma_0}^q \le C_1\, (1 + \abs{\hat\epsilon}^p + \abs{\hat\sigma}^q)
	\end{equation*}
	for some $(\hat\epsilon_0,\hat\sigma_0)\in  \D $;
	\item The coercivity condition \ref{it:assf:coercive} is equivalent to 
	\begin{equation*}
		\D
		\subset 
		\{(\hat\epsilon,\hat\sigma)
		:   C_A\, \hat\epsilon\cdot \hat\sigma + C_B > \abs{\hat\epsilon}^p+\abs{\hat\sigma}^q   
		\}
	\end{equation*}
	for some constants $C_A,C_B>0$.
	(cf. \cite[Lemma 5.6]{lienstromberg.etal_2022}), which tells us that the data should roughly follow a power-law model, $\sigma \approx c \abs{\epsilon}^{p-2}\epsilon$.
\end{enumerate}
Concerning \ref{it:assf:qcvx}, for arbitrary data we cannot ensure that the integrand $f$ in \eqref{eq:dataDrivenf} is $\widetilde\A$-quasiconvex in the sense of Definition~\ref{def:tildeQuasiconvexity}.
A suitable substitute is therefore given by the \emph{$\widetilde\A$-quasiconvex envelope} $\QTA f$, defined in \eqref{eq:qcvxEnvelope} (and adapted to the pseudo-differential case in the obvious way).
For the $\widetilde\A$-quasiconvex envelope, we obtain:
\begin{enumerate}[label=(\roman*')]
	\item 
	With Lemma~\ref{lemma:base}, $\QTA  f$ is continuous and non-negative;
	\item 
	Evidently $\QTA f  \le f$, so that
	\begin{equation*}
		\QTA  f(\hat\epsilon,\hat\sigma) 
		\le C_1\,(1 + \abs{\hat\epsilon}^p + \abs{\hat\sigma}^q);
	\end{equation*}
	\item
	Since the lower bound in \ref{it:assf:coercive} defines itself a $\widetilde\A$-quasiconvex function (cf. Lemma~\ref{lemma:base}), we conclude
	\begin{equation*}
		\QTA  f(\hat\epsilon,\hat\sigma) 
        \geq C_2 \, (\vert \hat\epsilon \vert^p + \vert \hat\sigma \vert^q) -C_3 - C_4\, \hat\epsilon \cdot \hat\sigma;
	\end{equation*}
	\item \label{it:relaxf:qcvx}
	$\QTA f $ is $\widetilde \A$-quasiconvex (due to Theorem \ref{thm:relaxation}; see also  \cite[Lemma 7.1.]{rindlerCalculusVariations2018}).
\end{enumerate}

The additional $\widetilde \A$-quasiconvexity of $\QTA f$ makes it the natural candidate to replace $f$ in \eqref{eq:NSfunctional}.
We define  
\begin{equation}\label{eq:relaxedNSfunctional}
	\widetilde{I}(\epsilon,\sigma)
	\coloneqq
	\begin{cases}
		\int_{(0,T)\times \Omega} \QTA {f}(\epsilon,\sigma)\dx[(t,x)] 
		&   \text{if there are $u$ and $\pi$ so that \eqref{NavSto} holds,}
		\\  \infty 
		&   \text{else}.
	\end{cases}
\end{equation}
In light of Theorem~\ref{thm:applNS:lsc} below, the functional $\widetilde I$ is sequentially lower semicontinuous.

Furthermore, as hinted to in Section \ref{sec:4_wlsc:3_relax}, one may show that $\widetilde I$ is the lower semicontinuous relaxation of $I$, i.e. 
\begin{equation*}
	\widetilde I (\epsilon,\sigma) 
	= \inf\Big\{
	\liminf_{j \to \infty} I(\epsilon_j,\sigma_j): 
	(\epsilon_j,\sigma_j)\wconv (\epsilon,\sigma)\txtin \leb p\times \leb q
	\Big\}.
\end{equation*}
We refer to \cite{Mskript,BFL} for details in the homogeneous case.

We can make the following observation:
The set 
\begin{equation*}
	\widetilde{  \D }\coloneqq \{(\hat\epsilon,\hat\sigma): \widetilde f(\hat\epsilon,\hat\sigma)=0\}
\end{equation*}
contains the original data set $ \D $ and any weak solution, i.e. $(\epsilon,\sigma)\in  (\leb p\times \leb q )$ such that $\widetilde{I}(\epsilon,\sigma)=0$, satisfies $(\epsilon,\sigma)(t,x)\in \widetilde{  \D }$ for almost every $(t,x)$.
This suggests the interpretation of $\widetilde{  \D }$ a \emph{convexification} of the data set that is necessary for the existence of weak solutions.
We call $\widetilde{  \D }$ the $\A$-$(p,q)$-quasiconvex envelope of $ \D $ and refer to \cite[Definition 6.2]{lienstromberg.etal_2022} for more information on this hull in the static case.


\subsection{Existence of minimisers for large exponents}\label{sec:6_applNS_largeP:3_existenceMinim}
We now come to the proofs for existence of minimisers of $I$.
To give the underlying equation \eqref{NavSto} meaning in the context of the minimisation problem \eqref{eq:NSfunctional}, we elaborate on the well posedness in terms of $(\epsilon,\sigma)$:
\begin{definition} \label{def:X}
	Fix $\tfrac{2d}{d+2}<p<\infty$ and $q=\tfrac{p}{p-1}$. 
	Define $r_0 \coloneqq \min  \{q,\tfrac{p}{2}\tfrac{d+2}{d} \}$.
	Consider an initial value $u_0\in \leb 2(\Omega)$ with $\div u_0=0$.
	Let $C_E\ge2$ be a suitable constant, specified below.
	We assume the following for $(\epsilon,\sigma)$:
	\begin{enumerate}[label=(X\arabic*)]
		\item\label{it:defX:pqInteg} 
		$\epsilon\in \leb p((0,T)\times \Omega;\IR^{d\times d}_{\mathrm {sym},0})\cap \leb \infty((0,T); \sobW{-1}2(\Omega;\IR^{d\times d}_{\mathrm {sym},0}))$ and\\ $\sigma\in  \leb q((0,T)\times \Omega;\IR^{d\times d}_{\mathrm {sym},0})$;
		\item \label{it:defX:NS}
		if $ u \in \WSob^{0,1}_p((0,T)\times \Omega;\IR^d)$ and $ \pi\in \leb {r_0} ((0,T)\times \Omega) $ denote the unique solutions (subject to periodic/Dirichlet boundary conditions) of 
		\begin{equation*}\label{eq:applNS:uPidef}
			\begin{cases}
				\tfrac{1}{2}(\nabla u +(\nabla u)^{T})  =\epsilon,
				\\	\int_{(0,T)\times\Omega}u\dx[(t,x)]=\int_{\Omega} u_0\dx
                
			\end{cases}
			\quad\text{and}\quad 
			\begin{cases}
				\Delta\pi = \div^2(\sigma - u\otimes u),
				\\	\int_{(0,T)\times\Omega} \pi\dx[(t,x)]=0
			\end{cases}
		\end{equation*}
		respectively, then
		\begin{equation}\label{eq:applNS:eqn}
			\begin{cases}
				\dell_tu + \div(u\otimes u)=\div \sigma -\nabla \pi 
				&\text{in }(0,T)\times\Omega,        
				\\  \div u=0
				&\text{in }(0,T)\times\Omega,
				\\  \epsilon= \tfrac{1}{2}(\nabla +\nabla^T)u
				&\text{in }(0,T)\times\Omega;
			\end{cases}
		\end{equation}
		\item \label{it:defX:kinEnergyBound}
		with $u$ as in \ref{it:defX:NS}, we have a kinetic-energy bound
		\begin{equation*}
			\norm{u}{\leb \infty(\leb2)}^2 \le C_E\,(\norm{u_0}{\leb 2}^2 + \norm{\epsilon}{\leb p}^p+ \norm{\sigma}{\leb q}^q+1);
		\end{equation*}
		\item\label{it:defX:bc} 
		we have the initial condition
		\begin{equation*}\label{eq:applNS:bc}
			u(0,\cdot) = u_0.
		\end{equation*}
	\end{enumerate}
	We write $(\epsilon,\sigma)\in X$ whenever \ref{it:defX:pqInteg}--\ref{it:defX:bc} hold.
\end{definition}
\begin{remark}\label{rem:applNS:wellPosed}
	Assume $(\epsilon,\sigma)$ satisfy \ref{it:defX:pqInteg}--\ref{it:defX:bc} and let $u,\pi$ be defined as in \ref{it:defX:NS}.
	We remark the following well-known facts for the non-Newtonian Navier--Stokes equations in our setting:
	\begin{enumerate}[label=(WP\arabic*),itemindent=*,leftmargin=1cm]
		\item \label{it:applNS:wellPosed:eqnInWr0}
		With the Korn--Poincaré estimate we obtain from \ref{it:defX:pqInteg} that $u$ lies in the classic energy space
		\begin{equation*}
			u\in \sobW{0,1}{p}( (0,T)\times\Omega  )\cap \leb\infty((0,T);\leb2(\Omega)).
		\end{equation*} 
		By interpolation, $u\in  \leb{p\frac{d+2}{d}}( (0,T)\times\Omega  )$.
		Thus, $u\otimes u\in \leb {\frac{p}{2}\frac{d+2}{d}}( (0,T)\times\Omega  ) $ and hence, $\pi \in \leb {r_0}( (0,T)\times\Omega  )$ where $r_0 = \min\{q,\tfrac{p}{2}\tfrac{d+2}{d}\}>1$.
		This also implies that the first equation in \eqref{eq:applNS:eqn} is well-posed in $\sobW{-1,0}{r_0}+\sobW{0,-1}{r_0}$.
		\item \label{it:applNS:wellPosed:bcWeak}
		By testing with $\phi\in C^\infty_c( \Omega  )$, we find for a.e. $t_0,t_1\in[0,T]$ that
		\begin{equation*}
			\int_{\Omega}(u(t_1)-u(t_0))\cdot\phi\dx 
			= -\int_{t_0}^{t_1} \int_{\Omega}(\sigma -  u\otimes u -E_d\pi)\cdot \nabla \phi\dx \dx[t].
		\end{equation*}
		This means $t\mapsto u(t)$ is continuous with resprect to the weak topology of $\leb 2(\Omega)$ and consequently the initial-value condition $u(0,\cdot)=u_0$ is well posed.
		\item \label{it:applNS:wellPosed:LinftyBound}
		For large exponents, $p\ge \tfrac{3d+2}{d+2}$, condition \ref{it:defX:kinEnergyBound} is redundant:
		We have
		\begin{equation*}
			(u\otimes u)\cdot \nabla u\in \leb 1((0,T)\times \Omega)
		\end{equation*}
		and thus we may test the equation \eqref{eq:applNS:eqn} with $u$ itself to derive the strong energy equality
		\begin{equation*}
			\tfrac 12 \norm{u(t_1)}{\leb 2}^2 = \tfrac 12 \norm{u(t_0)}{\leb 2}^2 - \int_{t_0}^{t_1} \int_{\Omega} \sigma\cdot\epsilon\dx\dx\dx[t],
			\quad\txtforae t_0,t_1\in [0,T].
		\end{equation*}
		This yields the kinetic-energy bound \ref{it:defX:kinEnergyBound} with $C_E=2$ by applying Young's inequality at $t_0=0, t_1=T$.
		\item \label{it:applNS:wellPosed:u_cptMap}
		The map $X\to \leb{2r}( (0,T)\times\Omega  ),\ (\epsilon,\sigma)\mapsto u$ is continuous from the weak $ (\leb p\times \leb q )$-topology to the strong $\leb {2r}$-topology for any $1<r<r_0$.
		Indeed, due to the previous bounds, the equation yields $u \in \sobW{1,-1}{r_0}( (0,T)\times\Omega  )$.
		A version of Aubin--Lions--Simons (e.g. \cite[Thm 7.4.1.]{Amann09}) implies that
		\begin{equation*}
			\leb\infty(\leb 2)\cap \sobW{0,1}{p}\cap \sobW{1,-1}{r_0}\cptmap \leb {2r}( (0,T)\times\Omega  ).
		\end{equation*}
		Here, we crucially need \ref{it:defX:kinEnergyBound} to conclude that $(\leb p\times\leb q)$-weakly convergent sequences in $X$ are uniformly bounded in the intersection above.
		The resulting continuity carries over to the non-linearity $u\otimes u$ so we conclude that the set $X$ is weakly closed in $ (\leb p\times \leb q ) ((0,T)\times\Omega)$.
		\item \label{it:applNS:wellPosed:epsSigForm}
		With the notation from Example \ref{ex:formulations:NSExample}, we see that $(\epsilon,\sigma)\in X$ implies
		\begin{equation*}
			\begin{cases}
				\dell_t \epsilon- \Qd^\ast\Pi_\Sd\Qd\sigma = -\Qd^\ast\Pi_\Sd\Qd (u\otimes u)
				&\text{in }(0,T)\times\Omega,
				\\  \widetilde \Pd^\ast \epsilon=0 
				&\text{in }(0,T)\times\Omega.
			\end{cases}
		\end{equation*}
		In light of Definition~\ref{def:tildeA}, we abbreviate this equation as
		\begin{equation*}
			\widetilde \A(\epsilon,\sigma) = -\Theta (u\otimes u).
		\end{equation*}
	\end{enumerate}
\end{remark}

\begin{definition}\label{def:applNS:defIonX}
	Let $p>\tfrac{2d}{d+2}$ and consider $u_0$  and $C_E$ as in Definition~\ref{def:X}.
	Define 
	\begin{equation*}
		I(\epsilon,\sigma)\coloneqq 
		\begin{cases}
			\int_{(0,T)\times\Omega} f(\epsilon,\sigma)\dx[(t,x)]
			&	\text{if }(\epsilon,\sigma)\in X,
			\\	\infty
			&	\text{else}.
		\end{cases}
	\end{equation*}
\end{definition}

We are now in a position to formulate the main results on lower semicontinuity of $I$ and existence of minimisers.

\begin{theorem}[Result \ref{it:mainResults:3}]\label{thm:applNS:lsc}
	If $p>\tfrac{2d}{d+2}$ and $f$ satisfies \ref{it:assf:contNN}, \ref{it:assf:growth} and \ref{it:assf:qcvx}, then $I$ is sequentially lower semicontinuous with respect to the weak topology on $ (\leb p\times \leb q ) ((0,T)\times \Omega )$.
\end{theorem}

\begin{proof}
	Let $(\epsilon_j,\sigma_j)\wconv(\epsilon,\sigma)$ in $ (\leb p\times\leb q ) ((0,T)\times\Omega )$.
	If $\liminf_{j\to\infty}I(\epsilon_j,\sigma_j)=\infty$ there is nothing to prove, so we may assume that $(\epsilon_j,\sigma_j)\in X$ for all $j\in\IN$.
	As $p>\tfrac{2d}{d+2}$, $X$ is weakly closed (cf. \ref{it:applNS:wellPosed:u_cptMap}) and we get $(\epsilon,\sigma)\in X$.
	Now, if we define $u_j,\pi_j,$ for $j\in \IN$ and $u,\pi$ according to \ref{it:defX:NS}, \ref{it:applNS:wellPosed:u_cptMap} implies that $u_j\conv u$ in $\leb{2r}$.
	Therefore, 
	\begin{equation*}
		\widetilde  \A(\epsilon_j,\sigma_j) 
		=   -\Theta(u_j\otimes u_j) 
		\conv -\Theta(u\otimes u)
		=   \widetilde \A(\epsilon,\sigma)
	\end{equation*}
	in $\WSob^{-1,-2}_{r,r}\times \sobW{0,-2}{r}$.
	As $f$ satisfies \ref{it:assf:growth} and \ref{it:assf:qcvx}, we can use the lower semicontinuity result from Corollary \ref{thm:pseudoAnisotropicSuff} to conclude
	\begin{equation*}
		I(\epsilon,\sigma)
		=   \int_{(0,T)\times \Omega} f(\epsilon,\sigma)\dx[(t,x)]
		\le \liminf_{j\to\infty}
		\int_{(0,T)\times \Omega} f(\epsilon_j,\sigma_j)\dx[(t,x)]
		= \liminf_{j\to \infty} I(\epsilon_j,\sigma_j).
	\end{equation*}
\end{proof}
While sequential weak lower-semicontinuity may be shown for any $p>\tfrac{2d}{d+2}$, the second main ingredient to the direct method, i.e. coercivity, is only valid for $p \geq \tfrac{3d+2}{d+2}$.
\begin{lemma}\label{lem:applNS:coercive}
	If $p\ge \tfrac{3d+2}{d+2}$ and $f$ satisfies \ref{it:assf:contNN} and \ref{it:assf:coercive}, then $I$ is coercive in the sense that for each $C\ge 0$, each set 
	\begin{equation*}
		\{(\epsilon,\sigma): I(\epsilon,\sigma)\le C\}
	\end{equation*}
	is relatively weakly compact in $ (\leb p\times \leb q ) ((0,T)\times \Omega )$.
\end{lemma}
\begin{proof}
	It suffices to produce a uniform bound on $\norm{\epsilon}{\leb p}^p+\norm{\sigma}{\leb q}^q$ whenever $I(\epsilon,\sigma)\le  C$ for some $ C\ge 0$.
	Note that this immediately implies that $(\epsilon,\sigma)\in X$.
	By rewriting the coercivity estimate \ref{it:assf:coercive}, we see
	\begin{equation*}
		C_1(\norm{\epsilon}{\leb p}^p+\norm{\sigma}{\leb q}^q)
		\le I(\epsilon,\sigma)
		+ C_2\, \int_{(0,T)\times \Omega} \epsilon\cdot \sigma\dx[(t,x)]
		+ C_3.
	\end{equation*}
	Since $\tfrac 12 (\nabla +\nabla^T)u = \epsilon$, we can integrate by parts to estimate the double integral with
	\begin{align*}
		\int_{(0,T)\times \Omega} \epsilon\cdot \sigma\dx[(t,x)]
		=&  -\int_0^T \pair{u,\div\sigma}{\sobW{1}{p}\times \sobW{-1}q}\dx[t]
		\\  =&  -\int_0^T \pair{u,\dell_t u + \div(u\otimes u - E_3\pi)}{\sobW{1}{p}\times \sobW{-1}q} \dx[t] 
		\\  =&  \tfrac{1}{2}
		\left(
		\norm{u_0}{\leb 2}^2
		-\norm{u(T,\cdot)}{\leb 2}^2
		\right)
		\\	\le& \tfrac{1}{2}\norm{u_0}{\leb 2}^2.
	\end{align*}
	This implies the uniform bound.
	Note that we crucially used $p\ge \tfrac{3d+2}{d+2}$ so that we have $r_0=q$ and $\dell_t u\in \sobW{0,-1}{q}$.
	This condition ensures that we can test the equation with $u$ itself and that the integral over $u\cdot \div(u\otimes u)$ vanishes.
\end{proof}

\begin{theorem}[Result \ref{it:mainResults:2}]\label{thm:existenceMinimisersLargeP}
	Let $p\ge \tfrac{3d+2}{d+2}$.
	Then $I$ admits a minimiser in $ (\leb p\times \leb q ) ((0,T)\times\Omega )$.
\end{theorem}
\begin{proof}
	We follow the structure of the direct method in the calculus of variation.
	Clearly, the functional is bounded below by zero, so that we may consider a sequence $\{w_j\}_{j\in\IN}\subset (\leb p\times \leb q )$ so that $I(w_j)\conv \inf I$.
	Lemma~\ref{lem:applNS:coercive} implies that $w_j$ admits a (not relabelled) weakly convergent subsequence $w_j\wconv w$ in $ (\leb p\times \leb q ) ((0,T)\times \Omega )$.
	According to Theorem~\ref{thm:applNS:lsc}, we find
	\begin{equation*}
		I(w)\le\liminf_{j\to\infty} I(w_j) = \inf I.
	\end{equation*}
	This proves the claim.
\end{proof}
\section{Subcritical exponents}\label{sec:7_applNS_smallP}

In the previous section, by proving weak lower-semicontinuity and coercivity, we derived the existence of minimisers for the range of exponents $p \ge \tfrac{3d+2}{d+2}$. 
Observe that, in the PDE setting, this corresponds to classical existence results for solutions of the non-Newtonian Navier--Stokes equations, cf. \cite{Lady1,Lady2,Lady3,Lionsbook}. 
The case $\tfrac{2d}{d+2} < p < \tfrac{3d+2}{d+2}$ is naturally more challenging. 
This section aims to discuss what is (and what is not) possible with our ansatz.

The difficulty lies (similarly to the PDE approach) in getting a correct, natural coercivity condition. 
In particular, in previous calculations, using the equation we rewrote
\[
- \int_0^T \int_{\IT_d} \epsilon \cdot \sigma \dx\dx[t] = \int_0^T \int_{\IT_d} ( u \otimes u) \cdot \nabla u \dx\dx[t].
\]
Due to the identity $( u \otimes u) \cdot \nabla u = \tfrac{1}{2}\divergence(\vert u \vert^2 u)$, the latter integrand is zero whenever it is well-defined, e.g. $\vert u \vert^2 u \in \WSob^{1,1}$; this is true in particular if $p \ge \tfrac{3d+2}{d+2}$.

The difficulty with dealing with this term is reflected in PDE-based approaches by the non-uniqueness of solutions in suitable Sobolev spaces with weaker integrability.
See e.g. \cite{BV,BMS,DLS3} and \cite{ABC}, where non-uniqueness in the critical space is shown under added forcing. 
On the other hand, showing that there are Leray-Hopf solutions is still possible in the PDE setting; see for instance \cite{Malek1,DRW,BDS}.

The plan of this section is to connect those results with our approach. In particular,
\begin{itemize}
	\item in Subsection \ref{sec:7_applNS_smallP:1_LerayHopf} we show that under strong assumptions on $f$ (so that it corresponds to certain constitutive laws), we recover known existence results for Leray--Hopf solutions via lower semicontinuity arguments, adding a variational flavour to the proofs of \cite{Malek1,DRW} etc.;
	\item in Subsection \ref{sec:7_applNS_smallP:2_discussionSmallP} we explain that this is not easily extendable to general $f$;
	\item in Subsection \ref{sec:7_applNS_smallP:3_bonusEnergyBound} we therefore briefly discuss a solution concept where an energy inequality is already incorporated.
\end{itemize}


\subsection{Existence of Leray--Hopf solutions} \label{sec:7_applNS_smallP:1_LerayHopf}
As advertised, we now consider a quite special setup for the function $f$. Let $W \colon \IR^{d \times d}_{\sym,0} \to [0,\infty)$ be a function that satisfies
\begin{enumerate} [label=(W\arabic*)]
	\item \label{it:W1} $W \in C^1(\IR^{d \times d}_{\sym,0},[0,\infty))$ is convex;
	\item \label{it:W2}$  C_1\,  (1+ \vert \epsilon \vert^p)\ge W(\epsilon) \ge C_2\, \vert \epsilon \vert^p - C_3$;
	\item \label{it:W3} $\vert DW(\epsilon) \vert \leq C_4\, (1 + \vert \epsilon \vert^{p-1})$.
\end{enumerate}
We remark that \ref{it:W3} follows from \ref{it:W1} and \ref{it:W2}. 
Consider the constitutive law $\sigma (\epsilon) = DW(\epsilon)$ that is represented by the set of zeros of
\begin{equation}\label{eq:lerayHopf:deff}
	f(\epsilon,\sigma) = W(\epsilon) + W^{\ast}(\sigma) - \epsilon \cdot \sigma,
\end{equation}
where $W^{\ast}$ is the convex conjugate of $W$.

We are now able to prove the following result that is well-known (cf. \cite{FMS,DRW,BDS}) using the variational methods of  Section~\ref{sec:4_wlsc:2_generalExponents}. Note that, in spirit, the proof is not so much different from the classical proofs: 
Also these use (more intricate) truncations to show convergence of the non-linearity $\epsilon \mapsto DW(\epsilon)$ of some approximating sequence. 
Fortunately, in our case, this work has been done in Section~\ref{sec:4_wlsc:2_generalExponents}. 
In particular, the main steps comprise of restoring equi-integrability and using the truncation of Lemma~\ref{lem:DiffOpVsEquiInt}. 
For simplicity, we only formulate it on the torus, but note that the method of proof works whenever the boundary conditions guarantee a sequence of solutions to regularised systems that are bounded in $\WSob^1_p$, which is the case for Dirichlet boundary conditions on Lipschitz domains.
\begin{theorem}[Result \ref{it:mainResults:4}]\label{thm:existenceLerayHopf}
	Let $p>\tfrac{2d}{d+2}$ and $W$ satisfy \ref{it:W1}--\ref{it:W3}. Then there exists a $\delta>0$, such that for any $u_0 \in L_2(\IT_d)$ with $\div u_0=0$ there exists a solution $u \in L_p((0,T);\WSob^1_p) \cap \WSob^1_{1+\delta}((0,T);(\WSob^{-1}_{1+\delta}))$ of the system 
	\begin{equation} \label{eq:NS}
		\begin{cases}
			\partial_t u + \divergence( u \otimes u) = \divergence (DW(\epsilon)) - \nabla \pi
		&	
			\text{in } (0,T)\times \IT_{d+1}
		, \\
			\divergence u =0
		&	
			\text{in } (0,T)\times \IT_{d+1}
		, \\
			\epsilon= \tfrac{1}{2}(\nabla u +(\nabla u)^T)
		&	
			\text{in } (0,T)\times \IT_{d+1}
		, \\
			u(0,\cdot) = u_0
		&
			\text{on }\IT_{d+1}.
		\end{cases}
	\end{equation}
	In addition, for almost every $t \in [0,T]$ $u$ satisfies the energy inequality
	\begin{equation} \label{eq:ei}
		\tfrac{1}{2} \norm{u(0,\cdot)}{\leb 2}^2 \geq  \tfrac{1}{2}\norm{u(t,\cdot)}{\leb 2}^2  + \int_0^t \int_{\IT_d} DW(\epsilon(s,x))\cdot \epsilon(s,x) \dd x \dd s.
	\end{equation}
\end{theorem}
\begin{remark}
	Note that \eqref{eq:ei} also implies condition \ref{it:defX:kinEnergyBound} with $C_E=2$, just like in \ref{it:applNS:wellPosed:LinftyBound}, hence there is no dependency on $C_E$ in this theorem.
\end{remark}

\begin{proof}[Proof of Theorem~\ref{thm:existenceLerayHopf}]
	Set $r= 1+\delta$ and choose $r\in(1,\min\{p,q,\tfrac{p}{2}\tfrac{d+2}{d},\tfrac{3d+2}{2d}\})$ so that $r'=\tfrac{r}{r-1}>\tfrac{3d+2}{d+2}$.
	We approximate $W$ with $W_\eta$, which satisfies stronger growth conditions:
	\begin{equation} \label{def:Wj}
		W_\eta(\hat\epsilon) \coloneqq W(\hat\epsilon) + \eta^{r'} \frac{\vert \hat\epsilon \vert^{r'}}{r'}
		,\quad \hat\epsilon\in \IR_{\sym,0}^{d\times d}.
	\end{equation}
	The conjugate function $W_\eta^\ast$ can be written as 
	\begin{equation*}
		W_\eta^\ast(\hat\sigma) = \inf_{\hat\sigma'} \left(W^\ast(\hat \sigma') + \eta^{-r} \frac{\abs{\hat\sigma-\hat \sigma'}^{r}}{r} \right),\quad \hat\sigma\in \IR_{\sym,0}^{d\times d}
	\end{equation*}
	(see \cite[Proposition 15.1.]{bauschke2017}).
	Let $f$ resp. $f_\eta$ be defined according to \eqref{eq:lerayHopf:deff} and denote with $I$ resp. $I_\eta$ the corresponding functionals in Definition~\ref{def:applNS:defIonX}.
	
	The classical existence theory for non-Newtonian fluids with large exponents (e.g. \cite{Lady1}) provides us with a solution $u_\eta\in \leb \infty(\leb 2)\cap \sobW{0,1}{r'}$ to the system
	\begin{equation} \label{approx:system}
		\begin{cases}
			\partial_t u + \divergence( u \otimes u) = \divergence (DW_\eta(\epsilon)) - \nabla \pi
			&	
			\text{in } (0,T)\times \IT_{d+1}
			, \\
			\divergence u =0
			&	
			\text{in } (0,T)\times \IT_{d+1}
			, \\
			\epsilon= \tfrac{1}{2}(\nabla u +(\nabla u)^T)
			&	
			\text{in } (0,T)\times \IT_{d+1}
			, \\
			u(0,\cdot) = u_0
				&	
			\text{on } \IT_{d+1}
		\end{cases}
	\end{equation} 
	for each $\eta>0$. 
	Due to the improved integrability, we may test with $u_\eta$ and derive the energy equality
	\begin{equation}\label{eq:exLerayHopf:energyEq}
		\tfrac{1}{2}\norm{u_0}{\leb 2}^2 
		= \tfrac{1}{2}\norm{u_\eta(t)}{\leb 2}^2 + \int_0^t\int_{\IT_d} W_\eta(\epsilon_\eta) + W_\eta^\ast(D W_\eta(\epsilon_\eta))\dx\ds,\quad   
	\end{equation}
	for a.e. $ t\in(0,T) $.
	The growth assumption on $W$ in \ref{it:W2} yields $W_\eta^\ast \ge -C$ uniformly in $\eta$, so that
	\begin{equation}\label{eq:exLerayHopf:uniformBound}
		\sup_\eta \norm{u_\eta}{\leb\infty(\leb 2)}^2 + \norm{\epsilon_\eta}{\leb p}^p + \norm{\eta\,\epsilon_\eta}{\leb {r'}}^{r'}\le C (1+\norm{u_0}{\leb 2}^2).
	\end{equation}
	We extract a subsequence so that $\epsilon_\eta\wconv \epsilon$ in $\leb p$, and $\eta\,\epsilon_\eta\wconv 0$ in $\leb {r'}$ as $\eta\to 0$.
	Furthermore, we set $\sigma_\eta\coloneqq  DW (\epsilon_\eta)$ and assume $\sigma_\eta\wconv \sigma$ in $\leb q$, which is possible due to \ref{it:W3}.
	Note that by Lemma  \ref{lem:verifyAssConstitLaws} \ref{it:propWtoF:gradIffVanish}, we have $I(\epsilon_\eta,\sigma_\eta)=0$ for each $\eta>0$.
	To apply Corollary \ref{thm:pseudoAnisotropicSuff}, we only have to verify
	\begin{equation*}
		\widetilde{\A} (\epsilon_\eta,\sigma_\eta) = \Theta(DW_\eta(\epsilon_\eta)-\sigma_\eta-u_\eta\otimes u_\eta)
		\conv \widetilde {\A}(\epsilon,\sigma)
		\quad\text{in }(\sobW{-1,0}{r}+\sobW{0,-2}r)\times \sobW{0,-2}r.
	\end{equation*}
	However, as argued in Remark \ref{rem:applNS:wellPosed} \ref{it:applNS:wellPosed:u_cptMap}, we have $u_\eta\conv u$ in $\leb {2r}$ and due to the uniform bound \eqref{eq:exLerayHopf:uniformBound}, 
	\begin{equation*}
		\abs{DW_\eta(\epsilon_\eta)-\sigma_\eta} = \eta\, \abs{\eta\, \epsilon_\eta}^{r'/r}\conv 0
		\quad \text{in }\leb r.
	\end{equation*}
	This also implies $(\epsilon,\sigma)\in X$ and we obtain
	\begin{equation*}
		0=\liminf_{\eta\to0}I(\epsilon_\eta,\sigma_\eta) \ge I(\epsilon,\sigma)\ge 0.
	\end{equation*}
	With Lemma~\ref{lem:verifyAssConstitLaws} \ref{it:propWtoF:gradIffVanish}, we conclude $\sigma = D W(\epsilon)$ almost everywhere.
	
	For the energy inequality, observe that $W_\eta\ge W$ and $W_\eta^\ast\ge W_{\eta_0}^\ast$, where $\eta<\eta_0$ ($W_\eta^{\ast}$ is monotonically increasing as $\eta \to 0$).
	For almost every $t$ so that $u_\eta(t)\conv u(t)$ in $\leb 2$, we can estimate \eqref{eq:exLerayHopf:energyEq} from below to get
	\begin{equation*}
		\tfrac{1}{2}\norm{u_0}{\leb 2}^2 
		\ge 
		\tfrac{1}{2}\norm{u_\eta(t)}{\leb 2}^2 
		+ \int_0^t\int_{\IT_d} W(\epsilon_\eta) + W_{\eta_0}^\ast(D W_\eta(\epsilon_\eta))\dx\ds
	\end{equation*}
	Since $W^\ast_{\eta_0}$ is convex and admits $r$-growth, and $DW_\eta(\epsilon_\eta)\wconv \sigma$ in $\leb r$, we can take the liminf in $\eta$ and send $\eta_0\to 0$ afterwards:
	\begin{equation*}
		\begin{split}
			\tfrac{1}{2}\norm{u_0}{\leb 2}^2 
			&\ge\liminf_{\eta\to0}
			\tfrac{1}{2}\norm{u_\eta(t)}{\leb 2}^2 
			+ \int_0^t\int_{\IT_d} W(\epsilon_\eta) + W_{\eta_0}^\ast(D W_\eta(\epsilon_\eta))\dx\ds
			\\	&\ge 
			\tfrac{1}{2}\norm{u(t)}{\leb 2}^2 
			+ \int_0^t\int_{\IT_d} W(\epsilon) + W_{\eta_0}^\ast(\sigma)\dx\ds
			\\	&\overset{\eta_0\to 0}{\conv} 
			\tfrac{1}{2}\norm{u(t)}{\leb 2}^2 
			+ \int_0^t\int_{\IT_d} W(\epsilon) + W^\ast(\sigma)\dx\ds
		\end{split}
	\end{equation*}
	With $I(\epsilon,\sigma)=0$, this yields the energy inequality \eqref{eq:ei}.
\end{proof}


\subsection{Remarks on the proof of Theorem~\ref{thm:existenceLerayHopf}} \label{sec:7_applNS_smallP:2_discussionSmallP}
Taking a closer look at the proof of Theorem~\ref{thm:existenceLerayHopf}, we approximate the problem by minimising a sequence of regularised functionals (involving $W_\eta$), for which we know that a minimiser exists. 
We then argue by weak lower semicontinuity that the limit must also be a minimiser.

At first glance, it appears that this proof can easily be generalised to a fully variational setting: 
We approximate the functional $I$ by functionals $I_\eta$ and then prove that $\{I_\eta\}_{\eta}$ $\Gamma$-converges to the functional $I$ with respect to weak convergence in suitable Sobolev spaces.

While the existence of such an $I_\eta$ is not completely ruled out, there are some complicating factors. 
First of all, observe that in the previous proof of Theorem~\ref{thm:existenceLerayHopf} we indeed show some kind of liminf-inequality; 
but \emph{we do not show} a limsup inequality: 
As we have the natural bound $I(u,\sigma) \geq 0$, finding a pair $(u,\sigma)$ that satisfies $I(u,\sigma)=0$ is already enough to infer that is a minimiser.

In particular, we \emph{do not} need to show that any $(u,\sigma)$ with $I(u,\sigma)$ can be approximated by regularised $(u_\eta,\sigma_\eta)$; 
i.e. if Leray--Hopf solutions are non-unique (which is, up to now, an open questions, also cf. \cite{JS}), due to the uniqueness of solutions to $I_\eta$, such a limsup inequality is ruled out.

Another argument \emph{against} such an approximation argument is as follows: 
The system involving $I_\eta$ is a regularised system, where $u_\eta$ enjoys some higher integrability properties. 
Consequently, $u_\eta$ satisfies at least some kind of energy inequality, for instance, in the setting of the previous subsection, a generalisation of \eqref{eq:exLerayHopf:energyEq},
\[
I_{\eta}(u_\eta,\sigma_\eta) + \tfrac{1}{2}\norm{u_0}{\leb 2}^2 
\geq \tfrac{1}{2}\norm{u_\eta(t)}{\leb 2}^2 + \int_0^t\int_{\IT_d} W_\eta(\epsilon_\eta) + W_\eta^\ast(\sigma_\eta)\dx\ds,\quad \text{for a.e. } t\in(0,T).
\]
Observe that the right-hand side of such an inequality is weakly lower semicontinuous. 
If the limiting $u$ does \emph{not} satisfy an energy inequality, it is therefore hardly possible that the limsup inequality 
\[
\limsup_{\eta \to 0} I_\eta(u_\eta,\sigma_\eta) = I(u,\sigma)
\]
is satisfied.


\subsection{Assuming an additional energy inequality}\label{sec:7_applNS_smallP:3_bonusEnergyBound}

The primary obstruction in constructing a minimiser in the regime $p\in(\tfrac{2d}{d+2}, \tfrac{3d+2}{d+2})$ is the potential failure of coercivity.
To combat this, we will consider flows with an additional energy bound. 
This energy bound might seem artificial, but is, to an extent, also physically reasonable.

If we rearrange the inequality due to \ref{it:assf:coercive}, we obtain
\begin{equation*}
	I(\epsilon,\sigma)+ C_3 \int_{(0,T)\times\IT_d}\epsilon\cdot\sigma\dx[(t,x)] + C_4\, T
	\ge C_2\, (\norm{\epsilon}{\leb p}^p+\norm{\sigma}{\leb q}^q),\quad (\epsilon,\sigma)\in X.
\end{equation*}
Natural bounds on the terms on the left-hand side motivate an additional restriction on $(\epsilon,\sigma)$ that ensures the existence of minimisers in the subcritical regime:
First, since we are interested in minimisers of $I$, it makes sense to restrict our attention only to functions that satisfy 
\begin{equation*}
	C_0\ge I(\epsilon,\sigma),
\end{equation*}
where $C_0$ is a suitable constant.
A natural choice would be
\begin{equation*}
	C_0\coloneqq I(\epsilon_{\mathrm{PL}},\sigma_{\mathrm{PL}}),
\end{equation*}
where $(\epsilon_{\mathrm{PL}},\sigma_{\mathrm{PL}})\in X$ is the solution to the power-law equation, $\sigma_\mathrm{PL} = \signabs{\epsilon_\mathrm{PL}}{p-2}$ , via Theorem~\ref{thm:existenceLerayHopf}. 

Second, inspired by the corresponding energy bound in \eqref{eq:ei}, it is physically plausible to require
\begin{equation*}
	\tfrac{1}{2}(\norm{u_0}{\leb 2}^2-\norm{u(t)}{\leb 2}^2)
	\ge\int_0^t \int_{\Omega} \epsilon\cdot\sigma \dx\dx[s]
	\quad \txtforae t\in[0,T].
\end{equation*}

Combining these two assumptions, we arrive at the following inequality:
\begin{equation}\tag*{(X5)}\label{it:defX:X5}
	C_0 + C_4 \, T + \frac{C_3}{2}\norm{u_0}{\leb 2}^2
	\ge \frac{C_3}{4}\norm{u}{\leb \infty(\leb 2)}^2 
	+ \frac{C_2}{2}\, \norm{\epsilon}{\leb p}^p
	+ \frac{C_2}{2}\, \norm{\sigma}{\leb q}^q.
\end{equation}

Note that \ref{it:defX:X5} implies \ref{it:defX:kinEnergyBound} for suitable $C_E$.

We stress that there is no guarantee that a flow satisfies \ref{it:defX:X5} and that this condition constitutes a proper restriction of the space on which we minimise. 
Nonetheless, the solution $(\epsilon_\mathrm{PL}, \sigma_\mathrm{PL})$ of the power-law equation shows that the space of functions satisfying \ref{it:defX:X5} is non-empty.
To work with this smaller set of admissible flows, we introduce the restricted functional 
\begin{equation}
	J(\epsilon,\sigma) \coloneqq  
	\begin{cases}
		\int_{(0,T)\times \IT_d} f(\epsilon,\sigma)\dx[(t,x)]
		&	\text{$(\epsilon,\sigma)$ satisfy \ref{it:defX:pqInteg}--\ref{it:defX:X5}},
		\\\infty &\text{else}.
	\end{cases}
\end{equation}

The existence of minimisers can now be easily established along the lines of Theorem~\ref{thm:existenceMinimisersLargeP}:
\begin{theorem} \label{thm:existenceSmallP}
	Assume $p>\tfrac{2d}{d+2}$ and $f$ satisfies \ref{it:assf:contNN}--\ref{it:assf:qcvx}. 
	Then the functional $J$ admits a minimiser in $\leb p\times\leb q$.
\end{theorem}
\begin{proof}[Proof (Sketch)]
	The proof follows the same structure as in Theorem~\ref{thm:existenceMinimisersLargeP}: 
	It suffices to show coercivity and lower semicontinuity, because then we can conclude that every minimising sequence admits a weak accumulation point, which has to be a minimiser.
	This time, we get coercivity as in Lemma~\ref{lem:applNS:coercive} immediately from \ref{it:defX:X5}.
	Since $I$ was already proven to be lower semicontinuous in the regime $p>\tfrac{2d}{d+2}$, we only need to make sure that \ref{it:defX:X5} is stable under weakly convergent sequences in $\leb p\times\leb q$. 
	This is immediate, since the right-hand side in \ref{it:defX:X5} is weakly lower semicontinuous as a sum of norms. 
\end{proof}


\bibliography{biblio.bib}
\bibliographystyle{abbrv}

\end{document}